\documentclass[a4paper,11pt]{article}
\usepackage[latin1]{inputenc}
\usepackage[hmargin={25mm,25mm},vmargin={25mm,30mm}]{geometry}
\usepackage{amsmath,amsfonts,amsthm,amssymb}
\usepackage{cancel}
\usepackage{comment}
\usepackage{enumerate}
\usepackage{graphicx}
\usepackage[ocgcolorlinks,allcolors={blue}]{hyperref}
\usepackage{xcolor}
\usepackage{authblk}
\usepackage{pgfplots}
\DeclareMathAlphabet{\mathpgoth}{OT1}{pgoth}{m}{n}
\usepackage{multirow}
\usepackage{booktabs}
\usepackage[compact]{titlesec}
\usepackage{rotating}
\usepackage{tikz}
\usetikzlibrary{spy}
\usepackage{pgfplots}
\usepackage{pgfplotstable}
\pgfplotsset{compat=newest}

\usepackage[font={footnotesize}]{subcaption}
\usepackage[absolute,showboxes,overlay]{textpos}     % déclaration du package
\textblockorigin{0pt}{0pt}   
\definecolor{darkbrown}{HTML}{996633}
\newcommand{\logLogSlopeTriangle}[5]
{
    \pgfplotsextra
    {
        \pgfkeysgetvalue{/pgfplots/xmin}{\xmin}
        \pgfkeysgetvalue{/pgfplots/xmax}{\xmax}
        \pgfkeysgetvalue{/pgfplots/ymin}{\ymin}
        \pgfkeysgetvalue{/pgfplots/ymax}{\ymax}

        \pgfmathsetmacro{\xArel}{#1}
        \pgfmathsetmacro{\yArel}{#3}
        \pgfmathsetmacro{\xBrel}{#1-#2}
        \pgfmathsetmacro{\yBrel}{\yArel}
        \pgfmathsetmacro{\xCrel}{\xArel}
        \pgfmathsetmacro{\lnxB}{\xmin*(1-(#1-#2))+\xmax*(#1-#2)} % in [xmin,xmax].
        \pgfmathsetmacro{\lnxA}{\xmin*(1-#1)+\xmax*#1} % in [xmin,xmax].
        \pgfmathsetmacro{\lnyA}{\ymin*(1-#3)+\ymax*#3} % in [ymin,ymax].
        \pgfmathsetmacro{\lnyC}{\lnyA+#4*(\lnxA-\lnxB)}
        \pgfmathsetmacro{\yCrel}{\lnyC-\ymin)/(\ymax-\ymin)} % THE IMPROVED EXPRESSION WITHOUT 'DIMENSION TOO LARGE' ERROR.

        \coordinate (A) at (rel axis cs:\xArel,\yArel);
        \coordinate (B) at (rel axis cs:\xBrel,\yBrel);
        \coordinate (C) at (rel axis cs:\xCrel,\yCrel);

        \draw[black]   (A)-- node[pos=0.5,anchor=north] {\scriptsize{1}}
                    (B)-- 
                    (C)-- node[pos=0.,anchor=west] {\scriptsize{\color{#5}#4}} %% node[pos=0.5,anchor=west] {#4}
                    (A);
    }
}

\usepackage{stackengine}

\usepackage{newtxtext,newtxmath}
\usepackage[style=numeric-comp,maxbibnames=99]{biblatex}

\bibliography{pstokes}

\newtheorem{theorem}{Theorem}
\newtheorem{proposition}[theorem]{Proposition}
\newtheorem{lemma}[theorem]{Lemma}
\newtheorem{corollary}[theorem]{Corollary}
\theoremstyle{remark}
\newtheorem{remark}[theorem]{Remark}
\theoremstyle{definition}
\newtheorem{assumption}{Assumption}

\newtheorem{example}[theorem]{Example}

\DeclareMathOperator{\card}{card}

\newcommand{\email}[1]{\href{mailto:#1}{#1}}

\title{A Hybrid High-Order method for creeping flows of non-Newtonian fluids}
\author[2]{Michele Botti\footnote{\email{michele.botti@polimi.it}}}
\author[1]{Daniel Castanon Quiroz \footnote{\email{danielcq.mathematics@gmail.com}}}	
\author[1]{Daniele A. Di Pietro \footnote{\email{daniele.di-pietro@umontpellier.fr}}}
\author[1]{Andr\'{e} Harnist \footnote{\email{andre.harnist@umontpellier.fr}, corresponding author}}
\affil[1]{IMAG, Univ Montpellier, CNRS, Montpellier, France}
\affil[2]{MOX, Department of Mathematics, Politecnico di Milano, Milano, Italy}

\def\b{\boldsymbol}
\newcommand*\cst[1]{\mathrm{#1}}

\newcommand{\R}{\mathbb{R}} %reals
\newcommand{\N}{\mathbb{N}} %naturals
\newcommand{\Poly}{\mathbb{P}} %polynomials
\newcommand{\M}[1]{\R^{#1 \times #1}} %square matrices
\newcommand{\Ms}[1]{\R^{#1 \times #1}_\cst{s}} %symmetric matrices
\DeclareMathOperator{\Kernel}{Ker} %kernel
\DeclareMathOperator{\Image}{Im} %image
\newcommand\bdU[2]{{\bu U}_{#1}^{#2}} %discrete velocity space
\newcommand\dP[2]{P_{#1}^{#2}} %discrete pressure space

\def\und{\underline} %discrete scalar function
\newcommand\bu[1]{\und{\b{#1}}} %discrete vector or matrix function
\newcommand{\T}{\mathcal{T}} %elements
\newcommand{\F}{\mathcal{F}} %faces
\newcommand{\Fb}{\F_h^\cst{b}} %boundary faces
\newcommand{\Fi}{\F_h^\cst{i}} %interfaces

\newcommand{\res}[1]{\ \!\!_{|_{#1}}} %restriction
\newcommand{\dgrads}[2]{\b{\cst{G}}^{#1}_{\cst{s},#2}} %discrete symmetric gradient
\newcommand{\bdrec}[2]{\b{\cst{r}}^{#1}_{#2}} %discrete reconstruction operator
\newcommand{\bdfbres}[2]{\b{\Delta}^{#1}_{#2}} %discrete face-based residual operator
\newcommand{\ddiv}[2]{\cst{D}^{#1}_{#2}} %discrete divergence
\newcommand{\GRAD}{\b\nabla} %gradient
\newcommand{\brkGRAD}{\b\nabla_h} %broken gradient
\newcommand{\GRADs}{\b{\nabla}_\cst{s}} %symmetric gradient
\newcommand{\GRADss}{\b{\nabla}_\cst{ss}} %skew symmetric gradient
\newcommand{\tsprod}[2]{#1 \otimes #2} %skew symmetric gradient
\newcommand{\brkGRADs}{\b{\nabla}_{\cst{s},h}} %broken symmetric gradient
\renewcommand{\div}{\b\nabla{\cdot}} %divergence for vector
\newcommand{\DIV}{\b\nabla{\cdot}} %divergence for matrices
\newcommand{\brkDIV}{\b\nabla_h \cdot} %broken divergence for matrices
\newcommand\bI[2]{\bu{I}_{#1}^{#2}} %interpolator
\newcommand\Iav[2]{\b{I}_{\cst{av},#1}^{#2}} %averagi<ng operator
\newcommand\proj[2]{\pi_{#1}^{#2}} %scalar projection
\newcommand\PROJ[2]{\b{\pi}_{#1}^{#2}} %projection for vectors or matrices
\newcommand\stress{\b\sigma} %shear stress-strain rate function

\newcommand\sob{r}
\newcommand\sobs{\tilde{r}}

\newcommand{\ud}{\,\mathrm{d}}

\begin{document}

\maketitle

\begin{abstract}
  In this paper, we design and analyze a Hybrid High-Order discretization method for the steady motion of non-Newtonian, incompressible fluids in the Stokes approximation of small velocities.
  The proposed method has several appealing features including the support of general meshes and high-order, unconditional inf-sup stability, and orders of convergence that match those obtained for scalar Leray--Lions problems.
  A complete well-posedness and convergence analysis of the method is carried out under new, general assumptions on the strain rate-shear stress law, which encompass several common examples such as the power-law and Carreau--Yasuda models.
  Numerical examples complete the exposition.
\end{abstract}

%% \tableofcontents

%------------------------------------------------------------------------------%

\section{Introduction}\label{sec:introduction}

In this paper, we design and analyze a Hybrid High-Order (HHO) discretization method for the steady motion of a non-Newtonian, incompressible fluid in the Stokes approximation of small velocities. Notable applications include ice sheet dynamics \cite{Isaac.Stadler.ea:15}, mantle convection \cite{Schubert.Turcotte.ea:01}, chemical engineering \cite{Ko.Pustejovska.ea:18}, and biological fluids rheology \cite{Lai.Kuei.ea:78,Galdi.Rannacher.ea:08}.
We focus on fluids with shear-rate-dependent viscosity, whose behavior is characterized by a nonlinear strain rate-shear stress function. Physical interpretations and discussions of non-Newtonian fluid models can be found, e.g., in \cite{Bird.Armstrong.ea:87,Malek.Rajagopal.ea:95}. Typical examples that are frequently used in the applications include the power-law and Carreau--Yasuda model, covered by the present analysis.

The earliest investigations of fluids with shear-dependent viscosity date back to the pioneering work of Ladyzhenskaya \cite{Ladyzhenskaya:69}. For a detailed mathematical study of the well-posedness and regularity of the continuous problem, see also \cite{Malek.Rajagopal:05, Ruzicka.Diening:07, Diening.Ettwein:08, Beirao-da-Veiga:09, Berselli.Ruzicka:20} and references therein. Early results on the numerical analysis of non-Newtonian fluid flow problems were given in \cite{Sandri:93,Barrett.Liu:94,Glowinski.Rappaz:03}. Later, these results were improved in \cite{Belenki.Berselli.ea:12} and \cite{Hirn:13} by proving error estimates that are optimal for fluids with shear thinning behavior (described by a power-law exponent $\sob\le 2$). In \cite{Belenki.Berselli.ea:12}, the authors considered a conforming inf-sup stable finite element discretization, while in \cite{Hirn:13} a low-order scheme with local projection stabilization was proposed. In both works, the use of Orlicz functions is instrumental to unify the treatment of the shear thinning and shear thickening cases (also called pseudoplastic and dilatant, respectively; cf. Example \ref{ex:Carreau--Yasuda}). 
More recently, a finite element method based on a four-field formulation of the nonlinear Stokes equations has been analyzed in \cite{Sandri:14}. Other notable contributions on the numerical approximation of generalized Stokes problems include \cite{Diening.Kreuzer.ea:13,Isaac.Stadler.ea:15,Kreuzer.Suli:16,Ko.Suli:18}.

The main issues to be accounted for in the numerical solution of non-Newtonian fluid flow problems are the presence of local features emerging from the nonlinear strain rate-shear stress relation, the incompressibility condition leading to indefinite systems, the roughly varying model coefficients, and, possibly, complex geometries requiring unstructured and highly-adapted meshes. The HHO method provides several advantages to deal with the complex nature of the problem, such as the support of general polygonal or polyhedral meshes, the possibility to select the approximation order, and unconditional inf-sup stability. Moreover, HHO schemes can be efficiently implemented thanks to the possibility of statically condensing a large subset of the unknowns for linearized versions of the problem encountered, e.g., when solving
the nonlinear system by the Newton method.
Hybrid High-Order methods have been successfully applied to the simulation of incompressible flows of Newtonian fluids governed by the Stokes \cite{Aghili.Boyaval.ea:15} and Navier--Stokes equations \cite{Di-Pietro.Krell:18,Botti.Di-Pietro.ea:19*1}, possibly driven by large irrotational volumetric forces \cite{Di-Pietro.Ern.ea:16,Castanon-Quiroz.Di-Pietro:20}. 
Works related to the problem of creeping flows of non-Newtonian fluids are \cite{Botti.Di-Pietro.ea:17} and \cite{Di-Pietro.Droniou:17,Di-Pietro.Droniou:17*1}, respectively dealing with nonlinear elasticity and Leray--Lions problems. Going from nonlinear coercive elliptic equations to the nonlinear Stokes system involves additional difficulties arising from the pressure and the divergence constraint.
Finally, we mention that HHO methods are members of a wider family of polytopal methods that also includes, e.g., Virtual Element methods (cf., e.g., \cite{Beirao-da-Veiga.Lovadina.ea:17,Beirao-da-Veiga.Lovadina.ea:18} for their application to Newtonian incompressible flows) and can fit within general frameworks for the approximation of nonlinear problems such as the one provided by the Gradient Discretisation Method (see \cite{Droniou.Eymard.ea:18,Di-Pietro.Droniou.ea:18}).

The HHO discretization presented in this paper hinges on discontinuous polynomial unknowns on the mesh and on its skeleton, from which discrete differential operators are reconstructed.
These operators are used to formulate discrete counterparts of the viscous and pressure-velocity coupling terms.
For the former, stability is ensured by a cleverly designed stabilization contribution involving the penalization of boundary differences.
We carry out a complete analysis of the proposed method. In particular, under general assumptions on the strain rate-shear stress function, we derive error estimates for the velocity and pressure approximations. The energy-norm error estimate for the velocity given in Theorem \ref{thm:error.estimate} yields the same convergence orders established in  \cite[Theorem 3.2]{Di-Pietro.Droniou:17*1} for the scalar Leray--Lions elliptic problem.
A key tool in our analysis is provided by Lemma \ref{lem:discrete.korn.inequality}, in which we prove a generalization of the discrete Korn inequality of \cite[Lemma 1]{Botti.Di-Pietro.ea:19*1} to the non-Hilbertian case. 
The other main contributions are a novel formulation of the requirements on the strain rate-shear stress function allowing a unified treatment of pseudoplastic and dilatant fluids and the identification of a set of general assumptions on the nonlinear stabilization function ensuring the desired consistency properties along with the well-posedness of the discrete problem.

The rest of the paper is organized as follows. In Section \ref{sec:continuous.setting} we introduce the strong and weak formulations of the nonlinear Stokes problem and present the assumptions on the strain rate-shear stress function.
The discrete setting is established in Section \ref{sec:discrete.setting}, including the definition of the discrete spaces for the velocity and the pressure.
The HHO scheme along with the main theoretical results are stated in Section \ref{sec:discrete.problem}, and a numerical validation is provided in Section \ref{sec:num.res}.
In Section \ref{sec:discrete.korn.inequality} we prove the discrete counterpart of the Korn inequality needed in the analysis of the method.
Section \ref{sec:analysis} contains the proof of the main results (well-posedness and error estimates).
Finally, in Appendix \ref{sec:properties.stress} we provide a sufficient condition for the strain rate-shear stress law to fulfil the assumptions presented in Section \ref{sec:continuous.setting}.
  The paper is structured so as to offer two levels of reading.
  In particular, the reader mainly interested in the formulation of the method and its numerical performance can focus on Section \ref{sec:continuous.setting}--\ref{sec:num.res}.
  The remaining sections cover technical aspects of the analysis, and can be skipped at first reading.

%------------------------------------------------------------------------------%

\section{Continuous setting}\label{sec:continuous.setting}

Let $\Omega \subset \R^d$, $d\in\{2,3\}$, denote a bounded, connected, polyhedral open set with Lipschitz boundary $\partial\Omega$. We consider a possibly non-Newtonian fluid occupying $\Omega$ and subjected to a volumetric force field $\b f : \Omega \to \R^d$. Its flow is governed by the generalized Stokes problem, which consists in finding the velocity field $\b u : \Omega \to \R^d$ and the pressure field $p : \Omega \to \R$ such that
\begin{subequations}\label{eq:stokes.continuous}
  \begin{alignat}{2} 
    -\DIV\stress(\cdot,\GRADs \b u) + \GRAD p &= \b f &\qquad& \mbox{  in  } \Omega, \label{eq:stokes.continuous:momentum} \\
    \div \b u &= 0 &\qquad& \mbox{ in }  \Omega, \label{eq:stokes.continuous:mass} \\
    \b u &= \b 0 &\qquad& \mbox{ on } \partial \Omega, \label{eq:stokes.continuous:bc} \\
    \int_\Omega p(\b x)\ud \b x &= 0, \label{eq:stokes.continuous:closure}
  \end{alignat}
\end{subequations} 
where $\DIV$ denotes the divergence operator applied to vector or tensor fields, $\GRADs$ is the symmetric part of the gradient operator $\GRAD$ applied to vector fields, and, denoting by $\Ms{d}$ the set of square, symmetric, real-valued $d\times d$ matrices, $\stress : \Omega \times \Ms{d} \to \Ms{d}$ is the strain rate-shear stress law.
In what follows, we formulate assumptions on $\stress$ that encompass common models for non-Newtonian fluids and state a weak formulation for problem \eqref{eq:stokes.continuous} that will be used as a starting point for its discretization.

\subsection{Strain rate-shear stress law}\label{sec:strain.rate.shear.stress.law}

We define the Frobenius inner product such that, for all $\b\tau= (\tau_{ij})_{1 \le i,j \le d}$ and $\b\eta= (\eta_{ij})_{1 \le i,j \le d}$ in $\M{d}$, $\b\tau : \b\eta \coloneqq \sum_{i,j=1}^d \tau_{ij}\eta_{ij}$, and we denote by $|\b\tau|_{d \times d}\coloneqq \sqrt{\b\tau : \b\tau}$ the corresponding norm.

\begin{assumption}[Strain rate-shear stress law]\label{ass:stress}
  Let a real number $\sob \in (1,\infty)$ be fixed, denote by $\sob' \coloneqq \frac{\sob}{\sob-1} \in (1,\infty)$  the conjugate exponent of $\sob$, and define the singular exponent of $\sob$ by
  \begin{equation}\label{eq:sing}
    \sobs \coloneq \min(\sob,2) \in (1,2].
  \end{equation}
  The strain rate-shear stress law satisfies
  \begin{subequations}\label{eq:ass:sigma}
    \begin{gather}
      \stress(\b x,\b 0) = \b 0 \text{ for almost every } \b x \in \Omega,\label{eq:ass-stress:0}
      \\
      \stress : \Omega \times \Ms{d} \to \Ms{d} \text{ is measurable}.\label{eq:ass-stress:power-framed} 
    \end{gather}
    Moreover, there exist real numbers $\sigma_\cst{de} \in [0,\infty)$ and $\sigma_\cst{hc},\sigma_\cst{sm} \in (0,\infty)$ such that, for all $\b\tau,\b\eta \in \Ms{d}$ and almost every $\b x \in \Omega$, we have the H\"older continuity property
      \label{eq:power-framed:s.holder.continuity.strong.monotonicity}
      \begin{align} 
        \left|
        \stress(\b x,\b\tau)-\stress(\b x,\b\eta)
        \right|_{d\times d} &\le \sigma_\cst{hc} \left(\sigma_\cst{de}^\sob+|\b\tau|_{d\times d}^\sob+|\b\eta|_{d\times d}^\sob\right)^\frac{\sob-\sobs}{\sob}| \b\tau-\b\eta |_{d\times d}^{\sobs-1},\label{eq:power-framed:s.holder.continuity} 
      \end{align}
      and the strong monotonicity property
      \begin{align}
        \left(\stress(\b x,\b\tau)-\stress(\b x,\b\eta)\right):(\b\tau-\b\eta) \left(\sigma_\cst{de}^\sob+|\b\tau|_{d\times d}^\sob+|\b\eta|_{d\times d}^\sob\right)^\frac{2-\sobs}{\sob} \ge \sigma_\cst{sm}|\b\tau-\b\eta|_{d\times d}^{\sob+2-\sobs}.\label{eq:power-framed:s.strong.monotonicity}
      \end{align}    
  \end{subequations}
\end{assumption}

Some remarks are in order.

\begin{remark}[Residual shear stress]
  Assumption \eqref{eq:ass-stress:0} can be relaxed by taking $\stress(\cdot,\b 0)  \in L^{\sob'}(\Omega,\Ms{d})$. This modification requires only minor changes in the analysis, not detailed for the sake of conciseness.
\end{remark}

\begin{remark}[Singular exponent]
  Inequalities \eqref{eq:power-framed:s.holder.continuity}--\eqref{eq:power-framed:s.strong.monotonicity} can be proved starting from the following assumptions, which correspond to the conditions \eqref{eq:od:power-framed:holder.continuity.strong.monotonicity} below characterizing an $\sob$-power-framed function:
  For all $\b\tau,\b\eta \in \Ms{d}$ with $\b\tau \neq \b\eta$ and almost every $\b x \in \Omega$,
  \[  \begin{aligned} 
      |\stress(\b x,\b\tau)-\stress(\b x,\b\eta)|_{d \times d} &\le \sigma_\cst{hc} \left(\sigma_\cst{de}^\sob+|\b\tau|_{d \times d}^\sob+|\b\eta|_{d \times d}^\sob\right)^\frac{\sob-2}{\sob}| \b\tau-\b\eta |_{d \times d},
      \\
      \left(\stress(\b x,\b\tau)-\stress(\b x,\b\eta)\right):\left(\b\tau-\b\eta\right) &\ge \sigma_\cst{sm}\left(\sigma_\cst{de}^\sob+|\b\tau|_{d \times d}^\sob+|\b\eta|_{d \times d}^\sob\right)^\frac{\sob-2}{\sob}|\b\tau-\b\eta|_{d \times d}^{2}.
    \end{aligned}
  \]
  These relations are reminiscent of the ones used in \cite{Di-Pietro.Droniou:17*1} in the context of scalar Leray--Lions problems.
  The advantage of assumptions \eqref{eq:power-framed:s.holder.continuity}-\eqref{eq:power-framed:s.strong.monotonicity}, expressed in terms of the singular index $\sobs$, is that they enable a unified treatment of the cases $\sob < 2$ and $\sob \ge 2$ in the proofs of Lemma \ref{lem:ah:holder.continuity.strong.monotonicity}, Theorem \ref{thm:well-posedness}, Lemma \ref{lem:consistency:ah}, and Theorem \ref{thm:error.estimate} below.
\end{remark}

\begin{remark}[Relations between the H\"older and monotonicity constants]
  Inequalities \eqref{eq:power-framed:s.holder.continuity} and \eqref{eq:power-framed:s.strong.monotonicity} give
  \begin{equation}\label{eq:power-framed:constants.bound}
    \sigma_\cst{sm} \leq \sigma_\cst{hc}.
  \end{equation}
  Indeed, let $\b\tau \in \Ms{d}$ be such that $|\b\tau|_{d\times d} > 0$. Using the strong monotonicity \eqref{eq:power-framed:s.strong.monotonicity} (with $\b \eta = \b 0$), the Cauchy--Schwarz inequality, and the H\"older continuity \eqref{eq:power-framed:s.holder.continuity} (again with $\b \eta = \b 0$), we infer that
  \[
  \begin{aligned}
    \sigma_\cst{sm}\left(\sigma_\cst{de}^\sob+|\b\tau|_{d\times d}^\sob\right)^\frac{\sobs-2}{\sob}|\b\tau|_{d\times d}^{\sob+2-\sobs} &\leq  \stress(\cdot,\b\tau):\b\tau\leq |\stress(\cdot,\b\tau)|_{d\times d}|\b\tau|_{d\times d}\leq \sigma_\cst{hc}\left(\sigma_\cst{de}^\sob+|\b\tau|_{d\times d}^\sob\right)^\frac{\sob-\sobs}{\sob}| \b\tau|_{d\times d}^{\sobs}
  \end{aligned}
  \]
  almost everywhere in $\Omega$. Hence, $\frac{\sigma_\cst{sm}}{\sigma_\cst{hc}} \le \left(\frac{\sigma_\cst{de}^\sob+|\b\tau|_{d\times d}^\sob}{| \b\tau|_{d\times d}^\sob}\right)^\frac{|\sob-2|}{\sob}$. Letting $|\b\tau|_{d\times d} \to \infty$ gives \eqref{eq:power-framed:constants.bound}.
\end{remark}

\begin{example}[Carreau--Yasuda fluids]\label{ex:Carreau--Yasuda}
  $(\mu,\delta,a,\sob)$-Carreau--Yasuda fluids, introduced in \cite{Yasuda.Armstrong.ea:81} and later generalized in \cite[Eq. (1.2)]{Hirn:13}, are fluids for which it holds, for almost every $\b x\in\Omega$ and all $\b\tau \in \Ms{d}$,
  \begin{equation}\label{eq:Carreau--Yasuda}
    \stress(\b x,\b\tau) = \mu(\b x)\left(\delta^{a(\b x)}+|\b\tau|_{d \times d}^{a(\b x)}\right)^\frac{\sob-2}{a(\b x)}\b\tau,
  \end{equation}
  where $\mu : \Omega \to [\mu_-,\mu_+]$ is a measurable function with $\mu_-,\mu_+ \in (0,\infty)$ corresponding to the local flow consistency index, $\delta \in [0,\infty)$ is the degeneracy parameter, $a : \Omega \to [a_-,a_+]$ is a measurable function with $a_-,a_+ \in (0,\infty)$ expressing the local transition flow behavior index, and $\sob \in (1,\infty)$ is the flow behavior index. The Carreau--Yasuda law is a generalization of the Carreau law (corresponding to $a_- = a_+ = 2$) that takes into account the different local levels of flow behavior in the fluid. The degenerate case $\delta=0$ corresponds to the power-law model. 
    Non-Newtonian fluids described by constitutive laws with a $(\mu,\delta,a,\sob)$-structure exhibit a different behavior according to the value of $\sob$. If $\sob > 2$, then the fluid shows shear thickening behavior and is called \emph{dilatant}. Examples of dilatant fluids are wet sand and oobleck. The case $\sob < 2$, on the other hand, corresponds to \emph{pseudoplastic} fluids having shear thinning behavior, such as blood. Finally, if $\sob = 2$, then the fluid is Newtonian and \eqref{eq:stokes.continuous} becomes the classical (linear) Stokes problem. 
    We show in Appendix \ref{sec:properties.stress} that the strain rate-shear stress law \eqref{eq:Carreau--Yasuda} is an $\sob$-power-framed function with $\sigma_\cst{de} = \delta$,
    \[
    \sigma_\cst{hc} = \begin{cases}
      \frac{\mu_+}{\sob-1}2^{\left[-\left(\frac{1}{a_+}-\frac{1}{\sob}\right)^\ominus-1\right](\sob-2)+\frac{1}{r}} & \text{if } \sob < 2, 
      \\
      \mu_+(\sob-1)2^{\left(\frac{1}{a_-}-\frac{1}{\sob}\right)^\oplus(\sob-2)} & \text{if } \sob \ge 2,
    \end{cases}
    \quad\text{and}\quad
    \sigma_\cst{sm} = \begin{cases}
      \mu_-(\sob-1)2^{\left(\frac{1}{a_-}-\frac{1}{\sob}\right)^\oplus(\sob-2)} & \text{if } \sob \le 2, 
      \\
      \frac{\mu_-}{\sob-1}2^{\left[-\left(\frac{1}{a_+}-\frac{1}{\sob}\right)^\ominus-1\right](\sob-2)-1} & \text{if } \sob > 2,
    \end{cases}
    \]
where $\xi^\oplus\coloneq\max(0,\xi)$ and $\xi^\ominus\coloneq-\min(0,\xi)$ denote, respectively, the positive and negative parts of a real number $\xi$. As a consequence, it matches Assumption \ref{ass:stress}.
\end{example}

\subsection{Weak formulation}\label{sec:weak.formulation}

From this point on, we omit both the integration variable and the measure from integrals, as they can be in all cases inferred from the context.
We define the following velocity and pressure spaces embedding, respectively, the homogeneous boundary condition and the zero-average constraint: 
\[
\b U \coloneqq \left\{\b v \in W^{1,\sob}(\Omega,\R^d)\ : \ \b v\res{\partial\Omega} = \b 0 \right\},
\qquad
P \coloneqq L^{\sob'}_0(\Omega,\R) \coloneqq \left\{q \in L^{\sob'}(\Omega,\R)\ : \ \textstyle\int_\Omega q = 0 \right\}. 
\]
Assuming $\b f \in L^{\sob'}(\Omega,\R^d)$, the weak formulation of problem \eqref{eq:stokes.continuous} reads:
Find $(\b u,p) \in \b U \times P$ such that
\begin{subequations}\label{eq:stokes.weak}
  \begin{alignat}{2}
     a(\b u,\b v)+b(\b v,p) &= \displaystyle\int_\Omega \b f \cdot \b v &\qquad \forall \b v \in \b U,\label{eq:stokes.weak:momentum} \\
     -b(\b u,q) &= 0 &\qquad \forall q \in P,  \label{eq:stokes.weak:mass} 
  \end{alignat}
\end{subequations}
where the function $a : \b U \times \b U \to \R$ and the bilinear form $b : \b U \times L^{\sob'}(\Omega,\R) \to \R$ are defined such that, for all $\b v,\b w \in \b U$ and all $q \in L^{\sob'}(\Omega,\R)$,
\begin{equation}\label{eq:a.b}
  a(\b w,\b v) \coloneqq \displaystyle\int_\Omega \stress(\cdot,\GRADs \b w) : \GRADs \b v,\qquad
  b(\b v,q)  \coloneqq -\displaystyle\int_\Omega (\div \b v) q.
\end{equation}

\begin{remark}[Mass equation]
  The test space in \eqref{eq:stokes.weak:mass} can be extended to $L^{\sob'}(\Omega,\R)$ since, for all $\b v \in \b U$, the divergence theorem and the fact that $\b v\res{\partial\Omega} = \b 0$ yield $b(\b v,1) = -\int_\Omega \div \b v = - \int_{\partial\Omega} \b v \cdot \b n_{\partial \Omega} = 0$, with $\b n_{\partial \Omega}$ denoting the unit vector normal to $\partial\Omega$ and pointing out of $\Omega$.
\end{remark}

\begin{remark}[Well-posedness and a priori estimates]\label{rem:a-priori}
  It can be checked that, under Assumption \ref{ass:stress}, the continuous problem \eqref{eq:stokes.weak} admits a unique solution $(\b u,p) \in \b U \times P$; see, e.g., \cite[Section 2.4]{Hirn:13}, where slightly stronger assumptions are considered.
  For future use, we also note the following a priori bound on the velocity:
  \begin{equation}\label{eq:continuous.solution:bounds:uh}
    |\b u|_{W^{1,\sob}(\Omega,\R^d)} \le \left(2^{\frac{2-\sobs}{\sob}}C_\cst{K}\sigma_\cst{sm}^{-1}\| \b f \|_{L^{\sob'}(\Omega,\R^d)}\right)^\frac{1}{\sob-1}+\left(2^{\frac{2-\sobs}{\sob}}C_\cst{K}|\Omega|_d^{\frac{2-\sobs}{\sob}}\sigma_\cst{de}^{2-\sobs}\sigma_\cst{sm}^{-1}\| \b f \|_{L^{\sob'}(\Omega,\R^d)}\right)^\frac{1}{\sob+1-\sobs},
  \end{equation}
    where $C_\cst{K}>0$ comes from the Korn inequality given at \eqref{eq:Korn} below.
  To prove \eqref{eq:continuous.solution:bounds:uh}, use the strong-monotonicity \eqref{eq:power-framed:s.strong.monotonicity} of $\stress$, sum \eqref{eq:stokes.weak:momentum} written for $\b v=\b u$ to \eqref{eq:stokes.weak:mass} written for $q=p$, and use the H\"older inequality together with the Korn inequality \eqref{eq:Korn} to write
  \[
  \begin{aligned}
    \sigma_\cst{sm}\left(
    |\Omega|_d\sigma_\cst{de}^\sob + \|\GRADs \b u\|_{L^\sob(\Omega,\M{d})}^\sob
    \right)^\frac{\sobs-2}{\sob} \|\GRADs \b u\|_{L^\sob(\Omega,\M{d})}^{\sob+2-\sobs}
    &\le  a(\b u,\b u) \\
    &= \displaystyle\int_\Omega \b f \cdot \b u
    \le C_\cst{K} \| \b f \|_{L^{\sob'}(\Omega,\R^d)}\|\GRADs \b u\|_{L^\sob(\Omega,\M{d})},
  \end{aligned}
  \]
where $|\Omega|_d$ is the measure of $\Omega$, that is,
  \begin{equation}\label{eq:continuous.solution:bounds:uh:1}
    \mathcal{N}\coloneqq \left(
    |\Omega|_d\sigma_\cst{de}^\sob + \|\GRADs \b u\|_{L^\sob(\Omega,\M{d})}^\sob
    \right)^\frac{\sobs-2}{\sob} \|\GRADs \b u\|_{L^\sob(\Omega,\M{d})}^{\sob+1-\sobs} \le C_\cst{K} \sigma_\cst{sm}^{-1}\| \b f \|_{L^{\sob'}(\Omega,\R^d)}.
  \end{equation}
  Observing that $\|\GRADs \b u\|_{L^\sob(\Omega,\M{d})}^{\sob+1-\sobs}  \le 2^{\frac{2-\sobs}{\sob}}\max\left(\|\GRADs \b u\|_{L^\sob(\Omega,\M{d})}^r,|\Omega|_d\sigma_\cst{de}^r\right)^\frac{2-\sobs}{r}\mathcal{N} $, we obtain, enumerating the cases for the maximum and summing the corresponding bounds, $\|\GRADs \b u\|_{L^\sob(\Omega,\M{d})} \le (2^{\frac{2-\sobs}{\sob}}\mathcal{N})^\frac{1}{\sob-1}+\Big(2^{\frac{2-\sobs}{\sob}}|\Omega|_d^{\frac{2-\sobs}{\sob}}\sigma_\cst{de}^{2-\sobs}\mathcal{N}\Big)^\frac{1}{\sob+1-\sobs} $.
  Combining this inequality with \eqref{eq:continuous.solution:bounds:uh:1} gives \eqref{eq:continuous.solution:bounds:uh}.
\end{remark}

%------------------------------------------------------------------------------%

\section{Discrete setting}\label{sec:discrete.setting}

\subsection{Mesh and notation for inequalities up to a multiplicative constant}

We define a mesh as a couple $\mathcal M_h\coloneq(\T_h,\F_h)$, where $\T_h$ is a finite collection of polyhedral elements $T$ such that $h=\max_{T\in\T_h}h_T$ with $h_T$ denoting the diameter of $T$, while $\F_h$ is a finite collection of planar faces $F$ with diameter $h_F$.
Notice that, here and in what follows, we use the three-dimensional nomenclature also when $d=2$, i.e., we speak of polyhedra and faces rather than polygons and edges.
It is assumed henceforth that the mesh $\mathcal M_h$ matches the geometrical requirements detailed in \cite[Definition 1.7]{Di-Pietro.Droniou:20}.
In order to have the boundedness property \eqref{eq:I:boundedness} for the interpolator, we additionally assume that the mesh elements are star-shaped with respect to every point of a ball of radius uniformly comparable to the element diameter; see \cite[Lemma 7.12]{Di-Pietro.Droniou:20} for the Hilbertian case.
Boundary faces lying on $\partial\Omega$ and internal faces contained in
$\Omega$ are collected in the sets $\F_h^{\rm b}$ and $\F_h^{\rm i}$, respectively.
For every mesh element $T\in\T_h$, we denote by $\F_T$ the subset of $\F_h$ containing the faces that lie on the boundary $\partial T$ of $T$. For every face $F \in \F_h$, we denote by $\T_F$ the subset of $\T_h$ containing the one (if $F\in\F_h^{\rm b}$) or two (if $F\in\F_h^{\rm i}$) elements on whose boundary $F$ lies.
Finally, for each mesh element $T\in\T_h$ and face $F\in\F_T$, $\b n_{TF}$ denotes the (constant) unit vector normal to $F$ pointing out of $T$.

Our focus is on the $h$-convergence analysis, so we consider a sequence of refined meshes that is regular in the sense of \cite[Definition 1.9]{Di-Pietro.Droniou:20} with regularity parameter uniformly bounded away from zero.
The mesh regularity assumption implies, in particular, that the diameter of a mesh element and those of its faces are comparable uniformly in $h$ and that the number of faces of one element is bounded above by an integer independent of $h$.

To avoid the proliferation of generic constants, we write henceforth $a\lesssim b$ (resp., $a\gtrsim b$) for the inequality $a\le Cb$ (resp., $a\ge Cb$) with real number $C>0$ independent of $h$, of the constants $\sigma_\cst{de},\sigma_\cst{hc},\sigma_\cst{sm}$ in Assumption \ref{ass:stress}, and, for local inequalities, of the mesh element or face on which the inequality holds.
We also write $a\simeq b$ to mean $a\lesssim b$ and $b\lesssim a$.
The dependencies of the hidden constants are further specified when needed.

\subsection{Projectors and broken spaces}
  
Given $X \in \T_h \cup \F_h$ and $l \in \N$, we denote by $\Poly^l(X,\R)$ the space spanned by the restriction to $X$ of scalar-valued, $d$-variate polynomials of total degree $\le l$.
The local $L^2$-orthogonal projector $\proj{X}{l} : L^{1}(X,\R) \to \Poly^l(X,\R)$ is defined such that, for all $v \in L^{1}(X,\R)$,
\begin{equation}\label{eq:proj}
  \displaystyle\int_X (\proj{X}{l} v-v) w = 0 \qquad \forall w \in  \Poly^{l}(X,\R).
\end{equation}
When applied to vector-valued fields in $L^1(X,\R^d)$ (resp., tensor-valued fields in $L^1(X,\M{d})$), the $L^2$-orthogonal projector mapping on $\Poly^l(X,\R^d)$ (resp., $\Poly^l(X,\M{d})$) acts component-wise and is denoted in boldface font.
Let $T\in\T_h$, $n\in[0,l+1]$ and $m\in[0,n]$.
The following $(n,\sob,m)$-approximation properties of $\proj{T}{l}$ hold:
For any $v\in W^{n,\sob}(T,\R)$,
\begin{subequations}\label{eq:proj:app}
\begin{equation}\label{eq:proj:app:T}
  |v-\proj{T}{l}v|_{W^{m,\sob}(T,\R)} \lesssim h_T^{n-m}|v|_{W^{n,\sob}(T,\R)}.
\end{equation}
The above property will also be used in what follows with $\sob$ replaced by its conjugate exponent $\sob'$.
If, additionally, $n\ge 1$, we have the following $(n,\sob')$-trace approximation property:
\begin{equation}\label{eq:proj:app:F}
    \|v-\proj{T}{l}v\|_{L^{\sob'}(\partial T,\R)}\lesssim h_T^{n-\frac{1}{\sob'}}|v|_{W^{n,\sob'}(T,\R)}.
\end{equation}
\end{subequations}
The hidden constants in \eqref{eq:proj:app} are independent of $h$ and $T$, but possibly depend on $d$, the mesh regularity parameter, $l$, $n$, and $\sob$.
The approximation properties \eqref{eq:proj:app} are proved for integer $n$ and $m$ in \cite[Appendix A.2]{Di-Pietro.Droniou:17} (see also \cite[Theorem 1.45]{Di-Pietro.Droniou:20}), and can be extended to non-integer values using standard interpolation techniques (see, e.g., \cite[Theorem 5.1]{Lions.Magenes:72}).

At the global level, for a given integer $l\ge 0$, we define the broken polynomial space $\Poly^l(\T_h,\R)$ spanned by functions in $L^1(\Omega,\R)$ whose restriction to each mesh element $T\in\T_h$ lies in $\Poly^l(T,\R)$, and we define the global $L^2$-orthogonal projector $\proj{h}{l} : L^{1}(\Omega,\R) \to \Poly^l(\T_h,\R)$ such that, for all $v \in L^{1}(\Omega,\R)$ and all $T \in \T_h$,
\[
(\proj{h}{l} v)\res{T} \coloneq \proj{T}{l} v\res{T}.
\]
Broken polynomial spaces are subspaces of the broken Sobolev spaces
\[
W^{n,\sob}(\T_h,\R)\coloneq\left\{ v\in L^\sob(\Omega,\R)\ : \ v\res{T}\in W^{n,\sob}(T,\R)\quad\forall T\in\T_h\right\}.
\]
We define the broken gradient operator $\brkGRAD : W^{1,1}(\T_h,\R) \rightarrow L^1(\Omega,\R^d)$ such that, for all $v \in W^{1,1}(\T_h,\R)$ and all $T \in \T_h$, $(\brkGRAD v)\res{T} \coloneq \GRAD v\res{T}$. We define similarly the broken gradient acting on vector fields along with its symmetric part $\brkGRADs$, as well as the broken divergence operator $\brkDIV$ acting on tensor fields.
The global $L^2$-orthogonal projector $\PROJ{h}{l}$ mapping vector-valued fields in $L^1(\Omega,\R^d)$ (resp., tensor-valued fields in $L^1(\Omega,\M{d})$) on $\Poly^l(\T_h,\R^d)$ (resp., $\Poly^l(\T_h,\M{d})$) is obtained applying $\proj{h}{l}$ component-wise.

\subsection{Discrete spaces and norms}

Let an integer $k\ge 1$ be fixed. The HHO space of discrete velocity unknowns is
\[
\bdU{h}{k} \coloneqq \left\{
\bu v_h = ((\b v_T)_{T \in \T_h},(\b v_F)_{F\in \F_h}) \ : \ \b v_T \in \Poly^k(T,\R^d)\ \ \forall T \in \T_h\ \mbox{ and }\ \b v_F \in \Poly^k(F,\R^d)\ \ \forall F \in \F_h \right\}.
\]
The interpolation operator $\bI{h}{k} : W^{1,1}(\Omega,\R^d) \to  \bdU{h}{k} $ maps a function $\b v \in W^{1,1}(\Omega,\R^d)$ on the vector of discrete unknowns $\bI{h}{k}\b v$ defined as follows:
\[
  \bI{h}{k} \b v \coloneqq ((\PROJ{T}{k} \b v\res{T})_{T \in \T_h},(\PROJ{F}{k} \b v\res{F})_{F \in \F_h}).
  \]
For all $T \in \T_h$, we denote by $\bdU{T}{k}$ and $\bI{T}{k}$ the restrictions of $\bdU{h}{k}$ and $\bI{h}{k}$ to $T$, respectively and, for all $\bu v_h \in \bdU{h}{k}$, we let $\bu v_T \coloneqq (\b v_T,(\b v_F)_{F\in \F_T}) \in \bdU{T}{k}$ denote the vector collecting the discrete unknowns attached to $T$ and its faces.
Furthermore, for all $\bu v_h \in \bdU{h}{k}$, we define the broken polynomial field $\b v_h\in\Poly^k(\T_h,\R^d)$ obtained patching element unknowns, that is,
\begin{equation}\label{eq:vh}
  (\b v_h)\res{T} \coloneqq \b v_T\qquad\forall T \in \T_h.
\end{equation}

We define on $\bdU{h}{k}$ the $W^{1,\sob}(\Omega,\R^d)$-like strain seminorm $\| {\cdot} \|_{\sob,h}$ such that, for all $\bu v_h \in \bdU{h}{k}$,
\begin{subequations}\label{eq:norm.epsilon.r}
  \begin{gather}\label{eq:norm.epsilon.r.h}
    \| \bu v_h \|_{\sob,h} \coloneqq \left(\displaystyle\sum_{T \in \T_h}\| \bu v_T \|_{\sob,T}^\sob\right)^\frac{1}{\sob}
    \\\label{eq:norm.epsilon.r.T}
    \text{with 
      $\| \bu v_T \|_{\sob,T} \coloneqq \left(\| \GRADs \b v_T \|^\sob_{L^\sob(T,\M{d})} + \displaystyle\sum_{F \in \F_T} h_F^{1-\sob} \| \b v_F - \b v_T\|^\sob_{L^\sob(F,\R^d)}\right)^\frac{1}{\sob}$
      for all $T \in \T_h$.
    }
  \end{gather}
\end{subequations}
The following boundedness property for $\bI{T}{k}$ can be proved adapting the arguments of \cite[Proposition 6.24]{Di-Pietro.Droniou:20} and requires the star-shaped assumption on the mesh elements:
For all $T \in \T_h$ and all $\b v \in W^{1,\sob}(T,\R^d)$,
\begin{equation}\label{eq:I:boundedness}
  \|\bI{T}{k} \b v\|_{\sob,T} \lesssim | \b v |_{W^{1,\sob}(T,\R^d)},
\end{equation}
where the hidden constant depends only on $d$, the mesh regularity parameter, $\sob$, and $k$.

The discrete velocity and pressure are sought in the following spaces, which embed, respectively, the homogeneous boundary condition for the velocity and the zero-average constraint for the pressure:
\[
\bdU{h,0}{k} \coloneqq \left\{ \bu v_h = ((\b v_T)_{T \in \T_h},(\b v_F)_{F\in \F_h}) \in \bdU{h}{k} \ : \ \b v_F = \b 0 \quad \forall F \in \Fb \right\},\quad
\dP{h}{k} \coloneqq \Poly^k(\T_h,\R)\cap P.
\]
By the discrete Korn inequality proved in Lemma \ref{lem:discrete.korn.inequality} below, $\| {\cdot} \|_{\sob,h}$ is a norm on $\bdU{h,0}{k}$ (the proof is obtained reasoning as in \cite[Corollary 2.16]{Di-Pietro.Droniou:20}).

%------------------------------------------------------------------------------%

\section{HHO scheme}\label{sec:discrete.problem}

In this section, after introducing the discrete counterparts of the viscous and pressure-velocity coupling terms, we state the discrete problem along with the main results.

\subsection{Viscous term}

\subsubsection{Local symmetric gradient reconstruction}

For all $T \in \T_h$, we define the local symmetric gradient reconstruction $\dgrads{k}{T} : \bdU{T}{k} \to \Poly^{k}(T,\Ms{d})$ such that, for all $\bu v_T \in \bdU{T}{k}$,
\begin{equation}\label{eq:G}
  \displaystyle\int_T \dgrads{k}{T} \bu v_T : \b\tau = \int_T \GRADs \b v_T : \b\tau + \sum_{F \in \F_T} \int_F (\b v_F-\b v_T)\cdot (\b\tau \b n_{TF})\qquad \forall \b\tau \in  \Poly^{k}(T,\Ms{d}).
\end{equation}
This symmetric gradient reconstruction, originally introduced in \cite[Section 4.2]{Botti.Di-Pietro.ea:17}, is designed so that the following relation holds (see, e.g., \cite[Proposition 5]{Botti.Di-Pietro.ea:17*1} or \cite[Section 7.2.5]{Di-Pietro.Droniou:20}):
For all $\b v\in W^{1,1}(T,\R^d)$,
\begin{equation}\label{eq:G:proj}
  \dgrads{k}{T} (\bI{T}{k} \b v) = \PROJ{T}{k}(\GRADs \b v).
\end{equation}
The global symmetric gradient reconstruction $\dgrads{k}{h} : \bdU{h}{k} \to \Poly^{k}(\T_h,\Ms{d})$ is obtained patching the local contributions, that is, for all $\bu v_h \in \bdU{h}{k}$, we set
\begin{equation}\label{eq:Gh}
  (\dgrads{k}{h} \bu v_h)\res{T} \coloneq \dgrads{k}{T} \bu v_T \qquad \forall T\in\T_h.
\end{equation}

\subsubsection{Discrete viscous function}

The discrete counterpart of the function $a$ defined by \eqref{eq:a.b} is $\cst{a}_h : \bdU{h}{k} \times \bdU{h}{k} \to \R$ such that, for all $\bu v_h,\bu w_h \in \bdU{h}{k}$,
\begin{equation}\label{eq:ah}
  \cst{a}_h(\bu w_h, \bu v_h) \coloneqq \displaystyle\int_\Omega \stress(\cdot,\dgrads{k}{h} \bu w_h): \dgrads{k}{h} \bu v_h +\gamma \cst{s}_h(\bu w_h,\bu v_h).
\end{equation}
In the above definition, recalling \eqref{eq:power-framed:constants.bound}, $\gamma$ is a stabilization parameter such that
\begin{equation}\label{eq:gamma}
  \gamma \in [\sigma_\cst{sm},\sigma_\cst{hc}],
\end{equation}
while the stabilization function $\cst{s}_h : \bdU{h}{k} \times \bdU{h}{k} \to \R$ is such that, for all $\bu v_h,\bu w_h \in \bdU{h}{k}$,
\begin{equation}\label{eq:sh}
  \cst{s}_h(\bu w_h,\bu v_h) \coloneqq \displaystyle\sum_{T \in \T_h}\cst{s}_T(\bu w_T,\bu v_T),
\end{equation}
where the local contributions are assumed to satisfy the following assumption.

\begin{assumption}[Local stabilization function]\label{ass:sT}
  For all $T \in \T_h$, the local stabilization function $\cst{s}_T:\bdU{T}{k}\times\bdU{T}{k}\to\R$ is linear in its second argument and satisfies the following properties, with hidden constants independent of both $h$ and $T$:
  %% \begin{subequations}\label{eq:ass:sT}
    \begin{enumerate}
    \item \emph{Stability and boundedness.} Recalling the definition \eqref{eq:norm.epsilon.r.T} of the local $\|{\cdot}\|_{\sob,T}$-seminorm, for all $\bu v_T\in\bdU{T}{k}$ it holds:
      \begin{equation}\label{eq:sT:stability.boundedness}
        \| \dgrads{k}{T}\bu v_T \|_{L^\sob(T,\M{d})}^\sob + \cst{s}_T(\bu v_T,\bu v_T)
        \simeq \|\bu v_T\|_{\sob,T}^\sob.
      \end{equation}
    \item \emph{Polynomial consistency.}
      For all $\b w\in\Poly^{k+1}(T,\R^d)$ and all $\bu v_T\in\bdU{T}{k}$,
      \begin{equation}\label{eq:sT:polynomial.consistency}
        \cst{s}_T(\bI{T}{k} \b w,\bu v_T) = 0.
      \end{equation}
    \item \emph{H\"{o}lder continuity.}
      For all $\bu u_T, \bu v_T, \bu w_T \in \bdU{T}{k}$, it holds, setting $\bu e_T\coloneq\bu u_T - \bu w_T$,
      \begin{equation}\label{eq:sT:holder-continuity}   
        \hspace{-0.5cm} \left|\cst{s}_T(\bu u_T,\bu v_T)-\cst{s}_T(\bu w_T,\bu v_T)\right|
        \lesssim 
        \left(\cst{s}_T(\bu u_T,\bu u_T)+\cst{s}_T(\bu w_T,\bu w_T)\right)^\frac{\sob-\sobs}{\sob}\cst{s}_T(\bu e_T,\bu e_T)^\frac{\sobs-1}{\sob}\cst{s}_T(\bu v_T,\bu v_T)^\frac{1}{\sob}.
      \end{equation}
    \item \emph{Strong monotonicity.}
      For all $\bu u_T, \bu w_T \in \bdU{T}{k}$ , it holds, setting again $\bu e_T\coloneq\bu u_T - \bu w_T$,
      \begin{equation}\label{eq:sT:strong-monotonicity} 
        \left(\cst{s}_T(\bu u_T,\bu e_T) - \cst{s}_T(\bu w_T,\bu e_T)\right)\left( \cst{s}_T(\bu u_T,\bu u_T)+\cst{s}_T(\bu w_T,\bu w_T)\right)^\frac{2-\sobs}{\sob} \gtrsim \cst{s}_T(\bu e_T,\bu e_T)^\frac{\sob+2-\sobs}{\sob}.
      \end{equation}
    \end{enumerate}
  %% \end{subequations}
\end{assumption}

\begin{remark}[Comparison with the linear case]
  If $\sob=2$, $\cst{s}_T$ can be any symmetric bilinear form satisfying \eqref{eq:sT:stability.boundedness}--\eqref{eq:sT:polynomial.consistency}.
  Indeed, property \eqref{eq:sT:holder-continuity} coincides in this case with the Cauchy--Schwarz inequality, while, by linearity of $\cst{s}_T$, property \eqref{eq:sT:strong-monotonicity} holds with the equal sign.
\end{remark}

\subsubsection{An example of viscous stabilization function}

Taking inspiration from the scalar case (cf., e.g., \cite[Eq. (4.11c)]{Di-Pietro.Droniou:17}), a local stabilization function that matches Assumption \ref{ass:sT} can be obtained setting, for all $\bu v_T,\bu w_T \in \bdU{T}{k}$,
\begin{equation}\label{eq:sT}
  \cst{s}_T(\bu w_T,\bu v_T) \coloneqq \int_{\partial T} |\bdfbres{k}{\partial T} \bu w_T|^{\sob-2}\bdfbres{k}{\partial T} \bu w_T \cdot \bdfbres{k}{\partial T} \bu v_T,
\end{equation}
where, denoting by $\Poly^k(\F_T,\R^d)$ the space of vector-valued broken polynomials of total degree $\le k$ on $\F_T$, the boundary residual operator $\bdfbres{k}{\partial T} : \bdU{T}{k} \to \Poly^k(\F_T,\R^d)$ is such that, for all $\bu v_T \in \bdU{T}{k}$,
\[
(\bdfbres{k}{\partial T} \bu v_T)\res{F} \coloneq
h_F^{-\frac1{\sob'}}\left(
  \PROJ{F}{k}(\bdrec{k+1}{T} \bu v_T-\b v_F)-\PROJ{T}{k}(\bdrec{k+1}{T} \bu v_T-\b v_T)
  \right)
\qquad\forall F\in\F_T,
\]
with velocity reconstruction $\bdrec{k+1}{T} : \bdU{T}{k} \to \Poly^{k+1}(T,\R^d)$ such that
\begin{gather*}
  \displaystyle\int_T (\GRADs \bdrec{k+1}{T} \bu v_T - \dgrads{k}{T} \bu v_T) : \GRADs\b w = 0\qquad\forall\b w \in  \Poly^{k+1}(T,\R^d), \\
  \text{
    $\displaystyle\int_T \bdrec{k+1}{T} \bu v_T = \int_T \b v_T$, and
    $\int_T \GRADss \bdrec{k+1}{T} \bu v_T = \frac{1}{2}\sum_{F \in \F_T} \int_F (\tsprod{\b v_F}{\b n_{TF}} - \tsprod{\b n_{TF}}{\b v_F})$.
  }
\end{gather*}
Above, $\GRADss$ denotes the skew-symmetric part of the gradient operator $\GRAD$ applied to vector fields and $\tsprod{}{}$ is the tensor product such that, for all $\b x = (x_i)_{1 \le i \le d}$ and $\b y = (y_i)_{1 \le i \le d}$ in $\R^d$, $\tsprod{\b x}{\b y} \coloneq (x_i y_j)_{1 \le i,j \le d} \in \M{d}$.

\begin{lemma}[Stabilization function \eqref{eq:sT}]\label{lem:sT}
  The local stabilization function defined by \eqref{eq:sT} satisfies Assumption \ref{ass:sT}.
\end{lemma}
\begin{proof}
  The proof of \eqref{eq:sT:stability.boundedness} for $\sob=2$ is given in \cite[Eq. (25)]{Botti.Di-Pietro.ea:17}. The result can be generalized to $\sob\neq 2$ using the same arguments of \cite[Lemma 5.2]{Di-Pietro.Droniou:17}.
  Property \eqref{eq:sT:polynomial.consistency} is an immediate consequence of the fact that $\bdfbres{k}{\partial T}(\bI{T}{k}\b w) = \b 0$ for any $\b w\in\Poly^{k+1}(T,\R^d)$, which can be proved reasoning as in \cite[Proposition 2.6]{Di-Pietro.Droniou:20}.
  
  Let us prove \eqref{eq:sT:holder-continuity}. First, we remark that, since the function $\alpha \mapsto \alpha^{\sob-2}$ verifies the conditions in \eqref{eq:1d.power-framed:eta} below, we can apply Theorem \ref{thm:1d.power-framed} to infer that the function $\R^d\ni\b x \mapsto |\b x|^{\sob-2}\b x$ satisfies for all $\b x,\b y \in \R^d$,
  \begin{subequations}\label{eq:rho.holder.strong}
    \begin{gather}
      \big|
      |\b x|^{\sob-2}\b x-|\b y|^{\sob-2}\b y
      \big| \lesssim \big(
      |\b x|^\sob+|\b y|^\sob
      \big)^\frac{\sob-\sobs}{\sob}|\b x-\b y|^{\sobs-1},\label{eq:rho:s.holder.continuity}
      \\
      \big(
      |\b x|^{\sob-2}\b x-|\b y|^{\sob-2}\b y
      \big) \cdot (\b x-\b y) \big(
      |\b x|^\sob+|\b y|^\sob
      \big)^\frac{2-\sobs}{\sob} \gtrsim |\b x-\b y|^{\sob+2-\sobs}.\label{eq:rho:s.strong.monotonicity}
    \end{gather}
  \end{subequations}
  Recalling \eqref{eq:sT}, we can write
  \[
  \begin{aligned}
    \left|\cst{s}_T(\bu u_T,\bu v_T)-\cst{s}_T(\bu w_T,\bu v_T)\right| &\leq  \int_{\partial T} \left||\bdfbres{k}{\partial T} \bu u_T|^{\sob-2}\bdfbres{k}{\partial T} \bu u_T-|\bdfbres{k}{\partial T} \bu w_T|^{\sob-2}\bdfbres{k}{\partial T} \bu w_T\right|| \bdfbres{k}{\partial T} \bu v_T|\\
    &\lesssim  \int_{\partial T} \left(|\bdfbres{k}{\partial T} \bu u_T|^\sob+|\bdfbres{k}{\partial T} \bu w_T|^\sob\right)^\frac{\sob-\sobs}{\sob}| \bdfbres{k}{\partial T} \bu e_T|^{\sobs-1}| \bdfbres{k}{\partial T} \bu v_T|\\
    &\le \left(\cst{s}_T(\bu u_T,\bu u_T)+\cst{s}_T(\bu w_T,\bu w_T)\right)^\frac{\sob-\sobs}{\sob}\cst{s}_T(\bu e_T,\bu e_T)^\frac{\sobs-1}{\sob}\cst{s}_T(\bu v_T,\bu v_T)^\frac{1}{\sob},
  \end{aligned}
  \]
  where we have used \eqref{eq:rho:s.holder.continuity} to pass to the second line and the $(1;\frac{\sob}{\sob-\sobs},\frac{\sob}{\sobs-1},\sob)$-H\"{o}lder inequality to conclude.
  
  Moving to \eqref{eq:sT:strong-monotonicity}, \eqref{eq:rho:s.strong.monotonicity} and the $(1;\frac{\sob+2-\sobs}{2-\sobs},\frac{\sob+2-\sobs}{\sob})$-H\"{o}lder inequality yield
  \[
  \begin{aligned}
    &\cst{s}_T(\bu e_T,\bu e_T) \\
    &\quad = \int_{\partial T} |\bdfbres{k}{\partial T}\bu u_T-\bdfbres{k}{\partial T}\bu w_T|^\sob \\
    &\quad
    \lesssim \int_{\partial T} \left(|\bdfbres{k}{\partial T} \bu u_T|^\sob+|\bdfbres{k}{\partial T} \bu w_T|^\sob\right)^\frac{2-\sobs}{\sob+2-\sobs}\left[
      \left(|\bdfbres{k}{\partial T}\bu u_T|^{\sob-2}\bdfbres{k}{\partial T}\bu u_T-|\bdfbres{k}{\partial T}\bu w_T|^{\sob-2}\bdfbres{k}{\partial T}\bu w_T\right)\cdot \bdfbres{k}{\partial T}\bu e_T
      \right]^\frac{\sob}{\sob+2-\sobs}\\
    &\quad
    \le \left( \cst{s}_T(\bu u_T,\bu u_T)+\cst{s}_T(\bu w_T,\bu w_T)\right)^\frac{2-\sobs}{\sob+2-\sobs}\left(\cst{s}_T(\bu u_T,\bu e_T)-\cst{s}_T(\bu w_T,\bu e_T)\right)^\frac{\sob}{\sob+2-\sobs}.\qedhere
  \end{aligned}
  \]
\end{proof}

\subsection{Pressure-velocity coupling}

For all $T \in \T_h$, we define the local divergence reconstruction $\ddiv{k}{T} : \bdU{T}{k} \to \Poly^{k}(T,\R)$ by setting, for all $\bu v_T \in \bdU{T}{k}$, $\ddiv{k}{T}\bu v_T \coloneq \cst{tr}(\dgrads{k}{T}\bu v_T)$.
We have the following characterization of $\ddiv{k}{T}$:
For all $\bu v_T \in \bdU{T}{k}$,
\begin{equation}\label{eq:D}
  \int_T \ddiv{k}{T} \bu v_T~ q = \int_T (\div \b v_T)~q
  + \sum_{F \in \F_T} \int_F (\b v_F-\b v_T)\cdot \b n_{TF}~q \qquad \forall q \in  \Poly^{k}(T,\R),
\end{equation}
as can be checked writing \eqref{eq:G} for $\b\tau = q\mathrm{I}_d$.
Taking the trace of \eqref{eq:G:proj}, it is inferred that, for all $T\in\T_h$ and all $\b v\in W^{1,1}(T,\R^d)$,
$\ddiv{k}{T} (\bI{T}{k} \b v) = \proj{T}{k}(\div \b v)$.
The pressure-velocity coupling is realized by the bilinear form $\cst{b}_h : \bdU{h}{k} \times \Poly^k(\T_h,\R) \to \R$ such that, for all $(\bu v_h,q_h) \in \bdU{h}{k} \times \Poly^k(\T_h,\R)$, setting $q_T\coloneq (q_h)\res{T}$ for all $T\in\T_h$,
\begin{equation}\label{eq:bh}
  \cst{b}_h(\bu v_h,q_h) \coloneqq -\sum_{T \in \T_h}\int_T  \ddiv{k}{T} \bu v_T~ q_T.
\end{equation}

\subsection{Discrete problem and main results}\label{sec:discrete.problem.main.results}

The discrete problem reads: Find $(\bu u_h,p_h) \in \bdU{h,0}{k} \times \dP{h}{k}$ such that
\begin{subequations}\label{eq:stokes.discrete}
  \begin{alignat}{2}
    \label{eq:stokes.discrete:momentum} \cst{a}_h(\bu u_h,\bu v_h) + \cst{b}_h(\bu v_h,p_h) &= \displaystyle\int_\Omega \b f \cdot \b v_h &\qquad& \forall \bu v_h \in \bdU{h,0}{k}, \\
    \label{eq:stokes.discrete:mass} -\cst{b}_h(\bu u_h,q_h) &= 0 &\qquad& \forall q_h \in \dP{h}{k}. 
  \end{alignat}
\end{subequations}
%
%Before proceeding, some remarks are in order.
\begin{remark}[Discrete mass equation]
The space of test functions in \eqref{eq:stokes.discrete:mass} can be extended to $\Poly^k(\T_h,\R)$ since, for all $\bu v_h \in \bdU{h,0}{k}$, the divergence theorem together with the fact that $\b v_F = \b 0$ for all $F \in \Fb$ and $\sum_{T\in \T_F} \int_F \b v_F \cdot \b n_{TF} = 0$ for all $F \in \Fi$, yield 
  \[
  \cst{b}_h(\bu v_h,1) = -\sum_{T \in \T_h} \sum_{F \in \F_T} \int_F \b v_F \cdot \b n_{TF} = -\sum_{F \in \Fi}\sum_{T \in \T_F}  \int_F \b v_F \cdot \b n_{TF} = 0.
  \]
\end{remark}
\begin{remark}[Efficient implementation]
  When solving the system of nonlinear algebraic equations corresponding to \eqref{eq:stokes.discrete} by, e.g., the Newton algorithm, all element-based velocity unknowns and all but one pressure unknown per element can be locally eliminated at each iteration by static condensation.
  As all the computations are local, this procedure is an embarrassingly parallel task which can fully benefit from multi-thread and multi-processor architectures.
  This implementation strategy has been described for the linear Stokes problem in \cite[Section 6.2]{Di-Pietro.Ern.ea:16}.
  After further eliminating the boundary unknowns by strongly enforcing the boundary condition \eqref{eq:stokes.continuous:bc}, we end up solving, at each iteration of the nonlinear solver, a linear system of size
  $d\cst{card}(\Fi){k+d-1\choose d-1} + \cst{card}(\T_h)$.
  Concerning the interplay between the static condensation strategy and the performance of $p$-multilevel linear solvers, we refer to \cite{Botti.Di-Pietro:21}.
\end{remark}

In what follows, we state the main results for the HHO scheme \eqref{eq:stokes.discrete}. The proofs are postponed to Section \ref{sec:analysis}.

\begin{theorem}[Well-posedness]\label{thm:well-posedness}
  There exists a unique solution $(\bu u_h,p_h) \in \bdU{h,0}{k} \times \dP{h}{k}$ to the discrete problem \eqref{eq:stokes.discrete}.
  Additionally, the following a priori bounds hold:
  \begin{subequations}\label{eq:discrete.solution:bounds}
    \begin{align}
      \| \bu u_h \|_{\sob,h} &\lesssim \left(\sigma_\cst{sm}^{-1}\| \b f \|_{L^{\sob'}(\Omega,\R^d)}\right)^\frac{1}{\sob-1}+\left(\sigma_\cst{de}^{2-\sobs}\sigma_\cst{sm}^{-1}\| \b f \|_{L^{\sob'}(\Omega,\R^d)}\right)^\frac{1}{\sob+1-\sobs}, \label{eq:discrete.solution:bounds:uh}\\
      \| p_h \|_{L^{\sob'}(\Omega,\R)} &\lesssim \sigma_\cst{hc}\left(\sigma_\cst{sm}^{-1}\| \b f \|_{L^{\sob'}(\Omega,\R^d)}+\sigma_\cst{de}^{|\sob-2|(\sobs-1)}\left(\sigma_\cst{sm}^{-1}\| \b f \|_{L^{\sob'}(\Omega,\R^d)}\right)^\frac{\sobs-1}{\sob+1-\sobs}\right). \label{eq:discrete.solution:bounds:ph}
    \end{align}
  \end{subequations}
\end{theorem}

\begin{proof}
  See Section \ref{sec:analysis:well-posedness}.
\end{proof}

\begin{theorem}[Error estimate]\label{thm:error.estimate}
  Let $(\b u,p) \in \b U \times P$ and $(\bu u_h,p_h) \in \bdU{h,0}{k} \times \dP{h}{k}$ solve \eqref{eq:stokes.weak} and \eqref{eq:stokes.discrete}, respectively. Assume the additional regularity $\b u \in W^{k+2,\sob}(\T_h,\R^d)$, $\stress(\cdot,\GRADs \b u) \in W^{1,\sob'}(\Omega,\Ms{d}) \cap W^{(k+1)(\sobs-1),\sob'}(\T_h,\Ms{d})$, and $p \in W^{1,\sob'}(\Omega,\R) \cap W^{(k+1)(\sobs-1),\sob'}(\T_h,\R)$.
  Then, under Assumptions \ref{ass:stress} and \ref{ass:sT},
  \begin{subequations}\label{eq:error.estimate}
    \begin{align}\label{eq:error.estimate:velocity}
      \| \bu u_h - \bI{h}{k} \b u \|_{\sob,h}
      &\lesssim 
      h^\frac{(k+1)(\sobs-1)}{\sob+1-\sobs}      
      \left(\sigma_\cst{sm}^{-1}\mathcal N_{\b f}^{2-\sobs}\mathcal N_{\stress,\b u,p}\right)^\frac{1}{\sob+1-\sobs},
      \\
      \label{eq:error.estimate:pressure}
      \| p_h - \proj{h}{k} p \|_{L^{\sob'}(\Omega,\R)}
      &\lesssim
        h^{(k+1)(\sobs-1)}\mathcal N_{\stress,\b u,p}
        + 
        h^{\frac{(k+1)(\sobs-1)^2}{\sob+1-\sobs}} \sigma_\cst{hc}\mathcal N_{\b f}^{|\sob-2|(\sobs-1)}
        \left(\sigma_\cst{sm}^{-1}\mathcal N_{\stress,\b u,p}\right)^\frac{\sobs-1}{\sob+1-\sobs},
    \end{align}
  \end{subequations}
  where we have set, for the sake of brevity,
  \[
  \begin{aligned}
    \mathcal N_{\stress,\b u,p}
    &\coloneq
    \sigma_\cst{hc}\left(\sigma_\cst{de}^\sob+|\b u|_{W^{1,\sob}(\Omega,\R^d)}^\sob\right)^\frac{\sob-\sobs}{\sob}|\b u|_{W^{k+2,\sob}(\T_h,\R^d)}^{\sobs-1}
    \\
    &\quad
    + |\stress(\cdot,\GRADs \b u)|_{W^{(k+1)(\sobs-1),\sob'}(\T_h,\M{d})}
    + |p|_{W^{(k+1)(\sobs-1),\sob'}(\T_h,\R)},
    \\
    \mathcal N_{\b f} &\coloneqq \sigma_\cst{de}+\left(\sigma_\cst{sm}^{-1}\| \b f \|_{L^{\sob'}(\Omega,\R^d)}\right)^\frac{1}{\sob-1}+\left(\sigma_\cst{de}^{2-\sobs}\sigma_\cst{sm}^{-1}\| \b f \|_{L^{\sob'}(\Omega,\R^d)}\right)^\frac{1}{\sob+1-\sobs}.
  \end{aligned}
  \]
\end{theorem}

\begin{proof}
  See Section \ref{sec:analysis:error.estimate}.
\end{proof}

\begin{remark}[Orders of convergence]\label{rem:ocv}
  From \eqref{eq:error.estimate}, neglecting higher-order terms, we infer asymptotic convergence rates of $\mathcal O_\cst{vel}^k \coloneqq \frac{(k+1)(\sobs-1)}{\sob+1-\sobs}$ for the velocity and $\mathcal O_\cst{pre}^k \coloneqq \frac{(k+1)(\sobs-1)^2}{\sob+1-\sobs}$ for the pressure, that is,
  \begin{equation}\label{eq:asymptotic.order}
    \mathcal O_\cst{vel}^k
    = \begin{cases}
      (k+1)(r-1) & \text{if $r<2$}, \\
      \tfrac{k+1}{r-1} & \text{if $r\ge 2$,}
    \end{cases}\quad \text{and} \quad 
    \mathcal O_\cst{pre}^k
      = \begin{cases}
        (k+1)(r-1)^2 & \text{if $r<2$}, \\
        \tfrac{k+1}{r-1} & \text{if $r\ge 2$.}
      \end{cases}
  \end{equation}
  Notice that, owing to the presence of higher-order terms in the right-hand sides of \eqref{eq:error.estimate}, higher convergence rates may be observed before attaining the asymptotic ones; see Section \ref{sec:num.res}.
  The asymptotic order of convergence for the velocity coincides with the one proved in \cite[Theorem 3.2]{Di-Pietro.Droniou:17*1} for HHO discretizations of scalar Leray--Lions problems.
  We refer to \cite{Di-Pietro.Droniou.ea:21} for recent improvements on these estimates depending on the degeneracy of the problem.
\end{remark}

%------------------------------------------------------------------------------%

\section{Numerical examples}\label{sec:num.res}

In this section, we evaluate the numerical performance of the HHO method on a complete panel of numerical test cases.
We focus on the $(\mu,0,1,\sob)$-Carreau--Yasuda law \eqref{eq:Carreau--Yasuda} (corresponding to the power-law model) with values of the exponent $\sob$ ranging from $1.25$ to $2.75$.  
Our implementation relies on the SpaFEDTe library (cf. \url{https://spafedte.github.io}). 

\subsection{Trigonometric solution}\label{sec:trigonometric}

We begin by considering a manufactured solution to problem \eqref{eq:stokes.continuous} in order to assess the convergence of the method.
We take $\Omega=(0,1)^2$ and exact velocity $\b u$ and pressure $p$ given by, respectively,
\[
  \b u(x_1,x_2) = \left(\sin\left(\tfrac{\pi}{2}x_1\right)\cos\left(\tfrac{\pi}{2}x_2\right),-\cos\left(\tfrac{\pi}{2}x_1\right)\sin\left(\tfrac{\pi}{2}x_2\right)\right),\quad 
  p(x_1,x_2) = \sin\left(\tfrac{\pi}{2} x_1\right)\sin\left(\tfrac{\pi}{2} x_2\right)-\tfrac{4}{\pi^2}.
  \]
The volumetric load $\b f$ and the Dirichlet boundary condition are inferred from the exact solution. Considering $\mu=1$ and $\sob\in\{1.5, 1.75, \dots, 2.75\}$, this solution matches the assumptions required in Theorem \ref{thm:error.estimate} for $k = 1$, except the case $r = 1.5$ for which $\stress(\cdot,\GRADs \b u) \notin W^{1,\sob'}(\Omega,\Ms{d})$.
We consider the HHO scheme for $k=1$ on three mesh families, namely Cartesian orthogonal, distorted triangular, and distorted Cartesian; see Figure \ref{fig:meshes}.
Overall, the results are in agreement with the theoretical predictions, and in some cases the expected asymptotic orders of convergence are exceeded.
Specifically, for $\sob \neq 2$, the convergence rates computed on the last refinement surpass in some cases the theoretical ones.
As noticed in Remark \ref{rem:ocv}, this suggests that the asymptotic order is still not attained.
A similar phenomenon has been observed on certain meshes for the $p$-Laplace problem; see \cite[Section 3.5.2]{Di-Pietro.Droniou:17*1} and \cite[Section 3.7]{Di-Pietro.Droniou.ea:18}.
  In some cases, we observe a better convergence for the velocity on distorted triangular meshes than on Cartesian meshes.
  This phenomenon possibly results from the combination of two factors:
  on one hand, the improved robustness of HHO methods with respect to elongated elements when compared to classical discretization methods;
  on the other hand, the fact that unstructured triangular meshes have more elements than Cartesian meshes for a given meshsize and lack privileged directions, which reduces mesh bias.
  Further investigation is postponed to a future work.

\begin{figure}\centering
  \includegraphics[scale=0.3]{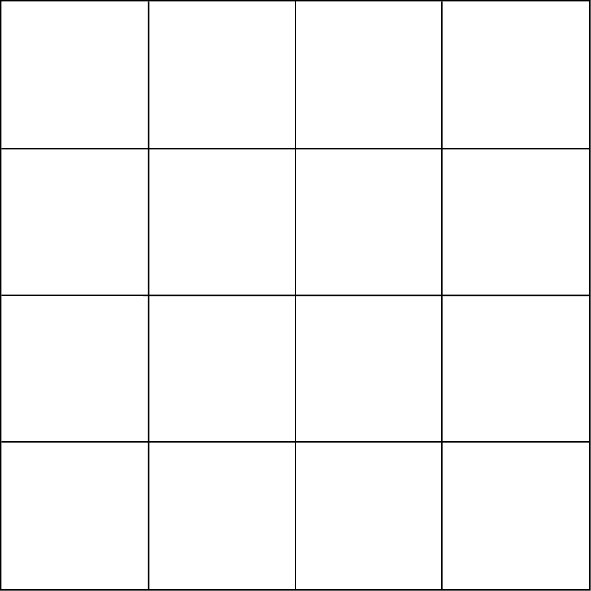} \qquad
  \includegraphics[scale=0.3]{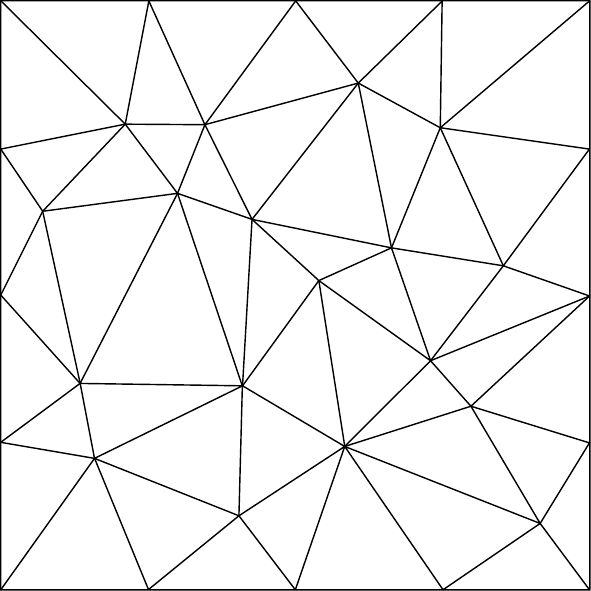} \qquad
  \includegraphics[scale=0.3]{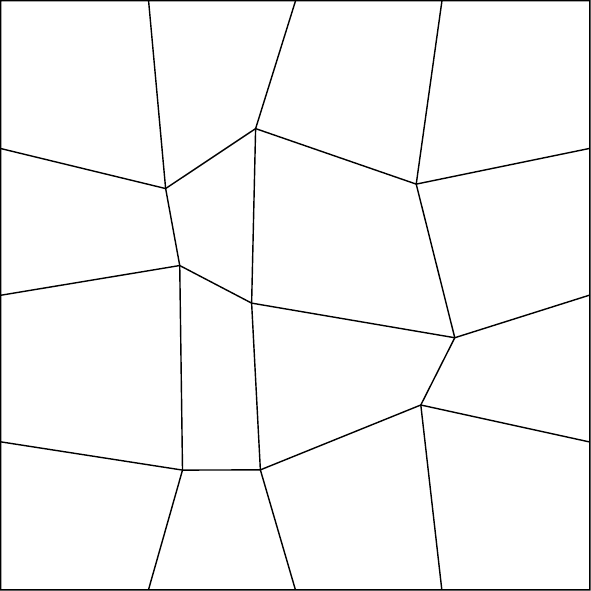}
  \caption{Coarsest Cartesian, distorted triangular, and distorted Cartesian meshes used in Section \ref{sec:num.res}.\label{fig:meshes}} 
\end{figure}

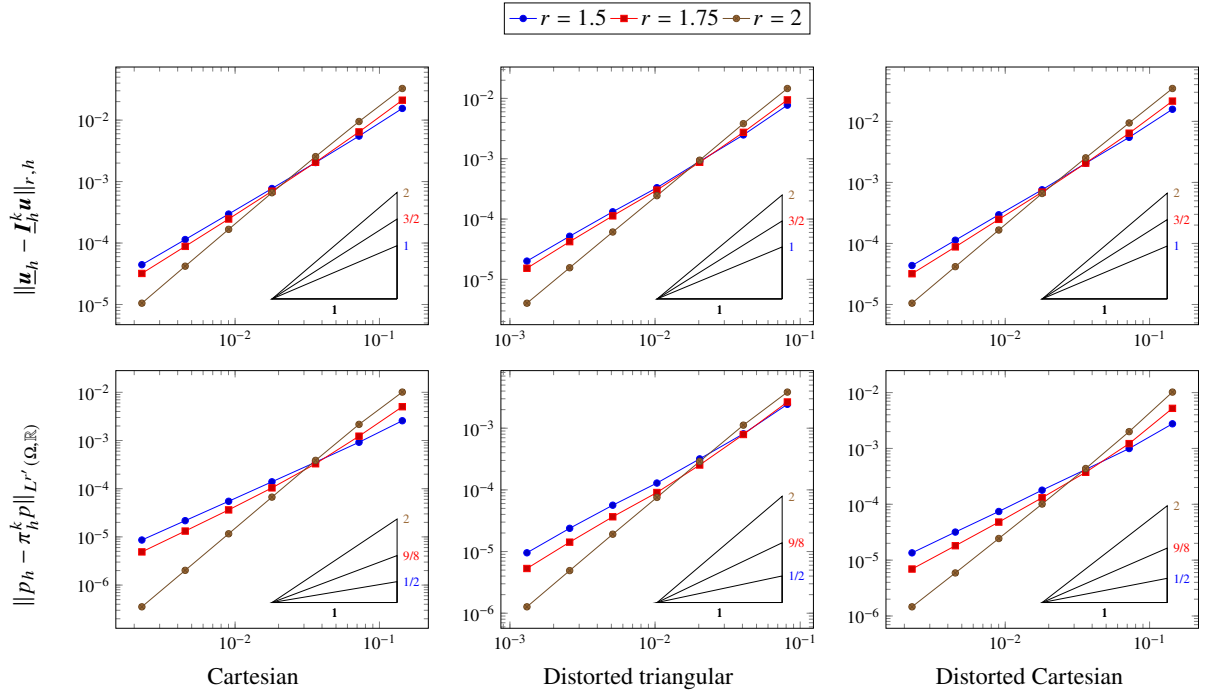
\begin{figure}
  \begin{center}
    %VITESSE
    \begin{minipage}[b]{0.04\columnwidth}
      \begin{tikzpicture}
      	\draw node[scale=0.8,rotate=90]{\hspace*{0.7cm} $\| \bu u_h - \bI{h}{k} \b u \|_{\sob,h}$};
      \end{tikzpicture}
    \end{minipage}
    \hfill
    %%triangular  error L2
    \begin{minipage}[b]{0.3\columnwidth}
      \begin{tikzpicture}[scale=0.6]
        \begin{loglogaxis}
          \addplot table[x=meshsize,y=err_u] {results/trigo/square_k1_r150.dat};
          \addplot table[x=meshsize,y=err_u] {results/trigo/square_k1_r175.dat};
          \addplot table[x=meshsize,y=err_u] {results/trigo/square_k1_r200.dat};
          \logLogSlopeTriangle{0.90}{0.4}{0.1}{1}{blue};
          \logLogSlopeTriangle{0.90}{0.4}{0.1}{3/2}{red};
          \logLogSlopeTriangle{0.90}{0.4}{0.1}{2}{darkbrown};
        \end{loglogaxis}
      \end{tikzpicture}
    \end{minipage}
    \hfill
    %%cartesian  error L2
    \begin{minipage}[b]{0.30\columnwidth}
      \begin{tikzpicture}[scale=0.6]
        \begin{loglogaxis}[
            legend style = { 
              at={(0.5,1.25)},
              anchor = north,
              legend columns=-1,
              font={\Large}
            }
          ]
          \addplot table[x=meshsize,y=err_u] {results/trigo/disttriangle_k1_r150.dat};
          \addplot table[x=meshsize,y=err_u] {results/trigo/disttriangle_k1_r175.dat};
          \addplot table[x=meshsize,y=err_u] {results/trigo/disttriangle_k1_r200.dat};
          \logLogSlopeTriangle{0.90}{0.4}{0.1}{1}{blue};
          \logLogSlopeTriangle{0.90}{0.4}{0.1}{3/2}{red};
          \logLogSlopeTriangle{0.90}{0.4}{0.1}{2}{darkbrown};
          \legend{$r=1.5$ , $r=1.75$ , $r=2$};
        \end{loglogaxis}
      \end{tikzpicture}
    \end{minipage}
    \hfill
    %%refined  error L2
    \begin{minipage}[b]{0.31\columnwidth}
      \begin{tikzpicture}[scale=0.6]
        \begin{loglogaxis}
          \addplot table[x=meshsize,y=err_u] {results/trigo/distsquare_k1_r150.dat};
          \addplot table[x=meshsize,y=err_u] {results/trigo/distsquare_k1_r175.dat};
          \addplot table[x=meshsize,y=err_u] {results/trigo/distsquare_k1_r200.dat};
          \logLogSlopeTriangle{0.90}{0.4}{0.1}{1}{blue};
          \logLogSlopeTriangle{0.90}{0.4}{0.1}{3/2}{red};
          \logLogSlopeTriangle{0.90}{0.4}{0.1}{2}{darkbrown};
        \end{loglogaxis}
      \end{tikzpicture}
    \end{minipage}
  \\
  \medskip
  	%PRESSION
    \begin{minipage}[b]{0.04\columnwidth}
      \begin{tikzpicture}
      	\draw node[scale=0.8,rotate=90]{\hspace*{1.3cm}$\| p_h - \proj{h}{k} p \|_{L^{\sob'}(\Omega,\R)}$};
      \end{tikzpicture}
    \end{minipage}
    \hfill
    \begin{minipage}[b]{0.3\columnwidth}
      \begin{tikzpicture}[scale=0.6]
        \begin{loglogaxis}
          \addplot table[x=meshsize,y=err_p] {results/trigo/square_k1_r150.dat};
          \addplot table[x=meshsize,y=err_p] {results/trigo/square_k1_r175.dat};
          \addplot table[x=meshsize,y=err_p] {results/trigo/square_k1_r200.dat};
          \logLogSlopeTriangle{0.90}{0.4}{0.1}{1/2}{blue};
          \logLogSlopeTriangle{0.90}{0.4}{0.1}{9/8}{red};
          \logLogSlopeTriangle{0.90}{0.4}{0.1}{2}{darkbrown};
        \end{loglogaxis}
      \end{tikzpicture}\vspace*{-0.1cm}
      \subcaption*{Cartesian}
    \end{minipage}
    \hfill
    \begin{minipage}[b]{0.3\columnwidth}
      \begin{tikzpicture}[scale=0.6]
        \begin{loglogaxis}
          \addplot table[x=meshsize,y=err_p] {results/trigo/disttriangle_k1_r150.dat};
          \addplot table[x=meshsize,y=err_p] {results/trigo/disttriangle_k1_r175.dat};
          \addplot table[x=meshsize,y=err_p] {results/trigo/disttriangle_k1_r200.dat};
          \logLogSlopeTriangle{0.90}{0.4}{0.1}{1/2}{blue};
          \logLogSlopeTriangle{0.90}{0.4}{0.1}{9/8}{red};
          \logLogSlopeTriangle{0.90}{0.4}{0.1}{2}{darkbrown};
        \end{loglogaxis}
      \end{tikzpicture}\vspace*{-0.1cm}
      \subcaption*{Distorted triangular}
    \end{minipage}
    \hfill
    \begin{minipage}[b]{0.31\columnwidth}
      \begin{tikzpicture}[scale=0.6]
        \begin{loglogaxis}
          \addplot table[x=meshsize,y=err_p] {results/trigo/distsquare_k1_r150.dat};
          \addplot table[x=meshsize,y=err_p] {results/trigo/distsquare_k1_r175.dat};
          \addplot table[x=meshsize,y=err_p] {results/trigo/distsquare_k1_r200.dat};
          \logLogSlopeTriangle{0.90}{0.4}{0.1}{1/2}{blue};
          \logLogSlopeTriangle{0.90}{0.4}{0.1}{9/8}{red};
          \logLogSlopeTriangle{0.90}{0.4}{0.1}{2}{darkbrown};
        \end{loglogaxis}
      \end{tikzpicture}\vspace*{-0.1cm}
      \subcaption*{Distorted Cartesian}
    \end{minipage}
  \end{center}
   \caption{Numerical results for the test case of Section \ref{sec:num.res}. The slopes indicate the order of convergence expected from Theorem \ref{thm:error.estimate}, i.e. $\mathcal O_\cst{vel}^1 = 2(r-1)$ and $\mathcal O_\cst{pre}^1  = 2(r-1)^2$ for $\sob \in \{1.5,1.75,2\}$.
   \label{tab:num.res.1}}
\end{figure}

\begin{figure}
  \begin{center}
    %VITESSE
    \begin{minipage}[b]{0.04\columnwidth}
      \begin{tikzpicture}
      	\draw node[scale=0.8,rotate=90]{\hspace*{0.7cm} $\| \bu u_h - \bI{h}{k} \b u \|_{\sob,h}$};
      \end{tikzpicture}
    \end{minipage}
    \hfill
    %%triangular  error L2
    \begin{minipage}[b]{0.3\columnwidth}
      \begin{tikzpicture}[scale=0.6]
        \begin{loglogaxis}
          \addplot table[x=meshsize,y=err_u] {results/trigo/square_k1_r225.dat};
          \addplot table[x=meshsize,y=err_u] {results/trigo/square_k1_r250.dat};
          \addplot table[x=meshsize,y=err_u] {results/trigo/square_k1_r275.dat};
          \logLogSlopeTriangle{0.90}{0.4}{0.1}{8/5}{blue};
          \logLogSlopeTriangle{0.90}{0.4}{0.1}{4/3}{red};
          \logLogSlopeTriangle{0.90}{0.4}{0.1}{8/7}{darkbrown};
        \end{loglogaxis}
      \end{tikzpicture}
    \end{minipage}
    \hfill
    %%cartesian  error L2
    \begin{minipage}[b]{0.30\columnwidth}
      \begin{tikzpicture}[scale=0.6]
        \begin{loglogaxis}[
            legend style = {
              at={(0.5,1.25)},
              anchor = north,
              legend columns=-1,
              font={\Large}
            }
          ]
          \addplot table[x=meshsize,y=err_u] {results/trigo/disttriangle_k1_r225.dat};
          \addplot table[x=meshsize,y=err_u] {results/trigo/disttriangle_k1_r250.dat};
          \addplot table[x=meshsize,y=err_u] {results/trigo/disttriangle_k1_r275.dat};
          \logLogSlopeTriangle{0.90}{0.4}{0.1}{8/5}{blue};
          \logLogSlopeTriangle{0.90}{0.4}{0.1}{4/3}{red};
          \logLogSlopeTriangle{0.90}{0.4}{0.1}{8/7}{darkbrown};
          \legend{$r=2.25$ , $r=2.5$ , $r=2.75$};
        \end{loglogaxis}
      \end{tikzpicture}
    \end{minipage}
    \hfill
    %%refined  error L2
    \begin{minipage}[b]{0.31\columnwidth}
      \begin{tikzpicture}[scale=0.6]
        \begin{loglogaxis}
          \addplot table[x=meshsize,y=err_u] {results/trigo/distsquare_k1_r225.dat};
          \addplot table[x=meshsize,y=err_u] {results/trigo/distsquare_k1_r250.dat};
          \addplot table[x=meshsize,y=err_u] {results/trigo/distsquare_k1_r275.dat};
          \logLogSlopeTriangle{0.90}{0.4}{0.1}{8/5}{blue};
          \logLogSlopeTriangle{0.90}{0.4}{0.1}{4/3}{red};
          \logLogSlopeTriangle{0.90}{0.4}{0.1}{8/7}{darkbrown};
        \end{loglogaxis}
      \end{tikzpicture}
    \end{minipage}
  \\
  \medskip
  	%PRESSION
    \begin{minipage}[b]{0.04\columnwidth}
      \begin{tikzpicture}
      	\draw node[scale=0.8,rotate=90]{\hspace*{1.3cm}$\| p_h - \proj{h}{k} p \|_{L^{\sob'}(\Omega,\R)}$};
      \end{tikzpicture}
    \end{minipage}
    \hfill
    \begin{minipage}[b]{0.3\columnwidth}
      \begin{tikzpicture}[scale=0.6]
        \begin{loglogaxis}
          \addplot table[x=meshsize,y=err_p] {results/trigo/square_k1_r225.dat};
          \addplot table[x=meshsize,y=err_p] {results/trigo/square_k1_r250.dat};
          \addplot table[x=meshsize,y=err_p] {results/trigo/square_k1_r275.dat};
          \logLogSlopeTriangle{0.90}{0.4}{0.1}{8/5}{blue};
          \logLogSlopeTriangle{0.90}{0.4}{0.1}{4/3}{red};
          \logLogSlopeTriangle{0.90}{0.4}{0.1}{8/7}{darkbrown};
        \end{loglogaxis}
      \end{tikzpicture}\vspace*{-0.1cm}
      \subcaption*{Cartesian}
    \end{minipage}
    \hfill
    \begin{minipage}[b]{0.3\columnwidth}
      \begin{tikzpicture}[scale=0.6]
        \begin{loglogaxis}
          \addplot table[x=meshsize,y=err_p] {results/trigo/disttriangle_k1_r225.dat};
          \addplot table[x=meshsize,y=err_p] {results/trigo/disttriangle_k1_r250.dat};
          \addplot table[x=meshsize,y=err_p] {results/trigo/disttriangle_k1_r275.dat};
          \logLogSlopeTriangle{0.90}{0.4}{0.1}{8/5}{blue};
          \logLogSlopeTriangle{0.90}{0.4}{0.1}{4/3}{red};
          \logLogSlopeTriangle{0.90}{0.4}{0.1}{8/7}{darkbrown};
        \end{loglogaxis}
      \end{tikzpicture}\vspace*{-0.1cm}
      \subcaption*{Distorted triangular}
    \end{minipage}
    \hfill
    \begin{minipage}[b]{0.31\columnwidth}
      \begin{tikzpicture}[scale=0.6]
        \begin{loglogaxis}
          \addplot table[x=meshsize,y=err_p] {results/trigo/distsquare_k1_r225.dat};
          \addplot table[x=meshsize,y=err_p] {results/trigo/distsquare_k1_r250.dat};
          \addplot table[x=meshsize,y=err_p] {results/trigo/distsquare_k1_r275.dat};
          \logLogSlopeTriangle{0.90}{0.4}{0.1}{8/5}{blue};
          \logLogSlopeTriangle{0.90}{0.4}{0.1}{4/3}{red};
          \logLogSlopeTriangle{0.90}{0.4}{0.1}{8/7}{darkbrown};
        \end{loglogaxis}
      \end{tikzpicture}\vspace*{-0.1cm}
      \subcaption*{Distorted Cartesian}
    \end{minipage}
  \end{center}
   \caption{Numerical results for the test case of Section \ref{sec:trigonometric}. The slopes indicate the order of convergence expected from Theorem \ref{thm:error.estimate}, i.e. $\mathcal O_\cst{vel}^1 = \mathcal O_\cst{pre}^1 = \frac{2}{r-1}$ for $\sob \in \{2.25,2.5,2.75\}$.
   \label{tab:num.res.2}}
\end{figure}
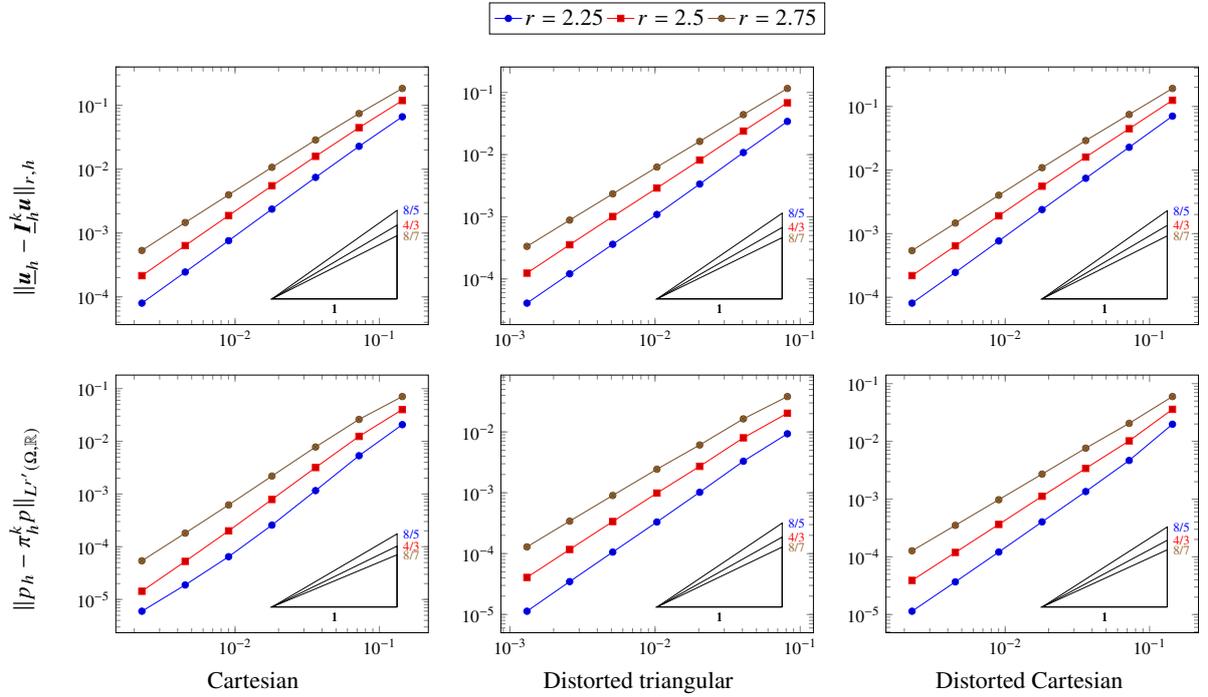

\subsection{Lid-driven cavity flow}\label{sec:cavity}

We next consider the lid-driven cavity flow, a well-known problem in fluid mechanics.
The domain is the unit square $\Omega=(0,1)^2$, and we enforce a unit tangential velocity $\boldsymbol{u}=(1,0)$ on the top edge (of equation $x_2=1$) and wall boundary conditions on the other edges.
This boundary condition is incompatible with the formulation \eqref{eq:stokes.weak}, even generalized to non-homogeneous boundary conditions, since $\b u \notin W^{1,\sob}(\Omega,\R^d)$.
However, this is a very classical test that demonstrates the quality of the method. 
We consider a low Reynolds number $\cst{Re} \coloneq \frac{2}{\mu} = 1$. 
For $r\in\{1.25,2,2.75\}$, we solve the discrete problem on Cartesian and distorted triangular meshes (cf. Figure \ref{fig:meshes}) of approximate size $128\times 128$ for $k=1$, and $16\times 16$ for $k=5$.
This choice is meant to compare the low-order version of the method on a fine mesh with the high-order version on a very coarse one.
The corresponding total number of degrees of freedom is: $130048$ for the fine Cartesian mesh with $k=1$; $5760$ for the coarse Cartesian mesh with $k=5$; $298676$ for the fine triangular mesh with $k=1$; and $14196$ for the coarse triangular mesh with $k=5$.
In the left column of Figure \ref{fig:cavity} we display the velocity magnitude, while in the right column we plot the horizontal component $u_1$ of the velocity along the vertical centreline $x_1=\frac12$ (resp., vertical component $u_2$ along the horizontal centreline $x_2=\frac12$).
The lines corresponding to $k=1$ on the fine mesh and to $k=5$ on the coarse mesh are perfectly superimposed, regardless of the mesh family and of the value of $r$.
This shows that, despite the lack of regularity of the exact solution, high-order versions of the scheme on very coarse meshes deliver similar results as low-order versions on very fine grids.
Furthermore, we observe significant differences in the behavior of the flow according to $r$, coherent with the expected physical behavior.
In particular, the viscous effects increase with $r$, as reflected by the size of the central vortex.

%r = 1.25
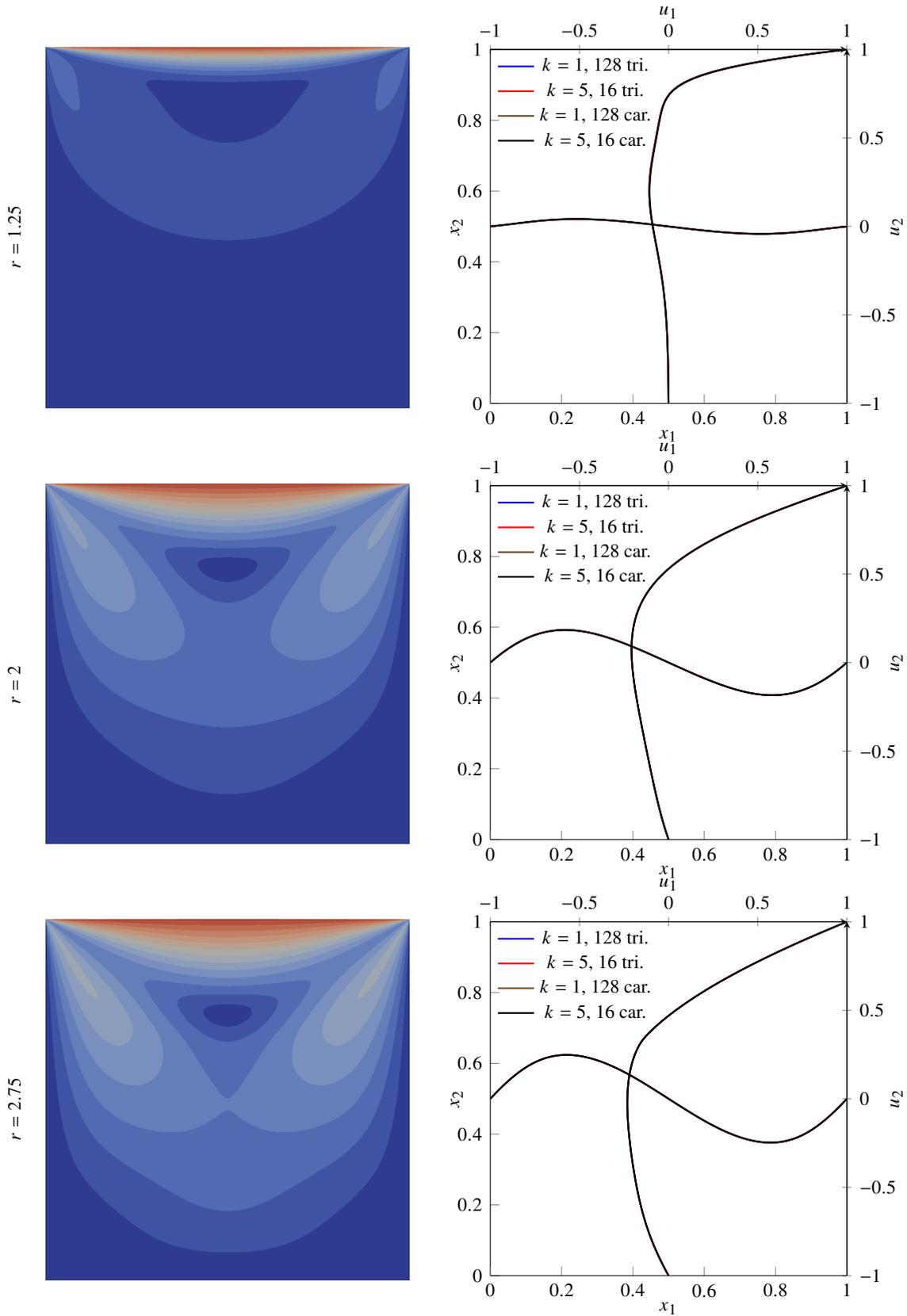
\begin{figure}
 	\begin{center}
 	  \begin{minipage}[b]{0.04\columnwidth}
      \begin{tikzpicture}
      	\draw node[scale=0.8,rotate=90]{\hspace*{2.7cm} $r=1.25$};
      \end{tikzpicture}
    \end{minipage}
 	  \begin{minipage}[b]{0.4\columnwidth}
   		\includegraphics[scale=0.28]{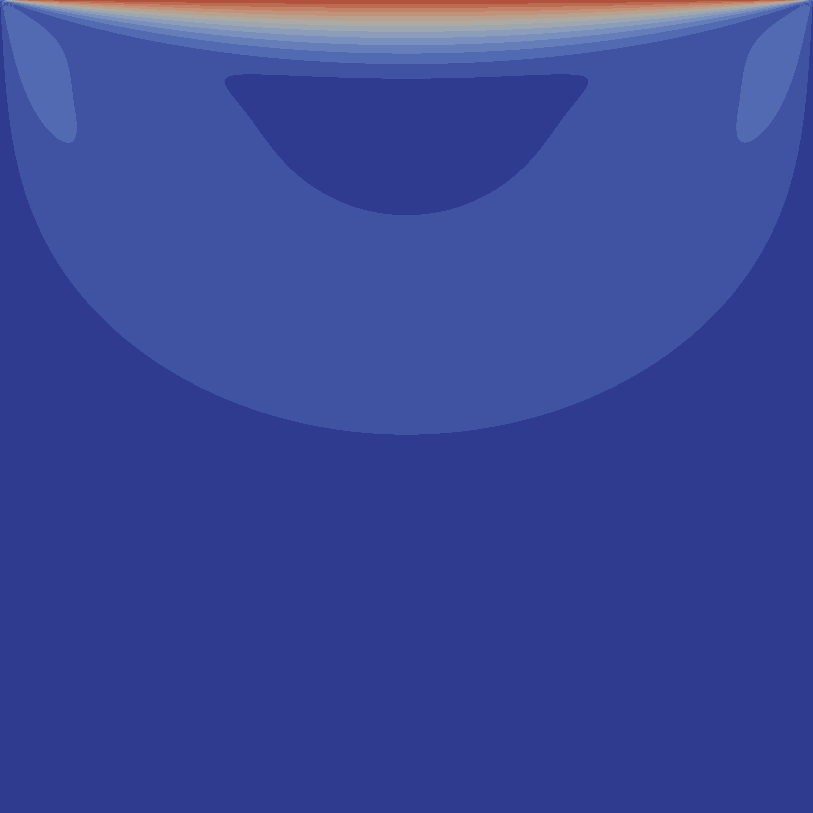}
   	\end{minipage}    
   	\begin{minipage}[b]{0.5\columnwidth}
   	  ~\vspace{-0.7cm}	
   		\begin{tikzpicture}[font=\footnotesize, %
          spy using outlines={magnification=4, size=3cm, connect spies, fill=none} %
        ]
        \begin{axis}[height=7.5cm, width=7.5cm,
            xmin=-1, xmax=1, ymin=0, ymax=1,
            xlabel={$u_1$}, ylabel={$x_2$},
            legend style = { at={(0,1)}, anchor=north west, draw=none, fill=none }, 
            axis x line=top, ytick pos=bottom,
            xlabel style={yshift=-5pt},
            ylabel style={yshift=-10pt}]
          
          \addplot +[mark=none, thick] table [x=a, y=Points:1] {results/lid_driven_cavity/disttrimesh/re1_r125_k1_128_u1.txt};
          \addplot +[mark=none, thick] table [x=a, y=Points:1] {results/lid_driven_cavity/disttrimesh/re1_r125_k5_16_u1.txt};
          \addplot +[mark=none, thick] table [x=a, y=Points:1] {results/lid_driven_cavity/squaremesh/re1_r125_k1_128_u1.txt};
          \addplot +[mark=none, thick] table [x=a, y=Points:1] {results/lid_driven_cavity/squaremesh/re1_r125_k5_16_u1.txt};

          \legend{%
            {$k=1$, $128$ tri.},
            {$k=5$, $16$ tri.},
            {$k=1$, $128$ car.},
            {$k=5$, $16$ car.}
          }
        \end{axis}
        \begin{axis}[height=7.5cm, width=7.5cm, %
            xmin=0, xmax=1, ymin=-1, ymax=1,%
            xlabel={$x_1$}, ylabel={$u_2$}, %
            axis y line=right, xtick pos=left,
            xlabel style={yshift=5pt},
            ylabel style={yshift=10pt}]
          \addplot +[mark=none, thick] table [x=Points:0, y=b] {results/lid_driven_cavity/disttrimesh/re1_r125_k1_128_u2.txt};
          \addplot +[mark=none, thick] table [x=Points:0, y=b] {results/lid_driven_cavity/disttrimesh/re1_r125_k5_16_u2.txt};
          \addplot +[mark=none, thick] table [x=Points:0, y=b] {results/lid_driven_cavity/squaremesh/re1_r125_k1_128_u2.txt};
          \addplot +[mark=none, thick] table [x=Points:0, y=b] {results/lid_driven_cavity/squaremesh/re1_r125_k5_16_u2.txt};
        \end{axis}
        
        %        \begin{scope}
        %          \spy on (1.75,3.85) in node [fill=none, anchor=south east] at (-0.75,4.25);
        %          \spy on (3.9,6.5) in node [fill=none, anchor=south west] at (8.35,4.25); 
        %          \spy on (5.7,3.55) in node [fill=none, anchor=south west] at (8.35,0.25);
        %          \spy on (3.65,1.5) in node [fill=none, anchor=south east] at (-0.75,0.25);
        %        \end{scope}
      \end{tikzpicture}
	  \end{minipage}
	\end{center}
	%r=2
	\begin{center}	
 	  \begin{minipage}[b]{0.04\columnwidth}
      \begin{tikzpicture}
      	\draw node[scale=0.8,rotate=90]{\hspace*{2.7cm} $r=2$};
      \end{tikzpicture}
    \end{minipage}
 	  \begin{minipage}[b]{0.4\columnwidth}
   		\includegraphics[scale=0.28]{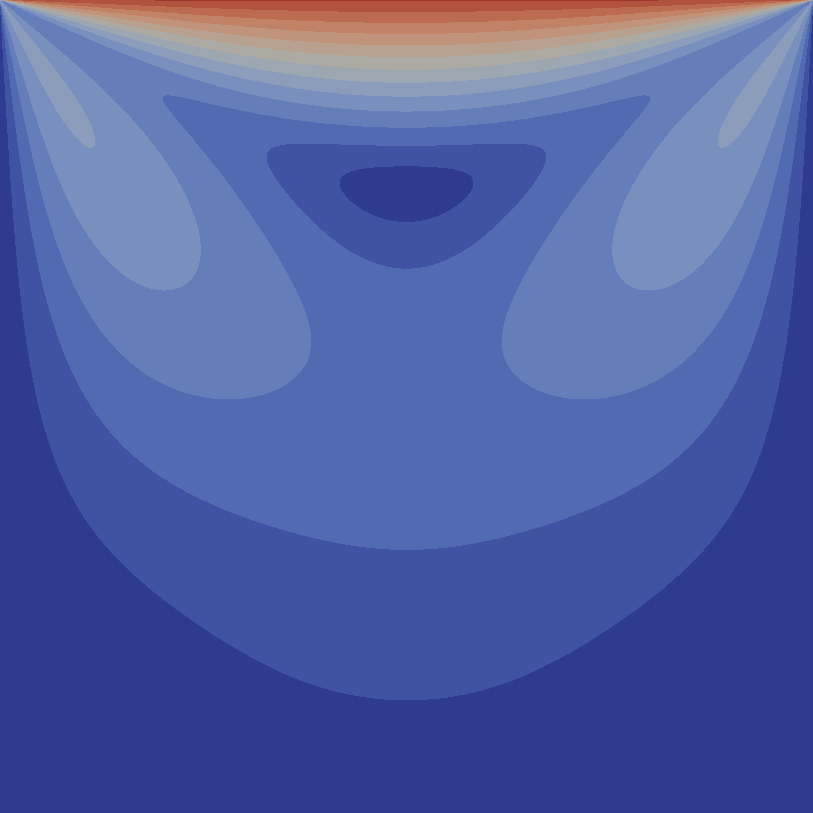}
   	\end{minipage}
   	\begin{minipage}[b]{0.5\columnwidth}
   	  ~\vspace{-0.7cm}	
      \begin{tikzpicture}[font=\footnotesize, %
          spy using outlines={magnification=4, size=3cm, connect spies, fill=none} %
        ]
        \begin{axis}[height=7.5cm, width=7.5cm,
            xmin=-1, xmax=1, ymin=0, ymax=1,
            xlabel={$u_1$}, ylabel={$x_2$},
            legend style = { at={(0,1)}, anchor=north west, draw=none, fill=none }, 
            axis x line=top, ytick pos=bottom,
            xlabel style={yshift=-5pt},
            ylabel style={yshift=-10pt}]
          
          \addplot +[mark=none, thick] table [x=a, y=Points:1] {results/lid_driven_cavity/disttrimesh/re1_r200_k1_128_u1.txt};
          \addplot +[mark=none, thick] table [x=a, y=Points:1] {results/lid_driven_cavity/disttrimesh/re1_r200_k5_16_u1.txt};
          \addplot +[mark=none, thick] table [x=a, y=Points:1] {results/lid_driven_cavity/squaremesh/re1_r200_k1_128_u1.txt};
          \addplot +[mark=none, thick] table [x=a, y=Points:1] {results/lid_driven_cavity/squaremesh/re1_r200_k5_16_u1.txt};

          \legend{%
            {$k=1$, $128$ tri.},
            {$k=5$, $16$ tri.},
            {$k=1$, $128$ car.},
            {$k=5$, $16$ car.}
          }
        \end{axis}
        \begin{axis}[height=7.5cm, width=7.5cm, %
            xmin=0, xmax=1, ymin=-1, ymax=1,%
            xlabel={$x_1$}, ylabel={$u_2$}, %
            axis y line=right, xtick pos=left,
            xlabel style={yshift=5pt},
            ylabel style={yshift=10pt}]
          \addplot +[mark=none, thick] table [x=Points:0, y=b] {results/lid_driven_cavity/disttrimesh/re1_r200_k1_128_u2.txt};
          \addplot +[mark=none, thick] table [x=Points:0, y=b] {results/lid_driven_cavity/disttrimesh/re1_r200_k5_16_u2.txt};
          \addplot +[mark=none, thick] table [x=Points:0, y=b] {results/lid_driven_cavity/squaremesh/re1_r200_k1_128_u2.txt};
          \addplot +[mark=none, thick] table [x=Points:0, y=b] {results/lid_driven_cavity/squaremesh/re1_r200_k5_16_u2.txt};
        \end{axis}

        %        \begin{scope}
        %          \spy on (1.5,4.37) in node [fill=none, anchor=south east] at (-0.75,4.25);
        %          \spy on (3.4,5.3) in node [fill=none, anchor=south west] at (8.35,4.25); 
        %          \spy on (6,3.05) in node [fill=none, anchor=south west] at (8.35,0.25);
        %          \spy on (3.4,1.2) in node [fill=none, anchor=south east] at (-0.75,0.25);
        %        \end{scope}
      \end{tikzpicture}
    \end{minipage}
  \end{center}
  %r=2.75
	\begin{center}	
 	  \begin{minipage}[b]{0.04\columnwidth}
      \begin{tikzpicture}
      	\draw node[scale=0.8,rotate=90]{\hspace*{2.7cm} $r=2.75$};
      \end{tikzpicture}
    \end{minipage}
 	  \begin{minipage}[b]{0.4\columnwidth}
   		\includegraphics[scale=0.28]{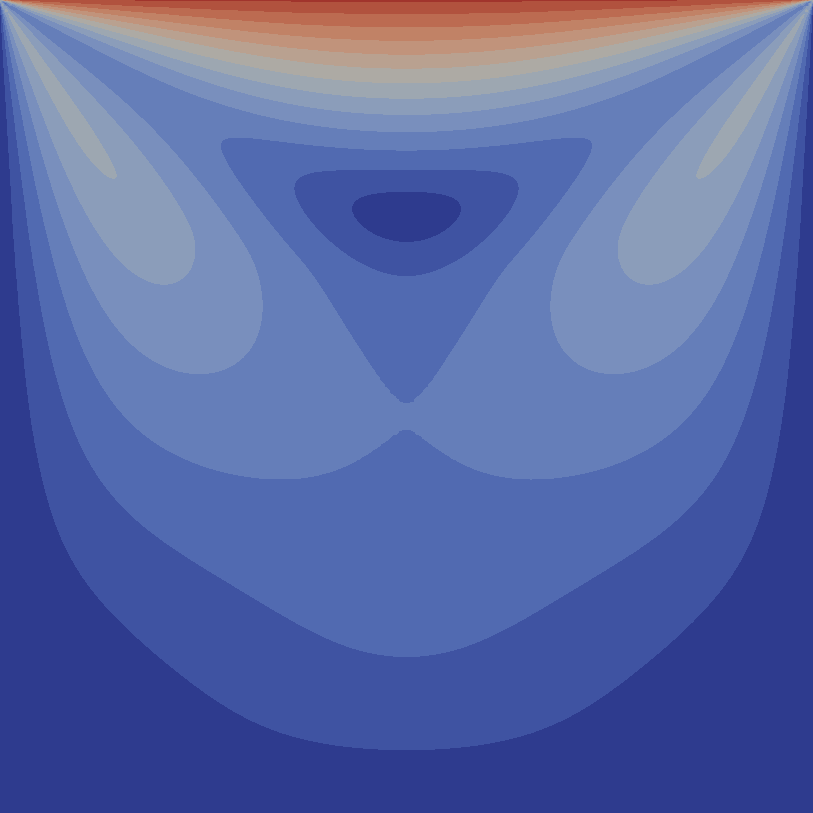}
   	\end{minipage}
   	\begin{minipage}[b]{0.5\columnwidth}
   	  ~\vspace{-0.7cm}	
   	  \begin{tikzpicture}[font=\footnotesize, %
          spy using outlines={magnification=4, size=3cm, connect spies, fill=none} %
        ]
        \begin{axis}[height=7.5cm, width=7.5cm,
            xmin=-1, xmax=1, ymin=0, ymax=1,
            xlabel={$u_1$}, ylabel={$x_2$},
            legend style = { at={(0,1)}, anchor=north west, draw=none, fill=none }, 
            axis x line=top, ytick pos=bottom,
            xlabel style={yshift=-5pt},
            ylabel style={yshift=-10pt}]
          
          \addplot +[mark=none, thick] table [x=a, y=Points:1] {results/lid_driven_cavity/disttrimesh/re1_r275_k1_128_u1.txt};
          \addplot +[mark=none, thick] table [x=a, y=Points:1] {results/lid_driven_cavity/disttrimesh/re1_r275_k5_16_u1.txt};
          \addplot +[mark=none, thick] table [x=a, y=Points:1] {results/lid_driven_cavity/squaremesh/re1_r275_k1_128_u1.txt};
          \addplot +[mark=none, thick] table [x=a, y=Points:1] {results/lid_driven_cavity/squaremesh/re1_r275_k5_16_u1.txt};

          \legend{%
            {$k=1$, $128$ tri.},
            {$k=5$, $16$ tri.},
            {$k=1$, $128$ car.},
            {$k=5$, $16$ car.}
          }
        \end{axis}
        \begin{axis}[height=7.5cm, width=7.5cm, %
            xmin=0, xmax=1, ymin=-1, ymax=1,%
            xlabel={$x_1$}, ylabel={$u_2$}, %
            axis y line=right, xtick pos=left,
            xlabel style={yshift=5pt},
            ylabel style={yshift=10pt}]
          \addplot +[mark=none, thick] table [x=Points:0, y=b] {results/lid_driven_cavity/disttrimesh/re1_r275_k1_128_u2.txt};
          \addplot +[mark=none, thick] table [x=Points:0, y=b] {results/lid_driven_cavity/disttrimesh/re1_r275_k5_16_u2.txt};
          \addplot +[mark=none, thick] table [x=Points:0, y=b] {results/lid_driven_cavity/squaremesh/re1_r275_k1_128_u2.txt};
          \addplot +[mark=none, thick] table [x=Points:0, y=b] {results/lid_driven_cavity/squaremesh/re1_r275_k5_16_u2.txt};
        \end{axis}
        
        %        \begin{scope}
        %          \spy on (1.6,4.6) in node [fill=none, anchor=south east] at (-0.75,4.25);
        %          \spy on (3.15,4.95) in node [fill=none, anchor=south west] at (8.35,4.25); 
        %          \spy on (5.85,2.8) in node [fill=none, anchor=south west] at (8.35,0.25);
        %          \spy on (3.2,1.2) in node [fill=none, anchor=south east] at (-0.75,0.25);
        %        \end{scope}
      \end{tikzpicture}
    \end{minipage}
 	\end{center}
  
  \caption{Numerical results for the test case of Section \ref{sec:cavity}: lid-driven cavity flow. \emph{Left:} velocity magnitude contours ($15$ equispaced values in the range $[0,1]$). Computations on a Cartesian mesh of size $128\times 128$ with $k=5$.
    \emph{Right:} horizontal component $u_1$ of the velocity along the vertical centreline $x_1=\frac12$ and vertical component $u_2$ of the velocity along the horizontal centreline $x_2=\frac12$. \label{fig:cavity}
  }
  \vspace{1cm}
\end{figure}

%------------------------------------------------------------------------------------%

\section{Discrete Korn inequality}\label{sec:discrete.korn.inequality}

We prove in this section a discrete counterpart of the following Korn inequality (see \cite[Theorem 1]{Geymonat.Suquet:86}) that will be needed in the analysis: There is $C_\cst{K} > 0$ depending only on $\Omega$ and $r$ such that for all $\b v \in\b U$,
\begin{equation}\label{eq:Korn}
\|\b v\|_{W^{1,\sob}(\Omega,\R^d)} \le C_\cst{K} \|\GRADs \b v\|_{L^\sob(\Omega,\M{d})}.
\end{equation}
We start by recalling the following preliminary result concerning the node-averaging interpolator (sometimes called Oswald interpolator).
Let $\mathfrak T_h$ be a matching simplicial submesh of $\mathcal M_h$ in the sense of \cite[Definition 1.8]{Di-Pietro.Droniou:20}.
The node-averaging operator $\Iav{h}{k} : \Poly^k(\T_h,\R^d) \to \Poly^k(\mathfrak T_h,\R^d) \cap W^{1,\sob}(\Omega,\R^d)$ is such that, for all $\b v_h \in \Poly^k(\T_h,\R^d)$ and all Lagrange node $V$ of $\mathfrak T_h$, denoting by $\mathfrak T_V$ the set of simplices sharing $V$,
\[
(\Iav{h}{k}\b v_h)(V) \coloneq
\begin{cases}
  \frac{1}{\rm{card}(\mathfrak T_V)}\sum_{\b\tau\in\mathfrak T_V} \b v_h\res{\b\tau}(V) & \text{if }\ V \in \Omega,
  \\ \b 0 & \text{if }\ V \in \partial\Omega.
\end{cases}
\]
For all $F \in \Fi$, denote by $T_1,T_2\in\T_h$ the elements sharing $F$, taken in an arbitrary but fixed order.
We define the jump operator such that, for any function $\b v\in W^{1,1}(\T_h,\R^d)$, $[\b v]_F \coloneq (\b v\res{T_1})\res{F}-(\b v\res{T_2})\res{F}$.
This definition is extended to boundary faces $F\in\Fb$ by setting $[\b v]_F \coloneq \b v\res{F}$.
\begin{proposition}[Boundedness of the node-averaging operator]
  For all $\b v_h \in \Poly^k(\T_h,\R^d)$, it holds
  \begin{equation}\label{eq:Iav:bound}
    |\b v_h-\Iav{h}{k}\b v_h|_{W^{1,\sob}(\T_h,\R^d)}^\sob \lesssim \sum_{F\in\F_h} h_F^{1-\sob} \|[\b v_h]_F\|^\sob_{L^\sob(F,\R^d)}.
  \end{equation}
\end{proposition}
\begin{proof}
  Combining \cite[Eq. (4.13)]{Di-Pietro.Droniou:20} (which corresponds to \eqref{eq:Iav:bound} for $\sob=2$) with the local Lebesgue embeddings of \cite[Lemma 1.25]{Di-Pietro.Droniou:20} (see also \cite[Lemma 5.1]{Di-Pietro.Droniou:17}) gives, for any $T\in\T_h$,
  \begin{equation}\label{eq:Iav:bound:0}
    \|\b v_h-\Iav{h}{k}\b v_h\|_{L^\sob(T,\R^d)}^\sob \lesssim \sum_{F\in\F_{\mathcal V,T}}h_F \|[\b v_h]_F\|^\sob_{L^\sob(F,\R^d)},
  \end{equation}
  where $\F_{\mathcal V,T}$ collects the faces whose closure has non-empty intersection with $\overline{T}$.
  Using the local inverse inequality of \cite[Lemma 1.28]{Di-Pietro.Droniou:20} (see also \cite[Eq. (A.1)]{Di-Pietro.Droniou:17}), we can write
  \[
  \begin{aligned}
    |\b v_h-\Iav{h}{k}\b v_h|_{W^{1,\sob}(\T_h,\R^d)}^\sob
    &\lesssim \sum_{T\in\T_h} h_T^{-\sob}\|\b v_h-\Iav{h}{k}\b v_h\|_{L^\sob(T,\R^d)}^\sob
    \\
    &\lesssim \sum_{T\in\T_h}\sum_{F\in\F_{\mathcal V,T}} h_F^{1-\sob} \|[\b v_h]_F\|^\sob_{L^\sob(F,\R^d)}
    \\
    &\lesssim \sum_{F\in\F_h} \sum_{T\in\T_{\mathcal V,F}} h_F^{1-\sob} \|[\b v_h]_F\|^\sob_{L^\sob(F,\R^d)}
    \\
    &\le\max_{F\in\F_h}\card(\T_{\mathcal V,F}) \sum_{F\in\F_h} h_F^{1-\sob} \|[\b v_h]_F\|^\sob_{L^\sob(F,\R^d)},
  \end{aligned}
  \]
  where we have used the fact that $h_T^{-\sob} \le h_F^{-\sob}$ along with inequality \eqref{eq:Iav:bound:0} to pass to the second line,
  and we have exchanged the sums after setting $\T_{\mathcal V,F} \coloneqq \big\{T \in \T_h : \overline F \cap \overline T \neq \emptyset\big\}$ for all $F \in \F_h$ to pass to the third line.
  Observing that $\max_{F\in\F_h}\card(\T_{\mathcal V,F}) \lesssim 1$ (since, for any $F\in\F_h$, $\card(\T_{\mathcal V,F})$ is bounded by the left-hand side of \cite[Eq. (4.23)]{Di-Pietro.Droniou:20} written for any $T\in\T_h$ to which $F$ belongs), \eqref{eq:Iav:bound} follows.
\end{proof}

The following inequalities between sums of powers will be often used in what follows without necessarily recalling this fact explicitly each time.
Let an integer $n\ge 1$ and a real number $m \in (0,\infty)$ be given. Then, for all $a_1,\ldots,a_n \in (0,\infty) $, we have
\begin{equation}\label{eq:sum-power}
  n^{-(m-1)^\ominus}\sum_{i=1}^n a_i^m \le \left(\sum_{i=1}^n a_i\right)^m  \le n^{(m-1)^\oplus}\sum_{i=1}^n a_i^m.
\end{equation}
If $m=1$, then \eqref{eq:sum-power} holds with the equal sign.
If $m < 1$, \cite[Eqs. (5) and (3)]{Ursell:59} with $\alpha = 1$ and $\beta = m$ give
$
n^{m-1}\sum_{i=1}^n a_i^m \le \left(\sum_{i=1}^n a_i\right)^m  \le \sum_{i=1}^n a_i^m.
$
If, on the other hand, $m > 1$, \cite[Eqs. (3) and (5)]{Ursell:59} with $\alpha = m$ and $\beta = 1$ give
$
\sum_{i=1}^n a_i^m \le \left(\sum_{i=1}^n a_i\right)^m  \le n^{m-1}\sum_{i=1}^n a_i^m.
$
Gathering the above cases yields \eqref{eq:sum-power}.

\begin{lemma}(Discrete Korn inequality)\label{lem:discrete.korn.inequality} 
We have, for all $\bu v_h  \in \bdU{h,0}{k}$, recalling the notation \eqref{eq:vh},
\begin{equation}\label{eq:discrete.Korn}
  \|  \b v_h \|_{L^\sob(\Omega,\R^d)}^\sob+|\b v_h|_{W^{1,\sob}(\T_h,\R^d)}^\sob \lesssim \|   \bu v_h \|_{\sob,h}^\sob.
\end{equation}
\end{lemma}

\begin{proof}
  Let $\bu v_h \in \bdU{h,0}{k}$. Using a triangle inequality followed by \eqref{eq:sum-power}, we can write
  \[
  \begin{aligned}
    |\b v_h|_{W^{1,\sob}(\T_h,\R^d)}^\sob
    &\lesssim |\Iav{h}{k} \b v_h|_{W^{1,\sob}(\T_h,\R^d)}^\sob+|\b v_h-\Iav{h}{k}\b v_h|_{W^{1,\sob}(\T_h,\R^d)}^\sob \\
    &\lesssim \|\GRADs (\Iav{h}{k} \b v_h) \|_{L^\sob(\Omega,\M{d})}^\sob+|\b v_h-\Iav{h}{k}\b v_h|_{W^{1,\sob}(\T_h,\R^d)}^\sob\\
    &\lesssim \|\brkGRADs \b v_h \|_{L^\sob(\Omega,\M{d})}^\sob+|\b v_h-\Iav{h}{k}\b v_h|_{W^{1,\sob}(\T_h,\R^d)}^\sob\\
    &\lesssim \|\brkGRADs \b v_h \|_{L^\sob(\Omega,\M{d})}^\sob+\sum_{F\in\F_h} h_F^{1-\sob} \|[\b v_h]_F\|^\sob_{L^\sob(F,\R^d)},
  \end{aligned}
  \]
  where we have used the continuous Korn inequality \eqref{eq:Korn} to pass to the second line,
  we have inserted $\pm\brkGRADs\b v_h$ into the first norm and used a triangle inequality followed by \eqref{eq:sum-power} to pass to the third line,
  and we have invoked the bound \eqref{eq:Iav:bound} to conclude.
  Observing that, for any $F\in\F_h$, $|[\b v_h]_F| \leq \sum_{T\in\T_F}|\b v_F-\b v_T|$ by a triangle inequality, and using \eqref{eq:sum-power}, we can continue writing
  \[
  |\b v_h|_{W^{1,\sob}(\T_h,\R^d)}^\sob
  \lesssim \|\brkGRADs \b v_h \|_{L^\sob(\Omega,\M{d})}^\sob+\sum_{F\in\F_h}\sum_{T\in\T_F}h_F^{1-\sob} \| \b v_F - \b v_T\|^\sob_{L^\sob(F,\R^d)}
  = \|\bu v_h\|_{\sob,h}^\sob,    
  \]
  where we have exchanged the sums over faces and elements and recalled definition \eqref{eq:norm.epsilon.r.h} to conclude. This proves the bound for the second term in the left-hand side of \eqref{eq:discrete.Korn}.
  Combining this result with the global discrete Sobolev embeddings of \cite[Proposition 5.4]{Di-Pietro.Droniou:17} yields the bound for the first term in \eqref{eq:discrete.Korn}.
\end{proof}

%------------------------------------------------------------------------------------%

\section{Well-posedness and convergence analysis}\label{sec:analysis}

In this section, after studying the stabilization function $\cst{s}_h$, we prove the main results stated in Section \ref{sec:discrete.problem.main.results}.

\subsection{Properties of the stabilization function}

\begin{lemma}[Consistency of $\cst{s}_T$]\label{lem:sT:consist}
  For any $T\in\T_h$ and any $\cst{s}_T$ satisfying Assumption \ref{ass:sT}, it holds, for all $\b w \in W^{k+2,\sob}(T,\R^d)$ and all $\bu v_T \in \bdU{T}{k}$,
  \begin{equation}\label{eq:sT:consist}
    |\cst{s}_T(\bI{T}{k} \b w,\bu v_T)| \lesssim h_T^{(k+1)(\sobs-1)}| \b w |_{W^{1,\sob}(T,\R^d)}^{\sob-\sobs}|\b w|_{W^{k+2,\sob}(T,\R^d)}^{\sobs-1}\| \bu v_T \|_{\sob,T},    
  \end{equation}
  where the hidden constant is independent of $h$, $T$, and $\b w$.
\end{lemma}

\begin{proof}
  The proof adapts the arguments of \cite[Propositon 2.14]{Di-Pietro.Droniou:20}.
  Using the polynomial consistency property \eqref{eq:sT:polynomial.consistency}, we can write
  \[
    \begin{aligned}
      |\cst{s}_T(\bI{T}{k} \b w,\bu v_T)|
      &= |\cst{s}_T(\bI{T}{k} \b w,\bu v_T)-\cst{s}_T(\bI{T}{k}(\PROJ{T}{k+1} \b w),\bu v_T)|
      \\ &\lesssim \cst{s}_T(\bI{T}{k} \b w,\bI{T}{k} \b w)^\frac{\sob-\sobs}{\sob}\cst{s}_T(\bI{T}{k}(\b w-\PROJ{T}{k+1} \b w),\bI{T}{k}(\b w-\PROJ{T}{k+1} \b w))^\frac{\sobs-1}{\sob}\cst{s}_T(\bu v_T,\bu v_T)^\frac{1}{\sob} \\
      &\lesssim \|\bI{T}{k} \b w\|_{\sob,T}^{\sob-\sobs}\|\bI{T}{k}(\b w-\PROJ{T}{k+1} \b w)\|_{\sob,T}^{\sobs-1}\|\bu v_T\|_{\sob,T}\\
      &\lesssim | \b w |_{W^{1,\sob}(T,\R^d)}^{\sob-\sobs}|\b w-\PROJ{T}{k+1} \b w|_{W^{1,\sob}(T,\R^d)}^{\sobs-1}\|\bu v_T\|_{\sob,T}\\
      &\lesssim h_T^{(k+1)(\sobs-1)}| \b w |_{W^{1,\sob}(T,\R^d)}^{\sob-\sobs}|\b w|_{W^{k+2,\sob}(T,\R^d)}^{\sobs-1}\|\bu v_T\|_{\sob,T},
    \end{aligned}
  \]
  where we have used the H\"older continuity \eqref{eq:sT:holder-continuity} and observed that, by the consistency property \eqref{eq:sT:polynomial.consistency}, $\cst{s}_T(\bI{T}{k}(\PROJ{T}{k+1} \b w),\bI{T}{k}(\PROJ{T}{k+1} \b w))=0$ to pass to the second line,
  we have used the boundedness property \eqref{eq:sT:stability.boundedness} to pass to the third line,
  the boundedness \eqref{eq:I:boundedness} of $\bI{T}{k}$ to pass to the fourth line,
  and the $(k+2,\sob,1)$-approximation property \eqref{eq:proj:app:T} of $\PROJ{T}{k+1}$ to conclude.
\end{proof}

In what follows, we will need generalized versions of the continuous and discrete H\"older inequalities, recalled hereafter for the sake of convenience.
Let $X \subset \R^d$ be measurable, $n \in \N^*$, and let $t,p_1,\ldots,p_n \in (0,\infty\rbrack$ be such that $\sum_{i=1}^n\frac{1}{p_i} = \frac{1}{t}$.
The continuous $(t;p_1,\ldots,p_n)$-H\"older inequality reads:
For any $(f_1,\ldots,f_n) \in \bigtimes_{i=1}^n L^{p_i}(X,\R)$, 
\begin{equation}\label{eq:holder}
  \left\| \prod_{i=1}^n f_i \right\|_{L^t(X,\R)} \le\ \prod_{i=1}^n\| f_i \|_{L^{p_i}(X,\R)}.
\end{equation}
Let $m \in \N^*$.
For all $f : \{1,\ldots,m\} \to \R$ and all $q \in [1,\infty)$, setting $\|f\|_q \coloneqq \left(\sum_{i=1}^m |f(i)|^q\right)^\frac{1}{q}$, and $\|f\|_\infty \coloneqq \max_{1 \le i \le m}|f(i)|$, the discrete $(t;p_1,\ldots,p_n)$-H\"older inequality reads:
For any $f_1,\ldots,f_n : \{1,\ldots,m\} \to \R$,
\begin{equation}\label{eq:discrete.holder}
  \left\| \prod_{i=1}^n f_i \right\|_{t} \le\ \prod_{i=1}^n\| f_i \|_{p_i}.
\end{equation}

\begin{proposition}[Properties of $\cst{s}_h$]
  Let $\cst{s}_h$ be given by \eqref{eq:sh} with, for all $T\in\T_h$, $\cst{s}_T$ satisfying Assumption \ref{ass:sT}.
  Then it holds, for all $\bu v_h \in\bdU{h}{k}$,
  \begin{subequations}\label{eq:sh:properties}
    \begin{equation}\label{eq:sh:stability.boundedness}
      \| \dgrads{k}{h}\bu v_h \|_{L^\sob(\Omega,\M{d})}^\sob + \cst{s}_h(\bu v_h,\bu v_h)
      \simeq \|\bu v_h\|_{\sob,h}^\sob.
    \end{equation}
    Furthermore, for all $\bu u_h, \bu v_h,\bu w_h\in\bdU{h}{k}$ it holds, setting $\bu e_h\coloneq\bu u_h - \bu w_h$,
    \begin{gather}\label{eq:sh:holder-continuity}
      \left|\cst{s}_h(\bu u_h,\bu v_h)-\cst{s}_h(\bu w_h,\bu v_h)\right|
      \lesssim
      \left(\cst{s}_h(\bu u_h,\bu u_h)+\cst{s}_h(\bu w_h,\bu w_h)\right)^\frac{\sob-\sobs}{\sob}\cst{s}_h(\bu e_h,\bu e_h)^\frac{\sobs-1}{\sob}\cst{s}_h(\bu v_h,\bu v_h)^\frac{1}{\sob},
      \\\label{eq:sh:strong-monotonicity}
      \left(
      \cst{s}_h(\bu u_h,\bu e_h) - \cst{s}_h(\bu w_h,\bu e_h)
      \right)\left(
      \cst{s}_h(\bu u_h,\bu u_h)+\cst{s}_h(\bu w_h,\bu w_h)
      \right)^\frac{2-\sobs}{\sob} \gtrsim \cst{s}_h(\bu e_h,\bu e_h)^\frac{\sob+2-\sobs}{\sob} .
    \end{gather}
  \end{subequations}
  Finally, for any $\b w \in \b U\cap W^{k+2,\sob}(\T_h,\R^d)$, it holds
  \begin{equation}\label{eq:sh:consist}
    \sup\limits_{\bu v_h \in \bdU{h,0}{k},\| \bu v_h \|_{\sob,h} = 1}\cst{s}_h(\bI{h}{k} \b w,\bu v_h)  \lesssim h^{(k+1)(\sobs-1)}| \b w |_{W^{1,\sob}(\Omega,\R^d)}^{\sob-\sobs}|\b w|_{W^{k+2,\sob}(\T_h,\R^d)}^{\sobs-1}.
  \end{equation}
  Above, the hidden constants are independent of $h$ and of the arguments of $\cst{s}_h$.
\end{proposition}

\begin{proof}
  For the sake of conciseness, we only sketch the proof and leave the details to the reader.
  Summing \eqref{eq:sT:stability.boundedness} over $T \in \T_h$ immediately yields \eqref{eq:sh:stability.boundedness}.
  The H\"older continuity property \eqref{eq:sh:holder-continuity} follows applying to the quantity in the left-hand side triangle inequalities, using \eqref{eq:sT:holder-continuity}, and concluding with a discrete $(1;\frac{\sob}{\sob-\sobs},\frac{\sob}{\sobs-1},\sob)$-H\"older inequality.
  Moving to \eqref{eq:sh:strong-monotonicity}, starting from $|\cst{s}_h(\bu e_h,\bu e_h)|$, we use \eqref{eq:sT:strong-monotonicity} and apply a discrete $(1;\frac{\sob+2-\sobs}{2-\sobs},\frac{\sob+2-\sobs}{\sob})$-H\"{o}lder inequality to conclude.
  Finally, to prove \eqref{eq:sh:consist} we start from $\cst{s}_h(\bu I_h^k\b w,\bu v_h)$, expand this quantity according to \eqref{eq:sh}, use, for all $T \in \T_h$, the local consistency property \eqref{eq:sT:consist} together with $h_T \le h$, invoke the discrete $(1;\frac{\sob}{\sob-\sobs},\frac{\sob}{\sobs-1},\sob)$-H\"older inequality, and pass to the supremum to conclude.
\end{proof}

\subsection{Well-posedness}\label{sec:analysis:well-posedness}

In this section, after proving H\"older continuity and strong monotonicity properties for the discrete viscous function $\cst{a}_h$ and the inf-sup stability of the pressure-velocity coupling bilinear form $\cst{b}_h$, we prove Theorem \ref{thm:well-posedness}.

\subsubsection{H\"older continuity and strong monotonicity of the viscous function}

\begin{lemma}[H\"older continuity and strong monotonicity of $\cst{a}_h$]\label{lem:ah:holder.continuity.strong.monotonicity}
  For all $\bu u_h, \bu v_h, \bu w_h \in \bdU{h}{k}$, setting $\bu e_h\coloneq\bu u_h - \bu w_h$, it holds
  \begin{subequations}\label{eq:ah:holder.continuity.strong.monotonicity}
    \begin{gather}\label{eq:ah:holder.continuity}
      \left|
      \cst{a}_h(\bu u_h,\bu v_h)-\cst{a}_h(\bu w_h,\bu v_h)
      \right| \lesssim \sigma_\cst{hc}\left( \sigma_\cst{de}^\sob+  \| \bu u_h \|_{\sob,h}^\sob+\| \bu w_h \|_{\sob,h}^\sob\right)^\frac{\sob-\sobs}{\sob}\| \bu e_h \|_{\sob,h}^{\sobs-1}\| \bu v_h \|_{\sob,h},
      \\\label{eq:ah:strong.monotonicity}
      \left(\cst{a}_h(\bu u_h,\bu e_h)-\cst{a}_h(\bu w_h,\bu e_h)\right)\left( \sigma_\cst{de}^\sob+\| \bu u_h \|_{\sob,h}^\sob+\| \bu w_h \|_{\sob,h}^\sob\right)^\frac{2-\sobs}{\sob} \gtrsim \sigma_\cst{sm}\| \bu e_h \|_{\sob,h}^{\sob+2-\sobs}.
    \end{gather}
  \end{subequations}
\end{lemma}

\begin{proof}
  (i) \textit{H\"older continuity.}
  Using a Cauchy--Schwarz inequality followed by the H\"older continuity \eqref{eq:power-framed:s.holder.continuity} of $\stress$, we can write
  \begin{equation}\label{C2:rge2:1}
    \begin{aligned}
      &\left|\int_\Omega
      \big(
      \stress(\cdot,\dgrads{k}{h}\bu u_h)-\stress(\cdot,\dgrads{k}{h}\bu w_h)
      \big):\dgrads{k}{h}\bu v_h\right| \\
      &\quad \le \sigma_\cst{hc} \int_\Omega  \left(\sigma_\cst{de}^\sob+|\dgrads{k}{h}\bu u_h|_{d\times d}^\sob+|\dgrads{k}{h}\bu w_h|_{d\times d}^\sob\right)^\frac{\sob-\sobs}{\sob}| \dgrads{k}{h}\bu e_h|_{d\times d}^{\sobs-1} |\dgrads{k}{h}\bu v_h|_{d \times d}  \\
      &\quad \lesssim \sigma_\cst{hc}\left(|\Omega|_d\sigma_\cst{de}^\sob+\|  \dgrads{k}{h}\bu u_h \|_{L^\sob(\Omega,\M{d})}^\sob+\|  \dgrads{k}{h}\bu w_h \|_{L^\sob(\Omega,\M{d})}^\sob\right)^\frac{\sob-\sobs}{\sob} \\
      & \qquad \times \|  \dgrads{k}{h}\bu e_h \|_{L^\sob(\Omega,\M{d})}^{\sobs-1}\| \dgrads{k}{h}\bu v_h \|_{L^\sob(\Omega,\M{d})}\\
      &\quad \lesssim \sigma_\cst{hc}\left(\sigma_\cst{de}^\sob+ \| \bu u_h \|_{\sob,h}^\sob+\| \bu w_h \|_{\sob,h}^\sob\right)^\frac{\sob-\sobs}{\sob}\| \bu e_h \|_{\sob,h}^{\sobs-1}\| \bu v_h \|_{\sob,h},
    \end{aligned}
  \end{equation}
  where we have used the $(1;\frac{\sob}{\sob-\sobs},\frac{\sob}{\sobs-1},\sob)$-H\"{o}lder inequality \eqref{eq:holder} in the second bound
  and the global seminorm equivalence \eqref{eq:sh:stability.boundedness} together with the fact that $|\Omega|_d\lesssim 1$ (since $\Omega$ is bounded) to conclude.
  For the stabilization term, combining the H\"older continuity \eqref{eq:sh:holder-continuity} of $\cst{s}_h$ and the seminorm equivalence \eqref{eq:sh:stability.boundedness} readily gives
  \begin{equation}\label{C2:rge2:2}
    \left|
    \cst{s}_h(\bu u_h,\bu v_h)-\cst{s}_h(\bu w_h,\bu v_h)
    \right|
    \lesssim \left(\sigma_\cst{de}^\sob+  \| \bu u_h \|_{\sob,h}^\sob+\| \bu w_h \|_{\sob,h}^\sob\right)^\frac{\sob-\sobs}{\sob}\| \bu e_h \|_{\sob,h}^{\sobs-1}\| \bu v_h \|_{\sob,h}, 
  \end{equation}
  where we have additionally noticed that $\sigma_\cst{de}^\sob\ge 0$ to add this term to the quantity inside parentheses.
  Using the definition \eqref{eq:ah} of $\cst{a}_h$, a triangle inequality followed by \eqref{C2:rge2:1} and \eqref{C2:rge2:2}, and recalling that $\gamma \le \sigma_\cst{hc}$ (cf. \eqref{eq:gamma}), \eqref{eq:ah:holder.continuity} follows.
  \medskip\\
  (ii) \textit{Strong monotonicity.} Using the strong monotonicity \eqref{eq:power-framed:s.strong.monotonicity} of $\stress$ and the $(1;\frac{\sob+2-\sobs}{2-\sobs},\frac{\sob+2-\sobs}{\sob})$-H\"{o}lder inequality \eqref{eq:holder}, we get
  \begin{equation}\label{eq:ah:sm:1}
    \hspace{-0.2cm}\begin{aligned}
      &\sigma_\cst{sm}^\frac{\sob}{\sob+2-\sobs}\| \dgrads{k}{h} \bu e_h\|_{L^\sob(\Omega,\M{d})}^\sob \\
      &\leq \int_\Omega \left(\sigma_\cst{de}^\sob+|\dgrads{k}{h}\bu u_h|_{d\times d}^\sob+|\dgrads{k}{h}\bu w_h|_{d\times d}^\sob\right)^{\frac{2-\sobs}{\sob+2-\sobs}}\left(
      \big(\stress(\cdot,\dgrads{k}{h}\bu u_h)-\stress(\cdot,\dgrads{k}{h}\bu w_h)\big):\dgrads{k}{h} \bu e_h
      \right)^\frac{\sob}{\sob+2-\sobs}\\
      &\lesssim \left(\sigma_\cst{de}^\sob+\| \dgrads{k}{h} \bu u_h\|_{L^\sob(\Omega,\M{d})}^\sob+\| \dgrads{k}{h} \bu w_h\|_{L^\sob(\Omega,\M{d})}^\sob\right)^\frac{2-\sobs}{\sob+2-\sobs} \\
      &\qquad \times \left(
        \int_\Omega \left(
        \stress(\cdot,\dgrads{k}{h}\bu u_h)-\stress(\cdot,\dgrads{k}{h}\bu w_h)
        \right):\dgrads{k}{h} \bu e_h
        \right)^\frac{\sob}{\sob+2-\sobs}\\
      &\lesssim \left(\sigma_\cst{de}^\sob+\| \bu u_h \|_{\sob,h}^\sob+\| \bu w_h \|_{\sob,h}^\sob\right)^\frac{2-\sobs}{\sob+2-\sobs} \left(
        \int_\Omega \left(
        \stress(\cdot,\dgrads{k}{h}\bu u_h)-\stress(\cdot,\dgrads{k}{h}\bu w_h)
        \right):\dgrads{k}{h} \bu e_h
        \right)^\frac{\sob}{\sob+2-\sobs},
    \end{aligned}
  \end{equation}
  where the conclusion follows from the global seminorm equivalence \eqref{eq:sh:stability.boundedness}.
  Additionally, using the strong monotonicity \eqref{eq:sh:strong-monotonicity} of $\cst{s}_h$ together with the fact that $\sigma_\cst{sm} \le \gamma$ (cf. \eqref{eq:gamma}) and invoking again the seminorm equivalence \eqref{eq:sh:stability.boundedness}, we readily obtain
  \begin{equation}\label{eq:ah:sm:2}
    \sigma_\cst{sm}^\frac{\sob}{\sob+2-\sobs}\cst{s}_h(\bu e_h,\bu e_h)
    \lesssim \left(\sigma_\cst{de}^\sob+\| \bu u_h \|_{\sob,h}^\sob+\| \bu w_h \|_{\sob,h}^\sob\right)^\frac{2-\sobs}{\sob+2-\sobs}\left(
      \gamma\cst{s}_h(\bu u_h,\bu e_h)-\gamma\cst{s}_h(\bu w_h,\bu e_h)
      \right)^\frac{\sob}{\sob+2-\sobs}.
  \end{equation}
  Finally, combining again the norm equivalence \eqref{eq:sh:stability.boundedness} with \eqref{eq:ah:sm:1} and \eqref{eq:ah:sm:2}, and using \eqref{eq:sum-power} yields
  \[
  \begin{aligned}
    \sigma_\cst{sm}^\frac{\sob}{\sob+2-\sobs}\| \bu e_h \|_{\sob,h}^\sob \lesssim &\left( \sigma_\cst{de}^\sob+\| \bu u_h \|_{\sob,h}^\sob+\| \bu w_h \|_{\sob,h}^\sob\right)^\frac{2-\sobs}{\sob+2-\sobs}\left(
      \cst{a}_h(\bu u_h,\bu e_h)-\cst{a}_h(\bu w_h,\bu e_h)
      \right)^\frac{\sob}{\sob+2-\sobs}.
  \end{aligned}
  \]
  Raising this inequality to the power $\frac{\sob-2-\sobs}{\sob}$ yields \eqref{eq:ah:strong.monotonicity}.
\end{proof}

\subsubsection{Stability of the pressure-velocity coupling}

\begin{lemma}[Inf-sup stability of $\cst{b}_h$]\label{lem:bh:inf-sup}
  It holds, for all $q_h \in \dP{h}{k}$,
  \begin{equation}\label{eq:bh:inf-sup}
    \| q_h \|_{L^{\sob'}(\Omega,\R)} \lesssim \sup\limits_{\bu v_h \in \bdU{h,0}{k}, \| \bu v_h \|_{\sob,h} = 1} \cst{b}_h(\bu v_h,q_h),
  \end{equation}
  with hidden constant depending only on $d$, $k$, $\sob$, $\Omega$, and the mesh regularity parameter.
\end{lemma}

\begin{proof}
  The proof follows the classical Fortin argument (cf., e.g., \cite[Section 8.4]{Boffi.Brezzi.ea:13}), adapted here to the non-Hilbertian setting.
  \medskip\\
  (i) \emph{Fortin operator.} We need to prove that the following properties hold for any $\b v\in W^{1,\sob}(\Omega,\R^d)$:
  \begin{subequations}\label{eq:fortin}
    \begin{gather}
      \| \bI{h}{k} \b v\|_{\sob,h} \lesssim | \b v |_{W^{1,\sob}(\Omega,\R^d),}\label{eq:fortin:boundedness}
      \\ 
      \cst{b}_h(\bI{h}{k} \b v,q_h) = b(\b v,q_h)\qquad\forall q_h\in\Poly^k(\T_h,\R).\label{eq:fortin:consistency}
    \end{gather}
  \end{subequations}
  Property \eqref{eq:fortin:boundedness} is obtained by raising both sides of \eqref{eq:I:boundedness} to the power $\sob$, summing over $T \in \T_h$, then taking the $r$th root of the resulting inequality.
  The proof of \eqref{eq:fortin:consistency} is given, e.g., in \cite[Lemma 8.12]{Di-Pietro.Droniou:20}.
  \medskip\\
  (ii) \emph{Inf-sup condition on $\cst{b}_h$.}
  Let $q_h \in \dP{h}{k}$ and set $c_h \coloneqq \int_\Omega |q_h|^{\sob'-2}q_h$.
  Using the triangle and H\"{o}lder inequalities, we get 
  \begin{equation}\label{eq:inf-sup:qT}
  \| |q_h|^{\sob'-2}q_h-c_h \|_{L^{\sob}(\Omega,\R)}
  \le \| q_h\|_{L^{\sob'}(\Omega,\R)}^{\sob'-1}+|c_h||\Omega|_d^\frac{1}{\sob}
  \le \left(1+|\Omega|_d\right)\| q_h\|_{L^{\sob'}(\Omega,\R)}^{\sob'-1} \lesssim \| q_h\|_{L^{\sob'}(\Omega,\R)}^{\sob'-1},
  \end{equation}
  where we have used the fact that $|c_h|\le\| q_h\|_{L^{\sob'}(\Omega,\R)}^{\sob'-1} |\Omega|_d^{\frac1{\sob'}}$ along with $\frac1\sob+\frac1{\sob'}=1$ in the second bound and the fact that $|\Omega|_d\lesssim 1$ to conclude.
  Thus, using the surjectivity of the continuous divergence operator $\div : \b U \to L^{\sob}_0(\Omega,\R)\coloneq \left\{ q \in L^\sob(\Omega,\R) : \int_\Omega q = 0 \right\}$, (c.f. \cite{Duran.Muschietti.ea:10} and also \cite[Theorem 1]{Bogovski:79}), we infer that there exists  $\b v_{q_h} \in \b U$ such that 
  \begin{equation}\label{eq:velocity.lifting}
    -\div \b v_{q_h} = |q_h|^{\sob'-2}q_h-c_h \quad \cst{and} \quad | \b v_{q_h}|_{W^{1,\sob}(\Omega,\R^d)} \lesssim \| |q_h|^{\sob'-2}q_h-c_h \|_{L^\sob(\Omega,\R)}.
  \end{equation}
  Denote by $\$$ the supremum in \eqref{eq:bh:inf-sup}. Using the fact that $q_h$ has zero mean value over $\Omega$, the equality in \eqref{eq:velocity.lifting} together with the definition \eqref{eq:a.b} of $b$, and the second Fortin property \eqref{eq:fortin:consistency}, we have
  \[
  \|q_h \|_{L^{\sob'}(\Omega,\R)}^{\sob'}
  {=} \int_\Omega \big(|q_h|^{\sob'-2}q_h-c_h\big) q_h
  = b(\b v_{q_h},q_h)
  = \cst{b}_h(\bI{h}{k}\b v_{q_h},q_h) 
  \le \$  \| \bI{h}{k} \b v_{q_h} \|_{\sob,h}
  \lesssim \$ \| q_h \|_{L^{\sob'}(\Omega,\R)}^{\sob'-1},
  \]
  where, to conclude, we have used \eqref{eq:fortin:boundedness} followed by \eqref{eq:velocity.lifting} and \eqref{eq:inf-sup:qT}.
  Simplifying yields \eqref{eq:bh:inf-sup}.
\end{proof}

\subsubsection{Proof of Theorem \ref{thm:well-posedness}}\label{sec:well-posedness:proof}

\begin{proof}[Proof of Theorem \ref{thm:well-posedness}]
  (i) \emph{Existence.}
  Denote by $\dP{h}{k,*}$ the dual space of $\dP{h}{k}$ and let $B_h : \bdU{h,0}{k} \to \dP{h}{k,*}$ be such that, for all $\bu v_h \in \bdU{h,0}{k}$,
  \[
  \langle B_h\bu v_h, q_h \rangle \coloneqq -\cst{b}_h(\bu v_h,q_h) \qquad \forall q_h \in \dP{h}{k}.
  \]
  Here and in what follows, $\langle{\cdot},{\cdot}\rangle$ denotes the appropriate duality pairing as inferred from its arguments.
  Define the following subspace of $\bdU{h,0}{k}$ spanned by vectors of discrete unknowns with zero discrete divergence:
  \begin{equation}\label{eq:Whk}
    \bu W_h^k \coloneq \Kernel(B_h) = \left\{
    \bu v_h \in \bdU{h,0}{k} : \cst{b}_h(\bu v_h,q_h) = 0 \quad \forall q_h \in \dP{h}{k}
    \right\},
  \end{equation}
  and consider the following problem:
  Find $\bu u_h\in\bu W_h^k$ such that
  \begin{equation}\label{eq:existence:auxiliary}
    \cst{a}_h(\bu u_h,\bu v_h) = \int_\Omega \b f\cdot \b v_h\qquad\forall \bu v_h\in\bu W_h^k.
  \end{equation}
  Existence of a solution to this problem for a fixed $h$ can be proved adapting the arguments of \cite[Theorem 4.5]{Di-Pietro.Droniou:17}.
  Specifically, equip $\bu W_h^k$ with an inner product $(\cdot,\cdot)_{\b W,h}$ (which need not be further specified), denote by $\|{\cdot}\|_{\b W,h}$ the induced norm, and let $\b\Phi_h:\bu W_h^k\to \bu W_h^k$ be such that, for all $\bu w_h\in\bu W_h^k$, $(\b\Phi_h(\bu w_h),\bu v_h)_{\b W,h} = \cst{a}_h(\bu w_h,\bu v_h)$ for all $\bu v_h\in\bu W_h^k$.
  The strong monotonicity \eqref{eq:ah:strong.monotonicity} of $\cst{a}_h$ yields, for any $\bu v_h\in\bu W_h^k$ such that $\| \bu v_h \|_{\sob,h} \ge \sigma_\cst{de}$, 
  \[
  (\b\Phi_h(\bu v_h),\bu v_h)_{\b W,h}\ge \sigma_\cst{sm} (\sigma_\cst{de}^\sob + \| \bu v_h \|_{\sob,h}^\sob)^\frac{\sobs-2}{\sob} \| \bu v_h \|_{\sob,h}^{\sob+2-\sobs} \gtrsim \sigma_\cst{sm}\| \bu v_h \|_{\sob,h}^\sob\ge C^\sob \sigma_\cst{sm}\| \bu v_h\|_{\b W,h}^\sob,
  \]
   where $C$ denotes the constant (possibly depending on $h$) in the equivalence of the norms $\|{\cdot}\|_{\sob,h}$ and $\|{\cdot}\|_{\b W,h}$ (which holds since $\bu W_h^k$ is finite-dimensional).
  This shows that $\b\Phi_h$ is coercive hence, by \cite[Theorem 3.3]{Deimling:85}, surjective.
  Let now $\bu w_h\in\bu W_h^k$ be such that $(\bu w_h,\bu v_h)_{\b W,h}=\int_\Omega \b f\cdot \b v_h$ for all $\bu v_h\in\bu W_h^k$.
  By the surjectivity of $\b\Phi_h$, there exists $\bu u_h\in \bu W_h^k$ such that $\b\Phi_h(\bu u_h)=\bu w_h$ which, by definition of $\bu w_h$ and $\b\Phi_h$, is a solution to the discrete problem \eqref{eq:existence:auxiliary}.

  The proof of existence now continues as in the linear case; see, e.g., \cite[Theorem 4.2.1]{Boffi.Brezzi.ea:13}.
  Denote by $\bdU{h,0}{k,*}$ the dual space of $\bdU{h,0}{k}$ and consider the linear mapping $\ell_h\in\bdU{h,0}{k,*}$ such that, for all $\bu v_h\in\bdU{h,0}{k}$,
  \[
  \langle\ell_h,\bu v_h\rangle\coloneq \int_\Omega \b f\cdot \b v_h - \cst{a}_h(\bu u_h, \bu v_h).
  \]
  Thanks to \eqref{eq:existence:auxiliary}, $\ell_h$ vanishes identically for every $\bu v_h\in\bu W_h^k$, that is to say, $\ell_h$ lies in the polar space of $\bu W_h^k$ which, denoting by $B_h^*:\dP{h}{k}\to\bdU{h,0}{k,*}$ the adjoint operator of $B_h$, coincides in our case with $\Image(B_h^*)$ (see, e.g., \cite[Theorem 4.14]{Boffi.Brezzi.ea:13}).
  Hence, $\ell_h\in\Image(B_h^*)$, and there exists therefore a $p_h\in \dP{h}{k}$ such that $B_h^* p_h = \ell_h$.
  This means that, for all $\bu v_h\in\bdU{h,0}{k}$,
  \[
  \cst{b}_h(\bu v_h,p_h)
  = \langle B_h^* p_h,\bu v_h\rangle
  = \langle\ell_h,\bu v_h\rangle
  = \int_\Omega \b f\cdot \b v_h - \cst{a}_h(\bu u_h, \bu v_h),
  \]
  i.e., the $(\bu u_h,p_h)$ satisfies the discrete momentum equation \eqref{eq:stokes.discrete:momentum}. On the other hand, since $\bu u_h\in\bu W_h^k$, we also have, by the definition \eqref{eq:Whk} of $\bu W_h^k$, $  \cst{b}_h(\bu u_h,q_h) = 0$ for all $q_h\in \dP{h}{k}$, which shows that the discrete mass equation \eqref{eq:stokes.discrete:mass} is also verified.
  In conclusion, $(\bu u_h,p_h)\in\bdU{h,0}{k}\times \dP{h}{k}$ solves \eqref{eq:stokes.discrete}.
  \\
  (ii) \emph{Uniqueness.}
  We start by proving uniqueness for the velocity.
  Let $(\bu u_h,p_h),(\bu u'_h,p'_h) \in \bdU{h,0}{k} \times \dP{h}{k}$ be two solutions of \eqref{eq:stokes.discrete}.
  Making $\bu v_h = \bu u_h - \bu u'_h$ in \eqref{eq:stokes.discrete:momentum} written first for $(\bu u_h, p_h)$ then for $(\bu u'_h, p'_h)$, then taking the difference and observing that $\cst{b}_h(\bu u_h-\bu u'_h,p_h)=\cst{b}_h(\bu u_h - \bu u'_h,p'_h)=0$ by \eqref{eq:stokes.discrete:mass}, we infer that
  \[
  \cst{a}_h(\bu u_h,\bu u_h - \bu u'_h)-\cst{a}_h(\bu u'_h,\bu u_h - \bu u'_h) = 0.
  \]
  Thus, the strong monotonicity \eqref{eq:ah:strong.monotonicity} of $\cst{a}_h$ yields $\| \bu u_h - \bu u'_h \|_{\sob,h} = 0$, which implies $\bu u_h = \bu u'_h$ since $\|{\cdot}\|_{\sob,h}$ is a norm on $\bdU{h,0}{k}$.
  Moreover, using the inf-sup stability \eqref{eq:bh:inf-sup} of $\cst{b}_h$ and \eqref{eq:stokes.discrete:momentum} written first for $\bu u_h$ then for $\bu u_h'$, we get
  \[
  \begin{aligned}
    \| p_h-p'_h \|_{L^{\sob'}(\Omega,\R)}
    &\lesssim \sup\limits_{\bu v_h \in \bdU{h,0}{k},\| \bu v_h \|_{\sob,h} = 1} \cst{b}_h(\bu v_h,p_h-p'_h)
    \\
    &=  \sup\limits_{\bu v_h \in \bdU{h,0}{k},\| \bu v_h \|_{\sob,h} = 1} \left(\cst{a}_h(\bu u'_h,\bu v_h)-\cst{a}_h(\bu u_h,\bu v_h)\right) = 0,
   \end{aligned} 
  \]
  hence $p_h=p'_h$.
  \medskip\\
  (iii) \emph{A priori estimates.}
  Using the strong monotonicity \eqref{eq:ah:strong.monotonicity} of $\cst{a}_h$ (with $\bu w_h = \bu 0$), equation \eqref{eq:stokes.discrete:momentum} together with \eqref{eq:stokes.discrete:mass}, and the H\"{o}lder inequality together with the discrete Korn inequality \eqref{eq:discrete.Korn}, we obtain
  \begin{equation}\label{eq:well-posedness:0}
    \begin{aligned}
      \sigma_\cst{sm}\big(
      \sigma_\cst{de}^\sob + \| \bu u_h \|_{\sob,h}^\sob
      \big)^\frac{\sobs-2}{\sob} \| \bu u_h \|_{\sob,h}^{\sob+2-\sobs}
      &\lesssim  \cst{a}_h(\bu u_h,\bu u_h)
      = \displaystyle\int_\Omega \b f \cdot \b u_h
      \lesssim \| \b f \|_{L^{\sob'}(\Omega,\R^d)}\| \bu u_h \|_{\sob,h}.
    \end{aligned}
  \end{equation}
  We then conclude as in the continuous case to infer \eqref{eq:discrete.solution:bounds:uh} (see Remark \ref{rem:a-priori}).
  To prove the bound \eqref{eq:discrete.solution:bounds:ph} on the pressure, we use the inf-sup stability \eqref{eq:bh:inf-sup} of $\cst{b}_h$ to write
  \[
  \begin{aligned}
    \| p_h \|_{L^{\sob'}(\Omega,\R)}
    &\lesssim \sup\limits_{\bu v_h \in \bdU{h,0}{k}, \| \bu v_h \|_{\sob,h} = 1} \cst{b}_h(\bu v_h,p_h)
    \\
    &= \sup\limits_{\bu v_h \in \bdU{h,0}{k}, \| \bu v_h \|_{\sob,h}
     = 1} \left(\displaystyle\int_\Omega \b f \cdot \b v_h - \cst{a}_h(\bu u_h,\bu v_h) \right) \\
    &\lesssim \| \b f \|_{L^{\sob'}(\Omega,\R^d)} +\sigma_\cst{hc}(\sigma_\cst{de}^\sob+\| \bu u_h \|_{\sob,h}^\sob)^\frac{\sob-\sobs}{\sob}\| \bu u_h \|_{\sob,h}^{\sobs-1} \\   
    &\lesssim \sigma_\cst{hc}\left(\sigma_\cst{sm}^{-1}\| \b f \|_{L^{\sob'}(\Omega,\R^d)}+\sigma_\cst{de}^{|\sob-2|(\sobs-1)}\left(\sigma_\cst{sm}^{-1}\| \b f \|_{L^{\sob'}(\Omega,\R^d)}\right)^\frac{\sobs-1}{\sob+1-\sobs}\right),
  \end{aligned}
  \]
  where we have used the discrete momentum equation \eqref{eq:stokes.discrete:momentum} to pass to the second line,
  the H\"{o}lder and discrete Korn \eqref{eq:discrete.Korn} inequalities together with the H\"older continuity \eqref{eq:ah:holder.continuity} of $\cst{a}_h$ to pass to the third line,
  and the a priori bound \eqref{eq:discrete.solution:bounds:uh} on the velocity together with $\frac{\sigma_\cst{hc}}{\sigma_\cst{sm}}\ge 1$ (see \eqref{eq:power-framed:constants.bound}) to conclude.
\end{proof}

\subsection{Error estimate}\label{sec:analysis:error.estimate}

In this section, after studying the consistency of the viscous and pressure-velocity coupling terms, we prove Theorem \ref{thm:error.estimate}.

\subsubsection{Consistency of the viscous function}\label{sec:consistency:viscous.function}

\begin{lemma}[Consistency of $\cst{a}_h$]\label{lem:consistency:ah}
  Let $\b w \in \b U \cap W^{k+2,\sob}(\T_h,\R^d)$ be such that $\stress(\cdot,\GRADs \b w) \in W^{1,\sob'}(\Omega,\Ms{d}) \cap W^{(k+1)(\sobs-1),\sob'}(\T_h,\Ms{d})$.
  Define the viscous consistency error linear form $\mathcal E_{\cst{a},h}(\b w;\cdot) : \bdU{h}{k} \to \R$ such that, for all $\bu v_h \in \bdU{h}{k}$,
  \begin{equation}\label{eq:Eah}
    \mathcal E_{\cst{a},h}(\b w;\bu v_h) \coloneqq \int_\Omega (\DIV \stress(\cdot,\GRADs \b w)) \cdot \b v_h + \cst{a}_h(\bI{h}{k} \b w,\bu v_h).
  \end{equation}
  Then, under Assumptions \ref{ass:stress} and \ref{ass:sT}, we have
  \begin{multline}\label{eq:consistency:ah}
    \sup_{\bu v_h \in \bdU{h,0}{k},\| \bu v_h \|_{\sob,h} = 1} \mathcal E_{\cst{a},h}(\b w;\bu v_h)
    \lesssim
    h^{(k+1)(\sobs-1)}\bigg[
      \sigma_\cst{hc}\left(\sigma_\cst{de}^\sob + |\b w|_{W^{1,\sob}(\Omega,\R^d)}^\sob\right)^\frac{\sob-\sobs}{\sob}|\b w|_{W^{k+2,\sob}(\T_h,\R^d)}^{\sobs-1}
      \\
      +|\stress(\cdot,\GRADs \b w)|_{W^{(k+1)(\sobs-1),\sob'}(\T_h,\M{d})}        
      \bigg].
  \end{multline}
\end{lemma}

\begin{proof}
  Let $\bu{\hat w}_h \coloneqq \bI{h}{k} \b w$ and $\bu v_h \in \bdU{h,0}{k}$.
  Expanding $\cst{a}_h$ according to its definition \eqref{eq:ah} in the expression \eqref{eq:Eah} of $\mathcal E_{\cst{a},h}$, inserting $\pm\left(
  \int_\Omega\stress(\cdot,\GRADs \b w): \dgrads{k}{h}\bu v_h
  +  \int_\Omega\PROJ{h}{k}\stress(\cdot,\GRADs \b w) : \dgrads{k}{h}\bu v_h
  \right)$, and rearranging, we obtain
  \begin{multline}\label{eq:consistency:ah:EJ}
    \mathcal E_{\cst{a},h}(\b w;\bu v_h)
    =
    \\
    \underbrace{%
      \int_\Omega (\DIV \stress(\cdot,\GRADs \b w)) \cdot \b v_h
      {+} \int_\Omega \PROJ{h}{k}\stress(\cdot,\GRADs \b w) : \dgrads{k}{h}\bu v_h
    }_{\mathcal T_1}    
    + \cancel{
      \int_\Omega \left( \stress(\cdot,\GRADs \b w) - \PROJ{h}{k}\stress(\cdot,\GRADs \b w)\right) : \dgrads{k}{h}\bu v_h}
    \\
    + \underbrace{%
      \int_\Omega \left( \stress(\cdot,\dgrads{k}{h} \bu{\hat w}_h) - \stress(\cdot,\GRADs \b w)\right): \dgrads{k}{h}\bu v_h
    }_{\mathcal T_2}
    + \underbrace{\vphantom{\int_\Omega}\gamma \cst{s}_h(\bu{\hat w}_h,\bu v_h)}_{\mathcal T_3},
  \end{multline}
  where have used the definition \eqref{eq:proj} of $\PROJ{h}{k}$ together with the fact that $\dgrads{k}{h}\bu v_h\in \Poly^{k}(\T_h,\Ms{d})$ in the cancellation.
  We proceed to estimate the terms in the right-hand side.
  For the first term, we start by noticing that
  \begin{equation}\label{eq:consistency:ah:null}
    \sum_{T \in \T_h}\sum_{F \in \F_T} \int_F  \b v_F \cdot\left(\stress(\cdot,\GRADs \b w)\b n_{TF}\right) = 0
  \end{equation}
  as a consequence of the continuity of the normal trace of $\stress(\cdot,\GRADs \b w)$ together with the single-valuedness of $\b v_F$ across each interface $F\in\Fi$ and of the fact that $\b v_F=\b 0$ for every boundary face $F\in\Fb$.
  Using an element by element integration by parts on the first term of $\mathcal T_1$ along with the definitions \eqref{eq:Gh} of $\dgrads{k}{h}$ and \eqref{eq:G} of $\dgrads{k}{T}$, we can write
  \[
  \begin{aligned}
    \mathcal T_1
    &= \cancel{\int_\Omega \left(\PROJ{h}{k}\stress(\cdot,\GRADs \b w)- \stress(\cdot,\GRADs \b w)\right) : \brkGRADs \b v_h} \\
    &\qquad + \sum_{T \in \T_h}\sum_{F \in \F_T} \left(\int_F  (\b v_F-\b v_T)\cdot(\PROJ{T}{k}\stress(\cdot,\GRADs \b w))\b n_{TF}+\int_F\b v_T\cdot\left(\stress(\cdot,\GRADs \b w)\b n_{TF}\right) \right) \\
    &=  \sum_{T \in \T_h}\sum_{F \in \F_T} \int_F  (\b v_F-\b v_T)\cdot\left(\PROJ{T}{k}\stress(\cdot,\GRADs \b w)-\stress(\cdot,\GRADs \b w)\right)\b n_{TF},
  \end{aligned}
  \]
  where we have used the definition \eqref{eq:proj} of $\PROJ{h}{k}$ together with the fact that $\brkGRADs \b v_h \in \Poly^{k-1}(\T_h,\Ms{d}) \subset \Poly^{k}(\T_h,\Ms{d})$ to cancel the term in the first line,
  and we have inserted \eqref{eq:consistency:ah:null} and rearranged to conclude.
  Therefore, applying the H\"{o}lder inequality together with the bound $h_F \le h_T$, we infer
  \begin{equation}\label{eq:consistency:ah:T1}
    \begin{aligned}
      \left|\mathcal T_1\right|
      &\le \left(\displaystyle\sum_{T \in \T_h}h_T \|\stress(\cdot,\GRADs \b w)- \PROJ{T}{k}\stress(\cdot,\GRADs \b w) \|_{L^{\sob'}(\partial T,\M{d})}^{\sob'} \right)^\frac{1}{\sob'}\left(\sum_{T \in \T_h}\sum_{F \in \F_T}h_F^{1-\sob}\| \b v_F-\b v_T\|_{L^\sob(F,\R^d)}^\sob\right)^\frac{1}{\sob}
      \\
      &\lesssim h^{(k+1)(\sobs-1)} |\stress(\cdot, \GRADs \b w)|_{W^{(k+1)(\sobs-1),\sob'}(\T_h,\M{d})}\| \bu v_h \|_{\sob,h},
    \end{aligned}
  \end{equation}
  where the conclusion follows using the $((k+1)(\sobs-1),\sob')$-trace approximation properties \eqref{eq:proj:app:F} of $\PROJ{T}{k}$ along with $h_T \le h$ for the first factor and the definition \eqref{eq:norm.epsilon.r} of the $\|{\cdot}\|_{\sob,h}$-norm for the second.
   
 For the second term, using the H\"older inequality and again \eqref{eq:sh:stability.boundedness}, we get
  \begin{equation}\label{eq:consistency:ah:T3:0}
    \left|\mathcal T_2\right|
    \le \|\stress(\cdot,\dgrads{k}{h} \bu{\hat w}_h)- \stress(\cdot,\GRADs \b w) \|_{L^{\sob'}(\Omega,\M{d})}\| \bu v_h \|_{\sob,h}.
  \end{equation}
  We estimate the first factor as follows:
  \[
    \begin{aligned}
      &\|\stress(\cdot,\dgrads{k}{h} \bu{\hat w}_h)- \stress(\cdot,\GRADs \b w) \|_{L^{\sob'}(\Omega,\M{d})}
      \\
      &\quad
      \le \sigma_\cst{hc} \left\| \left(\sigma_\cst{de}^\sob+| \dgrads{k}{h} \bu{\hat w}_h |_{d \times d}^\sob  +  | \GRADs \b w |_{d \times d}^\sob\right)^\frac{\sob-\sobs}{\sob}| \dgrads{k}{h} \bu{\hat w}_h - \GRADs \b w |_{d \times d}^{\sobs-1}\right\|_{L^{\sob'}(\Omega,\R)} \\
      &\quad
      \lesssim \sigma_\cst{hc} \left(\sigma_\cst{de}^\sob+\| \dgrads{k}{h} \bu{\hat w}_h \|_{L^\sob(\Omega,\M{d})}^\sob+\|  \GRADs \b w \|_{L^\sob(\Omega,\M{d})}^\sob\right)^\frac{\sob-\sobs}{\sob}\| \dgrads{k}{h} \bu{\hat w}_h - \GRADs \b w \|_{L^\sob(\Omega,\M{d})}^{\sobs-1} \\
      &\quad
      \lesssim \sigma_\cst{hc}\left(\sigma_\cst{de}^\sob+\|\bu{\hat w}_h \|_{\sob,h}^\sob  +  |\b w|_{W^{1,\sob}(\Omega,\R^d)}^\sob\right)^\frac{\sob-\sobs}{\sob}\| \PROJ{h}{k}(\GRADs \b w) - \GRADs \b w \|_{L^\sob(\Omega,\M{d})}^{\sobs-1} \\
      &\quad
      \lesssim h^{(k+1)(\sobs-1)}\sigma_\cst{hc}\left(\sigma_\cst{de}^\sob+|\b w|_{W^{1,\sob}(\Omega,\R^d)}^\sob\right)^\frac{\sob-\sobs}{\sob}|\b w|_{W^{k+2,\sob}(\T_h,\R^d)}^{\sobs-1},
    \end{aligned}
  \]
  where we have used the H\"older continuity \eqref{eq:power-framed:s.holder.continuity} of $\stress$ in the first bound,
  the $(\sob';\frac{\sob}{\sob-\sobs},\frac{\sob}{\sobs-1})$-H\"{o}lder inequality \eqref{eq:holder} in the second,
  the boundedness of $\Omega$ along with \eqref{eq:sh:stability.boundedness} and the commutation property \eqref{eq:G:proj} of $\dgrads{k}{h}$ in the third,
  and we have concluded invoking the $(k+1,\sob,0)$-approximation property \eqref{eq:proj:app:T} of $\PROJ{T}{k}$.
  Plugging this estimate into \eqref{eq:consistency:ah:T3:0}, we get
  \begin{equation}\label{eq:consistency:ah:T3}
    \left|\mathcal T_2\right|
    \lesssim h^{(k+1)(\sobs-1)}\sigma_\cst{hc}\left(\sigma_\cst{de}^\sob+|\b w|_{W^{1,\sob}(\Omega,\R^d)}^\sob\right)^\frac{\sob-\sobs}{\sob}|\b w|_{W^{k+2,\sob}(\T_h,\R^d)}^{\sobs-1}
    \| \bu v_h \|_{\sob,h}.
  \end{equation}

  Finally, using the fact that $\gamma \le \sigma_\cst{hc}$ together with the consistency \eqref{eq:sh:consist} of $\cst{s}_h$ and the norm equivalence \eqref{eq:sh:stability.boundedness}, we obtain for the third term
  \begin{equation}\label{eq:consistency:ah:T4}
    \left|\mathcal T_3\right|
    \lesssim h^{(k+1)(\sobs-1)}\sigma_\cst{hc}|\b w|_{W^{1,\sob}(\Omega,\R^d)}^{\sob-\sobs}|\b w|_{W^{k+2,\sob}(\T_h,\R^d)}^{\sobs -1}
    \| \bu v_h \|_{\sob,h}.
  \end{equation}

  Plug the bounds \eqref{eq:consistency:ah:T1}, \eqref{eq:consistency:ah:T3}, and \eqref{eq:consistency:ah:T4} into \eqref{eq:consistency:ah:EJ} and pass to the supremum to conclude.
\end{proof}

\subsubsection{Consistency of the pressure-velocity coupling bilinear form}

\begin{lemma}[Consistency of $\cst{b}_h$] Let $q \in W^{1,\sob'}(\Omega,\R) \cap W^{(k+1)(\sobs-1),\sob'}(\T_h,\R)$. 
Let $\mathcal E_{\cst{b},h}(q;\cdot) : \bdU{h}{k} \to \R$ be the pressure consistency error linear form such that, for all $\bu v_h \in \bdU{h}{k}$,
  \begin{equation}\label{eq:Ebh}
    \mathcal E_{\cst{b},h}(q;\bu v_h) \coloneqq \int_\Omega \GRAD q \cdot \b v_h - \cst{b}_h(\bu v_h,\proj{h}{k} q).
  \end{equation}
  Then, we have that
  \begin{equation}\label{eq:consistency:bh}
    \sup\limits_{\bu v_h \in \bdU{h,0}{k},\| \bu v_h \|_{\sob,h} = 1} \mathcal E_{\cst{b},h}(q;\bu v_h) \lesssim h^{(k+1)(\sobs-1)} |q|_{W^{(k+1)(\sobs-1),\sob'}(\T_h,\R)}.
  \end{equation}
\end{lemma}

\begin{proof} Let $\bu v_h \in \bdU{h,0}{k}$.
  Integrating by parts element by element, we can reformulate the first term in the right-hand side of \eqref{eq:Ebh} as follows:
  \begin{equation}\label{eq:consistency:bh:1}
    \displaystyle\int_\Omega \GRAD q\cdot \b v_h = - \sum_{T \in \T_h} \left( \int_T q(\div \b v_T) +\sum_{F \in \F_T} \int_F q(\b v_F - \b v_T)\cdot \b n_{TF} \right),
  \end{equation}
  where the introduction of $\b v_F$ in the boundary term is justified by the fact that the jumps of $q$ vanish across interfaces by the assumed regularity and that $\b v_F=\b 0$ on every boundary face $F \in \Fb$.
  On the other hand, expanding, for each $T \in \T_h$, $\ddiv{k}{T}$ according to its definition \eqref{eq:D}, we get
  \begin{equation}\label{eq:consistency:bh:2}
    -\cst{b}_h(\bu v_h, \proj{h}{k} q) = \sum_{T \in \T_h} \left( \int_T \proj{T}{k} q~(\div \b v_T) +\sum_{F \in \F_T} \int_F \proj{T}{k} q~(\b v_F - \b v_T)\cdot \b n_{TF} \right).
  \end{equation}
  Summing \eqref{eq:consistency:bh:1} and \eqref{eq:consistency:bh:2} and observing that the first terms in parentheses cancel out by the definition \eqref{eq:proj} of $\proj{T}{k}$ since $\div \b v_T \in \Poly^{k-1}(T,\R) \subset \Poly^k(T,\R)$ for all $T \in\T_h$, we can write
  \[
  \begin{aligned}
    \mathcal E_{\cst{b},h}(q;\bu v_h) &= \sum_{T \in \T_h} \left( \cancel{\int_T (\proj{T}{k} q-q) (\div \b v_T)} +\sum_{F \in \F_T} \int_F (\proj{T}{k} q-q) (\b v_F - \b v_T)\cdot \b n_{TF} \right)   \\
    &\le \left(\sum_{T \in \T_h} h_T\| \proj{T}{k} q-q \|_{L^{\sob'}(\partial T,\R)}^{\sob'} \right)^\frac{1}{\sob'} \left(\sum_{T \in \T_h}\sum_{F \in \F_T} h_F^{1-\sob}\| \b v_F-\b v_T \|_{L^{\sob}(F,\R^d)}^{\sob} \right)^\frac{1}{\sob} \\
    &\lesssim h^{(k+1)(\sobs-1)} |q|_{W^{(k+1)(\sobs-1),\sob'}(\T_h,\R)} \| \bu v_h \|_{\sob,h},
  \end{aligned}
  \]
  where we have used the H\"{o}lder inequality along with $h_F\ge h_T$ whenever $F \in \F_T$ in the second line
  and the $((k+1)(\sobs-1),\sob')$-trace approximation property \eqref{eq:proj:app:F} of $\proj{T}{k}$ together with the bound $h_F \le h$ and the definition \eqref{eq:norm.epsilon.r} of the $\|{\cdot}\|_{\sob,h}$-norm to conclude.
  Passing to the supremum yields \eqref{eq:consistency:bh}.
\end{proof}

\subsubsection{Proof of Theorem \ref{thm:error.estimate}}

\begin{proof}[Proof of Theorem \ref{thm:error.estimate}]
  Let $(\bu e_h, \epsilon_h) \coloneqq (\bu u_h - \bu{\hat u}_h,p_h - \hat p_h) \in \bdU{h,0}{k} \times \dP{h}{k}$ where $\bu{\hat u}_h \coloneqq\bI{h}{k} \b u \in \bdU{h,0}{k}$ and $\hat p_h \coloneqq \proj{h}{k} p \in \dP{h}{k}$.
  \medskip\\
  \textbf{Step 1.} \emph{Consistency error.} Let $\mathcal E_h : \bdU{h,0}{k} \to \R$ be the consistency error linear form such that, for all $\bu v_h \in \bdU{h,0}{k}$,
  \begin{equation}\label{eq:Eh}
    \mathcal E_h(\bu v_h) \coloneqq \int_\Omega \b f \cdot \b v_h - \cst{a}_h(\bu{\hat u}_h,\bu v_h)-\cst{b}_h(\bu v_h,{\hat p}_h).
  \end{equation}
  Using in the above expression the fact that $\b f = -\DIV \stress(\cdot,\GRADs \b u)+\GRAD p$ almost everywhere in $\Omega$ to write $\mathcal E_h(\bu v_h)=\mathcal E_{\cst{a},h}(\b u;\bu v_h) + \mathcal E_{\cst{b},h}(p;\bu v_h)$, and invoking the consistency properties \eqref{eq:consistency:ah} of $\cst{a}_h$ and  \eqref{eq:consistency:bh} of $\cst{b}_h$, we obtain
  \begin{equation}\label{eq:error.estimate:step1:eh0}
    \$\coloneq
    \sup\limits_{\bu v_h \in \bdU{h,0}{k},\| \bu v_h \|_{\sob,h} = 1}\mathcal E_h(\bu v_h)
    \lesssim
    h^{(k+1)(\sobs-1)} \mathcal N_{\stress,\b u,p}.
  \end{equation}
  \\
  \textbf{Step 2.} \emph{Error estimate for the velocity.}  
  Using the strong monotonicity \eqref{eq:ah:strong.monotonicity} of $\cst{a}_h$, we get
  \begin{equation}\label{eq:error.estimate:step2:eh0}
    \begin{aligned}
      \| \bu e_h \|_{\sob,h}^{\sob+2-\sobs}
      &\lesssim \sigma_\cst{sm}^{-1}\left(
      \sigma_\cst{de}^\sob+\| \bu u_h \|_{\sob,h}^\sob+\|\bu{\hat u}_h\|_{\sob,h}^\sob
      \right)^\frac{2-\sobs}{\sob}\left(
      \cst{a}_h(\bu u_h,\bu e_h)-\cst{a}_h(\bu{\hat u}_h,\bu e_h)
      \right) \\
      &\lesssim \sigma_\cst{sm}^{-1}\mathcal N_{\b f}^{2-\sobs}\left(
      \cst{a}_h(\bu u_h,\bu e_h)-\cst{a}_h(\bu{\hat u}_h,\bu e_h)
      \right),
    \end{aligned}
  \end{equation}
  where we have used the a priori bound \eqref{eq:discrete.solution:bounds:uh} on the discrete solution along with the boundedness \eqref{eq:fortin:boundedness} of the global interpolator and the a priori bound \eqref{eq:continuous.solution:bounds:uh} on the continuous solution to conclude.
  Using then the discrete mass equation \eqref{eq:stokes.discrete:mass} along with \eqref{eq:fortin:consistency} (written for $\b v=\b u$) and the continuous mass equation \eqref{eq:stokes.weak:mass} to write $\cst{b}_h(\bI{h}{k} \b u,q_h) = b(\b u,q_h) = 0$, we get $\cst{b}_h(\bu e_h,q_h) = 0$ for all $q_h \in P^k_h$.
  Hence, combining this result with \eqref{eq:Eh} and the discrete momentum equation \eqref{eq:stokes.discrete:momentum} (with $\bu v_h = \bu e_h$), we obtain
  \begin{equation}\label{eq:error.estimate:step1:eh1}
    \cst{a}_h(\bu u_h,\bu e_h)-\cst{a}_h(\bu{\hat u}_h,\bu e_h)
    = \int_\Omega \b f \cdot \b e_h - \cst{a}_h(\bu{\hat u}_h,\bu e_h)-\cancel{\cst{b}_h(\bu e_h,p_h)}
    = \mathcal E_h(\bu e_h).
  \end{equation}
  Plugging \eqref{eq:error.estimate:step1:eh1} into \eqref{eq:error.estimate:step2:eh0}, we get
  \[
  \| \bu e_h \|_{\sob,h}^{\sob+2-\sobs}
  \le \sigma_\cst{sm}^{-1}\mathcal N_{\b f}^{2-\sobs}\$\| \bu e_h \|_{\sob,h}.
  \]
  Simplifying, using \eqref{eq:error.estimate:step1:eh0}, and taking the $(\sob+1-\sobs)$th root of the resulting inequality yields \eqref{eq:error.estimate:velocity}.
  \medskip\\
  \textbf{Step 3.} \emph{Error estimate for the pressure.}
  Using the H\"older continuity \eqref{eq:ah:holder.continuity} of $\cst{a}_h$, we have, for all $\bu v_h \in \bdU{h,0}{k}$,
  \begin{equation}\label{eq:error.estimate:step3:ah}
   \begin{aligned}
     \left|\cst{a}_h(\bu{\hat u}_h,\bu v_h)-\cst{a}_h(\bu u_h,\bu v_h)\right|
     &\lesssim \sigma_\cst{hc}\left( \sigma_\cst{de}^\sob+  \| \bu{\hat u}_h \|_{\sob,h}^\sob+\| \bu u_h \|_{\sob,h}^\sob\right)^\frac{\sob-\sobs}{\sob}\| \bu e_h \|_{\sob,h}^{\sobs-1}\| \bu v_h \|_{\sob,h} \\
  &\lesssim \sigma_\cst{hc}\mathcal N_{\b f}^{\sob-\sobs}\| \bu e_h \|_{\sob,h}^{\sobs-1}\| \bu v_h \|_{\sob,h},
  \end{aligned}
  \end{equation}
  where the first factor is estimated as in \eqref{eq:error.estimate:step2:eh0}.
   Thus, using the inf-sup condition \eqref{eq:bh:inf-sup}, we can write
  \begin{equation}\label{eq:error.estimate:step3:epsilonh}
    \begin{aligned}
      \| \epsilon_h \|_{L^{\sob'}(\Omega,\R)} &\lesssim \sup\limits_{\bu v_h \in \bdU{h,0}{k},\| \bu v_h \|_{\sob,h} = 1} \cst{b}_h(\bu v_h,\epsilon_h) \\
      &= \sup\limits_{\bu v_h \in \bdU{h,0}{k},\| \bu v_h \|_{\sob,h} = 1} \left(\mathcal E_h(\bu v_h)+\cst{a}_h(\bu{\hat u}_h,\bu v_h)-\cst{a}_h(\bu u_h,\bu v_h)\right)\\
      &\lesssim \$+\sigma_\cst{hc}\mathcal N_{\b f}^{\sob-\sobs}\| \bu e_h \|_{\sob,h}^{\sobs-1} \\
      &\lesssim  h^{(k+1)(\sobs-1)}\mathcal N_{\stress,\b u,p}
      +h^{(k+1)(\sobs-1)^2}\sigma_\cst{hc}\mathcal N_{\b f}^{|\sob-2|(\sobs-1)}    
      \left(\sigma_\cst{sm}^{-1}\mathcal N_{\stress,\b u,p}\right)^\frac{\sobs-1}{\sob+1-\sobs},
    \end{aligned}
  \end{equation}
  where we have used the definition \eqref{eq:Eh} of the consistency error together with equation \eqref{eq:stokes.discrete:momentum} to pass to the second line, \eqref{eq:error.estimate:step3:ah} to pass to the third line (recall that $\$$ denotes here the supremum in the left-hand side of \eqref{eq:error.estimate:step1:eh0}), and the bounds \eqref{eq:error.estimate:step1:eh0} and \eqref{eq:error.estimate:velocity} (proved in Step 2) to conclude.
\end{proof}

\appendix

\section{Power-framed functions}\label{sec:properties.stress}

In the following theorem, we introduce the notion of power-framed function and discuss sufficient conditions for this property to hold.

\begin{theorem}[Power-framed function]\label{thm:1d.power-framed}
  Let $U$ be a measurable subset of $\R^n$ with $n\ge1$, $(W,(\cdot,\cdot)_W)$ an inner product space, and $\b \sigma : U \times W \to W$. Assume that there exists a Carath\'eodory function $\varsigma : U \times \lbrack0,\infty) \to \R$ such that, for all $\b\tau \in W$ and almost every $\b x \in U$,
  \begin{subequations}\label{eq:1d.power-framed:stress}
    \begin{equation}
      \stress(\b x,\b\tau) = \varsigma(\b x,\|\b\tau\|_W)\b\tau,
    \end{equation}
    where $\|{\cdot}\|_W$ is the norm induced by $(\cdot,\cdot)_W$. 
    Additionally assume that, for almost every $\b x \in U$, $\varsigma(\b x,\cdot)$ is differentiable on $(0,\infty)$ and there exist $\varsigma_\cst{de} \in \lbrack0,\infty)$ and $\varsigma_\cst{sm},\varsigma_\cst{hc} \in (0,\infty)$ independent of $\b x$ such that, for all $\alpha \in (0,\infty)$,
    \begin{align} 
      \varsigma_\cst{sm} (\varsigma_\cst{de}^\sob+\alpha^\sob)^\frac{\sob-2}{\sob} \leq \frac{\partial(\alpha\varsigma(\b x,\alpha))}{\partial \alpha} \leq \varsigma_\cst{hc}(\varsigma_\cst{de}^\sob+\alpha^\sob)^\frac{\sob-2}{\sob}. \label{eq:1d.power-framed:eta}
    \end{align}
  \end{subequations}
  Then, $\stress$ is an \emph{$\sob$-power-framed function}, i.e., for all $(\b\tau,\b\eta) \in W^2$ with $\b\tau \neq \b\eta$ and almost every $\b x \in U$, the function $\stress$ verifies the H\"older continuity property
  \begin{subequations}\label{eq:od:power-framed:holder.continuity.strong.monotonicity}
    \begin{equation} 
      \|\stress(\b x,\b\tau)-\stress(\b x,\b\eta)\|_W \le \sigma_\cst{hc} \left(\sigma_\cst{de}^\sob+\|\b\tau\|_W^\sob+\|\b\eta\|_W^\sob\right)^\frac{\sob-2}{\sob}\| \b\tau-\b\eta \|_W,\label{eq:od:power-framed:holder.continuity} 
    \end{equation}
    and the strong monotonicity property
    \begin{equation}
      \left(\stress(\b x,\b\tau)-\stress(\b x,\b\eta),\b\tau-\b\eta\right)_W \ge \sigma_\cst{sm}\left(\sigma_\cst{de}^\sob+\|\b\tau\|_W^\sob+\|\b\eta\|_W^\sob\right)^\frac{\sob-2}{\sob}\|\b\tau-\b\eta\|_W^{2},\label{eq:od:power-framed:strong.monotonicity}
    \end{equation}
  \end{subequations}  
with $\sigma_\cst{de} \coloneqq \varsigma_\cst{de}$, 
$\sigma_\cst{hc} \coloneqq 2^{2-\sobs+\sob^{-1}\left\lceil\hspace{0.02cm}2-\sobs \right\rceil}(\sobs-1)^{-1} \varsigma_\cst{hc}$, and 
$\sigma_\cst{sm} \coloneqq 2^{\sobs-\sob-\left\lceil\sob^{-1}(\sob-\sobs)\right\rceil}(\sob+1-\sobs)^{-1} \varsigma_\cst{sm}$, where $\sobs$ is given by \eqref{eq:sing} and $\lceil{\cdot}\rceil$ is the ceiling function.
\end{theorem}

\begin{remark}[Notation]
  The boldface notation for the elements of $W$ is reminiscent of the fact that Theorem \ref{thm:1d.power-framed} is used with $W = \Ms{d}$ in Corollary \ref{cor:Carreau--Yasuda} to characterize the Carreau-Yasuda law as an $\sob$-power-framed function and in Lemma \ref{lem:sT} with $W = \R^d$ to study the local stabilization function $\cst{s}_T$.
\end{remark}

\begin{proof}[Proof of Theorem \ref{thm:1d.power-framed}]
  Let $\b x \in U$ be such that \eqref{eq:1d.power-framed:stress} holds, and $\b\tau,\b\eta \in W$. By symmetry of inequalities \eqref{eq:od:power-framed:holder.continuity.strong.monotonicity} and the fact that $\stress$ is continuous, we can assume, without loss of generality, that $\|\b\tau\|_W > \|\b\eta\|_W > 0$.
  \medskip\\
  (i) \emph{Strong monotonicity.}
  Let $\beta \in (0,\infty)$ and let $g : \lbrack\beta,\infty) \to \R$ be such that, for all $\alpha \in \lbrack\beta,\infty)$,
  \[
    g(\alpha) \coloneqq \alpha\varsigma(\b x,\alpha)-\beta\varsigma(\b x,\beta)-C_\cst{sm}(\varsigma_\cst{de}^\sob+\alpha^\sob+\beta^\sob)^\frac{\sob-2}{\sob}(\alpha-\beta),
  \;\text{ with }\ C_\cst{sm} \coloneqq \tfrac{2^{\sobs-\sob}}{\sob+1-\sobs}\varsigma_\cst{sm}.
  \]
  Differentiating $g$ and using the first inequality in \eqref{eq:1d.power-framed:eta}, we obtain, for all $\alpha \in \lbrack\beta,\infty)$, 
  \[
  \begin{aligned}
    \frac{\partial}{\partial \alpha}g(\alpha)
    &\geq \varsigma_\cst{sm}(\varsigma_\cst{de}^\sob+\alpha^\sob)^\frac{\sob-2}{\sob}-C_\cst{sm}\left((\sob-2)(\varsigma_\cst{de}^\sob+\alpha^\sob+\beta^\sob)^{-\frac{2}{\sob}}(\alpha-\beta)\alpha^{\sob-1}+(\varsigma_\cst{de}^\sob+\alpha^\sob+\beta^\sob)^\frac{\sob-2}{\sob}\right) \\
    &\geq \varsigma_\cst{sm}(\varsigma_\cst{de}^\sob+\alpha^\sob)^\frac{\sob-2}{\sob}-(\sob+1-\sobs)C_\cst{sm} (\varsigma_\cst{de}^\sob+\alpha^\sob+\beta^\sob)^\frac{\sob-2}{\sob} \\
    &\geq \varsigma_\cst{sm}2^{\sobs-\sob}(\varsigma_\cst{de}^\sob+\alpha^\sob+\beta^\sob)^\frac{\sob-2}{\sob}-(\sob+1-\sobs)C_\cst{sm} (\varsigma_\cst{de}^\sob+\alpha^\sob+\beta^\sob)^\frac{\sob-2}{\sob} = 0,
  \end{aligned}
  \]
  where, to pass to the second line, we have removed negative contributions if $\sob < 2$ and used the fact that $(\alpha-\beta)\alpha^{\sob-1} \le \varsigma_\cst{de}^\sob+\alpha^\sob+\beta^\sob$ if $\sob\ge 2$,
  to pass to the third line we have used the fact that $t \mapsto t^{\sob-2}$ is non-increasing if $\sob < 2$, and the fact that $\beta \le \alpha$ otherwise, while the conclusion follows from the definition of $C_\cst{sm}$.
  This shows that $g$ is non-decreasing. Hence, for all $\alpha\in\lbrack\beta,\infty)$, $g(\alpha)\geq g(\beta)=0$, i.e.
  \begin{equation}\label{eq:1d.power-framed:bound:1}
    \alpha\varsigma(\b x,\alpha)-\beta\varsigma(\b x,\beta) \geq C_\cst{sm}(\varsigma_\cst{de}^\sob+\alpha^\sob+\beta^\sob)^\frac{\sob-2}{\sob}(\alpha-\beta).
  \end{equation}
  Moreover, for all $\alpha,\beta \in (0,\infty)$, using \eqref{eq:1d.power-framed:bound:1} (with $\beta = 0$) along with the fact that $t \mapsto t^{\sob-2}$ is decreasing if $\sob < 2$ and inequality \eqref{eq:sum-power} if $\sob\ge 2$, we infer that
  \begin{equation}\label{eq:1d.power-framed:bound:2}
    \begin{aligned}
      \varsigma(\b x,\alpha)+\varsigma(\b x,\beta) &\ge C_\cst{sm}\left((\varsigma_\cst{de}^\sob+\alpha^\sob)^\frac{\sob-2}{\sob}+(\varsigma_\cst{de}^\sob+\beta^\sob)^\frac{\sob-2}{\sob}\right) \ge C_\cst{sm}2^{1-\left\lceil\frac{\sob-\sobs}{\sob}\right\rceil}(\varsigma_\cst{de}^\sob+\alpha^\sob+\beta^\sob)^\frac{\sob-2}{\sob}.
    \end{aligned}
  \end{equation}
  We conclude that $\stress$ verifies \eqref{eq:od:power-framed:strong.monotonicity} by using \eqref{eq:1d.power-framed:bound:1} and \eqref{eq:1d.power-framed:bound:2} with $\alpha = \|\b\tau\|_W$ and $\beta = \|\b\eta\|_W$ as follows:
  \[
  \begin{aligned}
    &(\stress(\b x,\b\tau)-\stress(\b x,\b\eta),\b\tau-\b\eta)_W
    \\
    &\quad= (\b\tau\varsigma(\b x,\|\b\tau\|_W)-\b\eta\varsigma(\b x,\|\b\eta\|_W),\b\tau-\b\eta)_W \\
    &\quad = \|\b\tau\|_W^2\varsigma(\b x,\|\b\tau\|_W)+\|\b\eta\|_W^2\varsigma(\b x,\|\b\eta\|_W) 
    -(\b\tau,\b\eta)_W \left[\varsigma(\b x,\|\b\tau\|_W) + \varsigma(\b x,\|\b\eta\|_W)\right] \\
    &\quad = \left[
      \|\b\tau\|_W\varsigma(\b x,\|\b\tau\|_W)-\|\b\eta\|_W\varsigma(\b x,\|\b\eta\|_W)
      \right](\|\b\tau\|_W-\|\b\eta\|_W)
    \\
    &\qquad
    + \left[
      \varsigma(\b x,\|\b\tau\|_W)+\varsigma(\b x,\|\b\eta\|_W)
      \right](\|\b\tau\|_W\|\b\eta\|_W-(\b\tau,\b\eta)_W)
    \\
    &\quad\geq  C_\cst{sm}2^{-\left\lceil\frac{\sob-\sobs}{\sob}\right\rceil}\left(
    \varsigma_\cst{de}^\sob+\|\b\tau\|_W^\sob+\|\b\eta\|_W^\sob
    \right)^\frac{\sob-2}{\sob}\left[
      (\|\b\tau\|_W-\|\b\eta\|_W)^2+2(\|\b\tau\|_W\|\b\eta\|_W-(\b\tau,\b\eta)_W)
      \right] \\
    &\quad= C_\cst{sm}2^{-\left\lceil\frac{\sob-\sobs}{\sob}\right\rceil}\left(
    \varsigma_\cst{de}^\sob+\|\b\tau\|_W^\sob+\|\b\eta\|_W^\sob
    \right)^\frac{\sob-2}{\sob}\|\b\tau-\b\eta\|_W^2.
  \end{aligned}
  \]
  \medskip\\
  (ii) \emph{H\"older continuity.}
  Now, setting $C_\cst{hc} \coloneqq \frac{\varsigma_\cst{hc}}{\sobs-1}$ and reasoning in a similar way as for the proof of \eqref{eq:1d.power-framed:bound:1} to leverage the second inequality in \eqref{eq:1d.power-framed:eta}, we have, for all $\alpha \in \lbrack\beta,\infty)$,
  \begin{equation}\label{eq:1d.power-framed:bound:3}
    \alpha\varsigma(\b x,\alpha)-\beta\varsigma(\b x,\beta) \le C_\cst{hc}\left(
    \varsigma_\cst{de}^\sob+\alpha^\sob+\beta^\sob
    \right)^\frac{\sob-2}{\sob}(\alpha-\beta).
  \end{equation}
  First, let $\sob \ge 2$. Using \eqref{eq:1d.power-framed:bound:3} (with $\beta=0$) and  the fact that $t \mapsto t^{\sob-2}$ is non-decreasing, we have, for all $\alpha,\beta \in (0,\infty)$,
  \begin{equation}\label{eq:1d.power-framed:bound:4}
      \varsigma(\b x,\alpha)\varsigma(\b x,\beta)
      \le C_\cst{hc}^2\left(
      \varsigma_\cst{de}^\sob+\alpha^\sob
      \right)^\frac{\sob-2}{\sob}\left(
      \varsigma_\cst{de}^\sob+\beta^\sob
      \right)^\frac{\sob-2}{\sob} \le \left[
        C_\cst{hc}\left(\varsigma_\cst{de}^\sob+\alpha^\sob+\beta^\sob\right)^\frac{\sob-2}{\sob}
        \right]^2.
  \end{equation}
  Thus, using inequalities \eqref{eq:1d.power-framed:bound:3} and \eqref{eq:1d.power-framed:bound:4} with $\alpha = \|\b\tau\|_W$ and $\beta = \|\b\eta\|_W$, we infer
  \begin{equation}\label{eq:1d.power-framed:bound:5}
    \begin{aligned}
      &\|\stress(\b x,\b\tau)-\stress(\b x,\b\eta)\|_W^2
      \\
      &\quad = \left(
      \b\tau\varsigma(\b x,\|\b\tau\|_W)-\b\eta\varsigma(\b x,\|\b\eta\|_W),
      \b\tau\varsigma(\b x,\|\b\tau\|_W)-\b\eta\varsigma(\b x,\|\b\eta\|_W)
      \right)_W
      \\
      &\quad = \left[\|\b\tau\|_W\varsigma(\b x,\|\b\tau\|_W)-\|\b\eta\|_W\varsigma(\b x,\|\b\eta\|_W)\right]^2
      \\
      &\quad \qquad + 2\varsigma(\b x,\|\b\tau\|_W)\varsigma(\b x,\|\b\eta\|_W)\left[
        \|\b\tau\|_W\|\b\eta\|_W-(\b\tau,\b\eta)_W
        \right]
      \\
      &\quad \le  \left[
        C_\cst{hc}\left(\varsigma_\cst{de}^\sob+\|\b\tau\|_W^\sob+\|\b\eta\|_W^\sob\right)^\frac{\sob-2}{\sob}
        \right]^2\left[
        (\|\b\tau\|_W-\|\b\eta\|_W)^2+2(\|\b\tau\|_W\|\b\eta\|_W-(\b\tau,\b\eta)_W)
        \right] \\
      &\quad = \left[
        C_\cst{hc}\left(\varsigma_\cst{de}^\sob+\|\b\tau\|_W^\sob+\|\b\eta\|_W^\sob\right)^\frac{\sob-2}{\sob}\|\b\tau-\b\eta\|_W
        \right]^2,
    \end{aligned}
  \end{equation}
  hence $\stress$ verifies \eqref{eq:od:power-framed:holder.continuity} for $\sob \ge 2$.
Assume now $\sob < 2$. Using a triangle inequality followed by \eqref{eq:1d.power-framed:bound:3} and the left inequality in \eqref{eq:sum-power}, it is inferred that
  \[
  \begin{aligned}
    \|\stress(\b x,\b\tau)-\stress(\b x,\b\eta)\|_W 
    &\leq \varsigma(\b x,\|\b\tau\|_W)\|\b\tau\|_W+\varsigma(\b x,\|\b\eta\|_W)\|\b\eta\|_W 
    \\
    & \leq C_\cst{hc}\left((\varsigma_\cst{de}^\sob+\|\b\tau\|_W^\sob)^\frac{\sob-1}{\sob}+(\varsigma_\cst{de}^\sob+\|\b\eta\|_W^\sob)^\frac{\sob-1}{\sob}\right) 
    \\
    &\leq 2^\frac{1}{\sob}C_\cst{hc}(2\varsigma_\cst{de}^\sob+\|\b\tau\|_W^\sob+\|\b\eta\|_W^\sob)^\frac{\sob-1}{\sob} 
    \\
    & =  2^\frac{1}{\sob} C_\cst{hc}(2\varsigma_\cst{de}^\sob+\|\b\tau\|_W^\sob+\|\b\eta\|_W^\sob)^\frac{\sob-2}{\sob}
    (2\varsigma_\cst{de}^\sob+\|\b\tau\|_W^\sob+\|\b\eta\|_W^\sob)^\frac1\sob,
    \\
    & \leq  2^\frac{1}{\sob} C_\cst{hc}(\varsigma_\cst{de}^\sob+\|\b\tau\|_W^\sob+\|\b\eta\|_W^\sob)^\frac{\sob-2}{\sob}
    (2\varsigma_\cst{de}+\|\b\tau\|_W+\|\b\eta\|_W),
  \end{aligned}
  \]
  where the last line follows from the fact that $t \mapsto t^{\sob-2}$ is decreasing and again \eqref{eq:sum-power}. 
  If $2\varsigma_\cst{de}+\|\b\tau\|_W+\|\b\eta\|_W\le 2^{2-r}\|\b\tau-\b\eta\|_W$, from the previous bound we directly get the conclusion, i.e. \eqref{eq:od:power-framed:holder.continuity} with $\sigma_\cst{hc}=2^{2-r+\frac{1}{r}}C_\cst{hc}$.
  Otherwise, using \eqref{eq:sum-power} and a triangle inequality yields
  \begin{equation}\label{eq:est_else}
  \begin{aligned}
    (\varsigma_\cst{de}^\sob+\|\b\tau\|_W^\sob)^\frac{1}{\sob}(\varsigma_\cst{de}^\sob+\|\b\eta\|_W^\sob)^\frac{1}{\sob} 
    &\ge 2^{-\frac{2}{r'}}(\varsigma_\cst{de}+\|\b\tau\|_W)(\varsigma_\cst{de}+\|\b\eta\|_W) \\
    &= 2^{-2(\frac{1}{r'}+1)}\left[
      \left(2\varsigma_\cst{de}+\|\b\tau\|_W+\|\b\eta\|_W\right)^2-\left(\|\b\tau\|_W-\|\b\eta\|_W\right)^2
      \right]
    \\
    &\ge 2^{-2(\frac{1}{r'}+1)}\left[
      \left(2\varsigma_\cst{de}+\|\b\tau\|_W+\|\b\eta\|_W\right)^2-\|\b\tau-\b\eta\|_W^2
      \right]
    \\
    &\ge 2^{-2(\frac{1}{r'}+1)}(1-4^{r-2})\left(2\varsigma_\cst{de}+\|\b\tau\|_W+\|\b\eta\|_W\right)^2\\
    &\ge 2^{\frac{2}{(r-2)r}-2}\left(\varsigma_\cst{de}^\sob+\|\b\tau\|_W^\sob+\|\b\eta\|_W^\sob\right)^\frac{2}{\sob},
    \end{aligned}
  \end{equation}
where we concluded with \eqref{eq:sum-power} together with the fact that $2^{-2(\frac{1}{r'}+1)}\left(1-4^{r-2}\right) \ge 2^{\frac{2}{(r-2)r}-2}$.
Finally, raising both sides of \eqref{eq:est_else} to the power $\sob-2$, we get a relation analogous to \eqref{eq:1d.power-framed:bound:4}. Hence, proceeding as in \eqref{eq:1d.power-framed:bound:5}, we infer \eqref{eq:od:power-framed:holder.continuity}.
\end{proof}

\begin{corollary}[Carreau--Yasuda]\label{cor:Carreau--Yasuda}
  The strain rate-shear stress law of the $(\mu,\delta,a,\sob)$-Carreau--Yasuda fluid defined in Example \ref{ex:Carreau--Yasuda} is an $\sob$-power-framed function.
\end{corollary}

\begin{proof}
  Let $\b x \in \Omega$ and $g : (0,\infty) \to \R$ be such that, for all $\alpha\in(0,\infty)$,
  \[
  g(\alpha)
  \coloneqq \frac{\partial}{\partial \alpha}\left[
  \alpha\mu(\b x)\left(\delta^{a(\b x)}+\alpha^{a(\b x)}
  \right)^\frac{\sob-2}{a(\b x)}
  \right]
  = \mu(\b x)\left(
  \delta^{a(\b x)}+\alpha^{a(\b x)}
  \right)^{\frac{\sob-2}{a(\b x)}-1}\left(
  \delta^{a(\b x)}+(\sob-1)\alpha^{a(\b x)}
  \right).
  \] 
  We have for all $\alpha \in (0,\infty)$,
  \[
  \mu_- (\sobs-1) \left(
  \delta^{a(\b x)}+\alpha^{a(\b x)}
  \right)^{\frac{\sob-2}{a(\b x)}} \le g(\alpha)
  \leq \mu_+ (\sob+1-\sobs)\left(
  \delta^{a(\b x)}+\alpha^{a(\b x)}
  \right)^{\frac{\sob-2}{a(\b x)}},
  \]
  and we conclude using \eqref{eq:sum-power} together with Theorem \ref{thm:1d.power-framed}.
\end{proof}

\section*{Acknowledgments}
The work of M. Botti was funded by the European Commission through the H2020-MSCA-IF-EF project PDGeoFF (Grant no. 896616). This support is gratefully acknowledged.

%------------------------------------------------------------------------------------%

\raggedright
\printbibliography

\end{document}